\documentclass[11pt]{amsart}
\usepackage{amssymb}

\usepackage{amsmath,amssymb}
\usepackage{epsfig, psfrag, graphics}

\usepackage{amscd}

\newtheorem{theorem}{Theorem}[section]
\newtheorem{proposition}[theorem]{Proposition}
\newtheorem{lemma}[theorem]{Lemma}
\newtheorem{corollary}[theorem]{Corollary}
\newtheorem{remark}[theorem]{Remark}

\newtheorem{definition}[theorem]{Definition}

\newtheorem{example}[theorem]{Example}
\newtheorem{prop}[theorem]{Proposition}

\theoremstyle{definition}

\newcommand{\edpf} { {$\square$ \par } }

\newcommand\Z{\mathbb{Z}}
\newcommand\R{\mathbb{R}}

\newcommand\C{\mathbb{C}}
\newcommand\N{\mathbb{N}}

\newcommand\g{\frak{g}}
\newcommand\vu{\frak{v}}
\newcommand\n{\frak{n}}
\newcommand{\eps} {\varepsilon}

\begin{document} \title[Asymptotic shape of balls in groups with polynomial growth]{Geometry of locally compact groups of polynomial
growth and shape of large balls.}
\author{Emmanuel Breuillard}
\email{emmanuel.breuillard@math.u-psud.fr}
\address{Universit\'e Paris-Sud 11, Laboratoire de Math\'ematiques, 91405 Orsay, France}
\date{April 2012}

\begin{abstract} We show that any locally compact group $G$ with
polynomial growth is weakly commensurable to some simply connected
solvable Lie group $S$, the Lie shadow of $G$. We then study the shape of
large balls and show, generalizing work of P. Pansu, that after a suitable
renormalization, they converge to a limiting compact set, which is isometric to the unit ball for a left-invariant subFinsler metric on the so-called graded nilshadow of $S$. As by-products, we obtain asymptotics for the volume
of large balls, we prove that balls are Folner and hence that the ergodic theorem holds for all ball averages. Along the way we also answer negatively a question of
Burago and Margulis \cite{Bur2} on asymptotic word metrics and recover some results of Stoll \cite{Sto2} of the rationality of growth series of Heisenberg groups.
\end{abstract}

\maketitle

\setcounter{tocdepth}{1}
\tableofcontents

\section{Introduction}

\subsection{Groups with polynomial growth}

Let $G$ be a locally compact group with left Haar measure $vol_{G}.$ We will
assume that $G$ is generated by a compact symmetric subset $\Omega .$
Classically, $G$ is said to have \textit{polynomial growth }if there exist $%
C>0$ and $k>0$ such that for any integer $n\geq 1$%
\begin{equation*}
vol_{G}(\Omega ^{n})\leq C\cdot n^{k},
\end{equation*}
where $\Omega^n=\Omega_\cdot\ldots \cdot \Omega$ is the $n$-fold product set. Another choice for $\Omega $ would only change the constant $C$, but not the
polynomial nature of the bound. One of the consequences of the analysis
carried out in this paper is the following theorem:

\begin{theorem}[Volume asymptotics]
\label{firsthm}Let $G$ be a locally compact group with polynomial growth and $%
\Omega $ a compact symmetric generating subset of $G.$ Then there exists $%
c(\Omega )>0$ and an integer $d(G)\geq 0$ depending on $G$ only such that
the following holds:
\begin{equation*}
\lim_{n\rightarrow +\infty }\frac{vol_{G}(\Omega ^{n})}{n^{d(G)}}=c(\Omega )
\end{equation*}
\end{theorem}

This extends the main result of Pansu \cite{Pan}. The integer $d(G)$ coincides with the exponent of
growth of a naturally associated graded nilpotent Lie group, the asymptotic cone of $G$, and is
given by the Bass-Guivarc'h formula $(\ref{BG})$ below. The constant $c(\Omega )$ will be
interpreted as the volume of the unit ball of a sub-Riemannian Finsler metric on this nilpotent Lie
group. Theorem \ref {firsthm} is a by-product of our study of the asymptotic behavior of
\textit{periodic pseudodistances} on $G$, that is pseudodistances that are invariant under a
co-compact subgroup of $G$ and satisfy a weak kind of the existence of geodesics axiom (see
Definition \ref{periodic}).

Our first task is to get a better understanding of the structure of locally
compact groups of polynomial growth. Guivarc'h \cite{Gui} proved that locally compact groups of polynomial growth are amenable and unimodular and that every compactly generated\footnote{in fact it follows from the Gromov-Losert structure theory that every closed subgroup is compactly generated.} closed subgroup also has polynomial growth.

Guivarc'h \cite{Gui} and Jenkins \cite{Jen} also characterized connected Lie groups with polynomial growth: a connected Lie group has polynomial growth if
and only if it is of type $(R),$ that is if for all $x\in Lie(S)$, $ad(x)$
has only purely imaginary eigenvalues. Such groups are solvable-by-compact
and any connected nilpotent Lie group is of type $(R)$. 

It is much more difficult to characterize discrete groups with polynomial growth, and this was done in a celebrated paper of Gromov \cite{Gro}, proving that they are virtually nilpotent. Losert \cite{Los} generalized Gromov's method of proof and showed that it applied with little modification to arbitrary locally compact groups with polynomial growth. In particular he showed that they contain a normal compact subgroup modulo which the quotient is a (not necessarily connected) Lie group. We will prove the following refinement.

\begin{theorem}[Lie shadow]
\label{weaklycom}\label{reduction}\label{general}Let $G$ be a locally
compact group of polynomial growth. Then there exists a connected and simply connected
solvable Lie group $S$ of type $(R),$ which is weakly commensurable to $G.$
We call such a Lie group a \textit{Lie shadow} of $G.$
\end{theorem}

Two locally compact groups are said to be weakly commensurable if, up to
moding out by a compact kernel, they have a common closed co-compact
subgroup. More precisely, we will show that, for some normal compact
subgroup $K$, $G/K$ has a co-compact subgroup $H/K$ which can be embedded as
a closed and co-compact subgroup of a connected and simply connected solvable Lie group $S$
of type $(R).$ 

We must be aware that being weakly commensurable is not an
equivalence relation among locally compact groups (unlike among finitely
generated groups). Additionally, the Lie shadow $S$ is not
unique up to isomorphism (e.g. $\Bbb{Z}^{3}$ is a co-compact lattice in both
$\Bbb{R}^{3}$ and the universal cover of the group of motions of the plane).

We cannot replace the word solvable by the word nilpotent in the above theorem. We refer the reader to Example \ref{exx} for an example of a connected solvable Lie group of type $(R)$ without compact normal subgroups, which admits no co-compact nilpotent subgroup. In fact this is typical for Lie groups of type $(R)$. So in the general locally compact case (or just the Lie case) groups of polynomial growth can be genuinely not nilpotent, unlike what happens in the discrete case. There are important differences between the discrete case and the general case. For example, we will show that no rate of convergence can be expected in Theorem \ref{firsthm} when $G$ is solvable not nilpotent, while some polynomial rate always holds in the nilpotent discrete case \cite{breuillard-ledonne}.

Theorem \ref{weaklycom} will enable us to reduce most geometric questions about locally compact groups of polynomial growth, and in particular the proof of Theorem \ref{firsthm}, to the connected Lie
group case. Observe also that Theorem \ref{weaklycom} subsumes Gromov's
theorem on polynomial growth, because it is not hard to see that a co-compact lattice in a solvable Lie group of
polynomial growth must be virtually nilpotent (see Remark \ref{polattice}). Of
course in the proof we make use of Gromov's
theorem, in its generalized form for locally compact groups due to Losert.
The rest of the proof combines ideas of Y. Guivarc'h, D. Mostow and a crucial embedding theorem of H.C.
Wang. It is given in Paragraph \ref{pfgen} and is largely independent of the rest of the paper.

\subsection{Asymptotic shapes}

The main part of the paper is devoted to the asymptotic behavior of \emph{periodic pseudodistances} on $G$. We refer the reader to Definition \ref{periodic} for the precise definition of this term, suffices it to say now that it is a class of pseudodistances which contains both left-invariant word metrics on $G$ and geodesic metrics on $G$ that are left-invariant under co-compact subgroup of $G$.

Theorem \ref{weaklycom} enables us to assume that $G$ is a co-compact subgroup of a simply connected solvable Lie
group $S,$ and rather than looking at pseudodistances on $G$, we will look
at pseudodistances on $S$ that are left-invariant under a co-compact
subgroup $H$. More precisely a direct consequence of Theorem \ref{general} is the following:

\begin{proposition}\label{bdeddistance} Let $G$ be a locally compact group with polynomial growth and $\rho$ a periodic metric on $G$. Then $(G,\rho)$ is $(1,C)$-quasi-isometric to $(S,\rho_S)$ for some finite $C>0$, where $S$ is a connected and simply connected solvable Lie group of type $(R)$ and $\rho_S$ some periodic metric on $S$.
\end{proposition}

Recall that two metric spaces $(X,d_X)$ and $(Y,d_Y)$ are called $(1,C)$-quasi-isometric if there exists a map $\phi:X \to Y$ such that any $y\in Y$ is at distance at most $C$ from some element in the image of $\phi$ and if $|d_Y(\phi(x),\phi(x')) - d_X(x,x')| \leq C$ for all $x,x' \in X$. 

In the case when $S$ is $\Bbb{R}^{d}$ and $H$ is $\Bbb{Z}^{d}$, it is a simple
exercise to show that any periodic pseudodistance is asymptotic to a norm on
$\Bbb{R}^{d},$ i.e. $\rho (e,x)/\left\| x\right\| \rightarrow 1$ as $%
x\rightarrow \infty $, where $\left\| x\right\| =\lim \frac{1}{n}\rho (e,nx)$ is
a well defined norm on $\Bbb{R}^{d}$. Burago in \cite{Bur} showed a much finer
result, namely that if $\rho $ is coarsely geodesic, then $\rho (e,x)-\left\|
x\right\| $ is bounded when $x$ ranges over $\Bbb{R}^{d}.$ When $S$ is a
nilpotent Lie group and $H$ a lattice in $S,$ then Pansu proved in his thesis
\cite{Pan}, that a similar result holds, namely that $ \rho (e,x)/\left|
x \right| \rightarrow 1$ for some (unique only after a choice of a one-parameter
group of dilations) homogeneous quasi-norm $\left| x\right| $ on the nilpotent
Lie group. However, we show in Section \ref {speed}, that it is not true in
general that $\rho (e,x)-\left| x\right| $ stays bounded, even for finitely
generated nilpotent groups, thus answering a question of Burago (see also Gromov
\cite{Gro3}). Our main purpose here will be to extend Pansu's result to
solvable Lie groups of polynomial growth.\\

As was first noticed by Guivarc'h in his thesis \cite{Gui}, when dealing
with geometric properties of solvable Lie groups, it is useful to consider
the so-called nilshadow of the group, a construction first introduced by
Auslander and Green in \cite{AG}. According to this construction, it is
possible to modify the Lie product on $S$ in a natural way, by so to speak
removing the semisimple part of the action on the nilradical, in order to
turn $S$ into a nilpotent Lie group, its nilshadow $S_{N}$. The two Lie
groups have the same underlying manifold, which is diffeomorphic to $\Bbb{R}%
^{n},$ only a different Lie product. They also share the same Haar measure.
This ``semisimple part'' is a commutative relatively compact subgroup $T(S)$
of automorphisms of $S$, image of $S$ under a homomorphism $T:S\rightarrow
Aut(S)$. The new product $g*h$ is defined as follows by twisting the old one $g \cdot h$ by means of $%
T(S)$,
\begin{equation}\label{nilproduct}
g*h:=g\cdot T(g^{-1})h
\end{equation}
The two groups $S$ and $S_{N}$ are easily seen to be
quasi-isometric, and this is why any locally compact group of polynomial
growth $G$ is quasi-isometric to some nilpotent Lie group. In particular, their asymptotic cones are bi-Lipschitz. The asymptotic cone of a nilpotent Lie group is a certain associated graded nilpotent Lie group endowed with a left invariant geodesic distance (or Carnot group). The graded group associated to $S_N$ will be called the \emph{graded nilshadow} of $S$. Section \ref{nilshadow} will be devoted to the construction and basic properties of the nilshadow and its graded group.

In this paper, we are dealing with a finer relation than quasi-isometry. We will be interested in when do two left invariant (or periodic) distances are asymptotic\footnote{Yet a finer equivalence relation is $(1,C)$-quasi-isometry, i.e. being at bounded distance in Gromov-Hausdorff metric; classifying periodic metrics up to this kind of equivalence is much harder.} (in the sense that $\frac{d_1(e,g)}{d_2(e,g)} \to 1$ when $g \to \infty$). In particular, for every locally compact group $G$ with polynomial growth, we will identify its asymptotic cone up to isometry and not only up to quasi-isometry or bi-Lipschitz equivalence (see Corollary \ref{asycones} below). One of our main results is the following:

\begin{theorem}[Main theorem]\label{main-theorem}\label{MetComp} Let $S$ be a simply connected solvable Lie group with polynomial growth. Let $\rho(x,y)$ be periodic pseudodistance on $S$ which is invariant under a co-compact subgroup $H$ of $S$ (see Def. \ref{periodic}). On the manifold $S$, one can put a new Lie group structure, which turns $S$ into a stratified nilpotent Lie group, the graded nilshadow of $S$, and a subFinsler metric $d_\infty(x,y)$ on $S$ which is left-invariant for this new group structure such that
$$\frac{\rho(e,g)}{d_\infty(e,g)} \to 1$$
as $g \to \infty$ in $S$. Moreover every automorphism in $T(H)$ is an isometry of $d_\infty$.
\end{theorem}

The reader who wishes to see a simple illustration of this theorem can go directly to subsection \ref{cones}, where we have treated in detail a specific example of periodic metric on the universal cover of the groups of motions of the plane.

The new stratified nilpotent Lie group structure on $S$ given by the graded nilshadow comes with a one-parameter family of so-called \emph{homogeneous dilations} $\{\delta_t\}_{t>0}$. It also comes with an extra group of automorphisms, namely the image of $H$ under the homomorphism $T$. This yields automorphisms of $S$ for both the original group structure on $S$ and the new graded nilshadow group structure. Moreover the dilations $\{\delta_t\}_{t>0}$ are automorphisms of the graded nilshadow and they commute with $T(H)$.

A subFinsler metric is a geodesic distance which is defined exactly as subRiemannian (or Carnot-Caratheodory) metrics on Carnot groups are defined (see  e.g. \cite{Monty}), except that the norm used to compute the length of horizontal paths is not necessarily a Euclidean norm. We refer the reader to Section \ref{CCmetrics} for a precise definition.

In Theorem \ref{main-theorem} the subFinsler metric $d_\infty$ is left invariant for the new Lie structure on $S$ and it is also invariant under all automorphisms in $T(H)$ (these form a relatively compact commutative group of automorphisms). Moreover it satisfies the following pleasing scaling law: $$d_\infty(\delta_t(x),\delta_t(y))=td_\infty(x,y) \textnormal{ } \forall t>0.$$\\

The proof of Theorem \ref{main-theorem} splits in two important steps. The first is a reduction to the nilpotent case and is performed in Section \ref{nilpotent-reduction}. Using a double averaging of the pseudodistance $\rho$ over both $K:=\overline{T(H)}$ and $S/H$, we construct an associated pseudodistance, which is periodic for the nilshadow structure on $S$ (i.e. left-invariant by a co-compact subgroup for this structure), and we prove that it is asymptotic to the original $\rho$. This reduces the problem to nilpotent Lie groups. The key to this reduction is the following crucial observation: that unipotent automorphisms of $S$ induce only a sublinear distortion, forcing the metric $\rho$ to be asymptotically invariant under $T(H)$. The second step of the proof assumes that $S$ is nilpotent. This part is dealt with in Section \ref{nilpotent-case} and is essentially a reformulation of the arguments used by Pansu in \cite{Pan}.
\\

Incidently, we stress the fact that the generality in which Section \ref{nilpotent-case} is treated (i.e. for general coarsely geodesic, and even asymptotically geodesic periodic metrics) is necessary to prove even the most basic case (i.e. word metrics) of Theorem \ref{main-theorem} for non-nilpotent solvable groups. So even if we were only interested in the asymptotics of left invariant word metrics on a solvable Lie group of polynomial growth $S$, we would still need to understand the asymptotics of arbitrary coarsely geodesic left invariant distances (and not only word metrics!) on nilpotent Lie groups. This is because the new pseudodistance obtained by averaging, see $(\ref{rhok})$, is no longer a word metric.\\

The subFinsler metric $d_{\infty }(e,x)$ in the above theorem is induced by a certain $T(H)$-invariant norm on the first stratum $m_1$ of the graded nilshadow (which is $T(H)$-invariant complementary subspace of the commutator subalgebra of the nilshadow). This norm can be described rather explicitly as follows.

Recall that we have\footnote{The subspace $m_1$ can be identified with the abelianized nilshadow (or abelianized graded nilshadow) by first identifying the nilshadow with its Lie algebra via the exponential map and then projecting modulo the commutator subalgebra. The map does not depend on the choice involved in the construction of the nilshadow. See also Remark \ref{indept}.} a canonical map $\pi_1: S \to m_1$, which is a group homomorphism for both the nilshadow and graded nilshadow structures. Then:
\begin{equation*}
\{v\in m_{1},\left\| v\right\| _{\infty}\leq 1\}=\bigcap_{ F \subset S}\overline{CvxHull}\left\{ \frac{%
\pi _{1}(h)}{\rho(e,h)},h\in H\backslash F \right\},
\end{equation*}
where the right hand side is the intersection over all compact subsets $F$ of $S$ of the closed convex hull of the points $\pi_1(h)/\rho(e,h)$ for $h \in H \backslash F$.\\

Figure 1 gives an illustration of the limit shape corresponding to the word metric on the $3$-dimensional discrete Heisenberg group with standard generators. We explain in the Appendix how one can compute explicitly the geodesics of the limit metric and the limit shape in this example.

When $S$ itself is nilpotent to begin with and $\rho $ is (in restriction to $H$) the word metric associated
to a symmetric compact generating set $\Omega $ of $H$ (namely $\rho_\Omega(e,h):=\inf\{n \in \N; h \in \Omega^n\}$), the above norm takes the
following simple form:
\begin{equation}
\{v\in m_{1},\left\| v\right\| _{\infty}\leq 1\}=CvxHull\left\{ \pi
_{1}(\omega ),\omega \in \Omega \right\}   \label{WordMetricBall}
\end{equation}
For instance, in the special case when $H$ is a torsion-free finitely generated nilpotent
group with generating set $\Omega $ and $S$ is its Malcev closure, the unit ball $\{v\in
m_{1},\left\| v\right\| _{\infty}\leq 1\}$ is a polyhedron in $m_{1}.$ This was Pansu's description in \cite{Pan}.

However when $S$ is not nilpotent, and is equipped with a word metric $\rho_\Omega$ on a co-compact subgroup, then the determination of the limit shape, i.e. the determination of the limit norm $\|\cdot\|_\infty$ on the abelianized nilshadow, is much more difficult. Clearly $\|\cdot\|_\infty$ is $K$-invariant and it is a simple observation that the unit ball for $\|\cdot\|_\infty$ is always contained in the convex hull of the $K$-orbit of $\pi_1(\Omega)$. Nevertheless the unit ball is typically smaller than that (unless $\Omega$ was $K$-invariant to begin with).

In general it would be interesting to determine whether there exists a simple description of the limit shape of an arbitrary word metric on a solvable Lie group with polynomial growth. We refer the reader to Section \ref{speed} and Paragraph \ref{solshape} for an example of a class of word metrics on the universal cover of the group of motions of the plane, for which we were able to compute the limit shape.





Another by-product of Theorem \ref{main-theorem} is the following result.

\begin{corollary}[Asymptotic shape] \label{shape}Let $S$ be a simply connected solvable Lie
group with polynomial growth and $H$ a co-compact subgroup. Let $\rho$ be an $H$-periodic pseudodistance on $S$.
Then in the Hausdorff metric,
\begin{equation*}
\lim_{t\rightarrow +\infty }\delta _{\frac{1}{t}}(B_{\rho }(t))=\mathcal{C},
\end{equation*}
where $\mathcal{C}$ is a $T(H)$-invariant compact neighborhood of the identity in $S$, $B_{\rho }(t)$ is the $\rho $-ball of radius $t$ in $S$ and $\{\delta _{t}\}_{t>0}$ is a one-parameter group of dilations on $S$ (equipped with the graded nilshadow structure). Moreover, $\mathcal{C}=\left\{ g\in S,d_{\infty }(e,g)\leq 1\right\} $ is
the unit ball of the limit subFinsler metric from Theorem \ref{main-theorem}.
\end{corollary}

\proof%
By Theorem \ref{main-theorem}, for every $\varepsilon >0$ we have $B_{d_{\infty
}}(t-\varepsilon t)\subset B_{\rho }(t)\subset B_{d_{\infty }}(t+\varepsilon
t)$ if $t$ is large enough. Since $\delta _{\frac{1}{t}}(B_{d_{\infty }}(t))=%
\mathcal{C}$, for all $t>0,$ we are done.%
\endproof

\begin{figure}\label{fig1}
\begin{center}
\includegraphics[scale=.5]{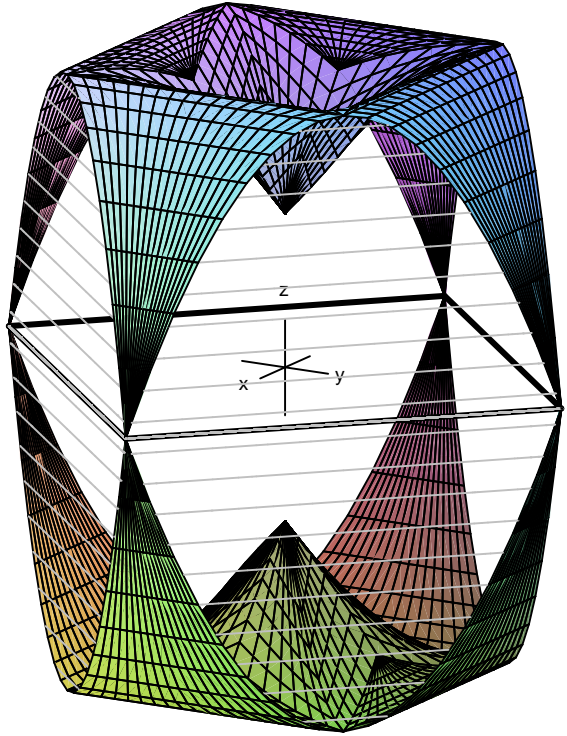}
\caption{The asymptotic shape of large balls in the Cayley graph of the Heisenberg group
$H(\Bbb{Z})=\left\langle x,y|[x,[x,y]]=[y,[x,y]]=1\right\rangle $ viewed in exponential coordinates.}
\end{center}
\end{figure}


Combining this with Theorem \ref{weaklycom}, we also get the following corollary, of which Theorem \ref{firsthm} is only
a special case with $\rho $ the word metric associated to the generating set
$\Omega $.

\begin{corollary}[Volume asymptotics] \label{volasym}Suppose that $G$ is a locally compact
group with polynomial growth and $\rho $ is a periodic pseudodistance on $G$%
. Let $B_{\rho }(t)$ be the $\rho $-ball of radius $t$ in $G,$ i.e. $B_{\rho
}(t)=\{x\in G,\rho (e,x)\leq t\},$ then there exists a constant $c(\rho )>0$
such that the following limit exists:
\begin{equation}
\lim_{t\rightarrow +\infty }\frac{vol_{G}(B_{\rho }(t))}{t^{d(G)}}=c(\rho )
\label{limitvol}
\end{equation}
\end{corollary}

Here $d(G)$ is the integer $d(S_{N})$, the so-called homogeneous dimension
of the nilshadow $S_{N}$ of a Lie shadow $S$ of $G$ (obtained by Theorem \ref
{weaklycom}), and is given by the Bass-Guivarc'h formula:
\begin{equation}
d(S_{N})=\sum_{k\geq 0}\dim (C^{k}(S_{N}))  \label{BG}
\end{equation}
where $\{C^{k}(S_{N})\}_{k}$ is the descending central series of $S_{N}.$

The limit $c(\rho )$ is equal to the volume $vol_{S}(\mathcal{C})$ of the
limit shape $\mathcal{C}$ from Corollary \ref{shape} once we make the right
choice of Haar measure on a Lie shadow $S$ of $G.$ Let us explain this
choice. Recall that according to Theorem \ref{general}, $G/K$ admits a
co-compact subgroup $H/K$ which embeds co-compactly in $S.$ Starting with a
Haar measure $vol_{G}$ on $G$, we get a Haar measure on $G/K$ after fixing
the Haar measure of $K$ to be of total mass $1,$ and we may then choose a
Haar measure on $H/K$ so that the compact quotient $G/H$ has volume $1.$
Finally we choose the Haar measure on $S$ so that the other compact quotient
$S/(H/K)$ has volume $1$. This gives the desired Haar measure $vol_{S}$ such
that $c(\rho )=vol_{S}(\mathcal{C}).$ 

Note that Haar measure on $S$ is also invariant under the group of automorphisms $T(S)$ and is thus left invariant for the nilshadow structure on $S$. It is also left invariant for the graded nilshadow structure. In both exponential coordinates of the first kind (on $S_N$) and of the second kind (as in Lemma \ref{coordo}), Haar measure is just Lebesgue measure.\\

In the case of the discrete Heisenberg group of dimension $3$ equipped with the word metric given by the standard generators, it is possible to compute the constant $c(\rho)$ and the volume of the limit shape as shown in Figure 1. In this case the volume is $\frac{31}{72}$ (see the Appendix). The $5$-dimensional Heisenberg group can also be worked out and the volume of its limit shape (associated to the word metric given by standard generators) is equal to $\frac{2009}{21870}+\frac{\log 2}{32805}$. The fact that this number is transcendental implies that the growth series of this group, i.e. the formal power series $\sum_{n \geq 0} |B_\rho(n)|z^n$ is not algebraic in the sense that it is not a solution of a polynomial equation with rational functions in $\C(z)$ as coefficients (see \cite[Prop. 3.3.]{Sto2}). This was observed by Stoll in \cite{Sto2} by more direct combinatorial means. Stoll also shows there the interesting fact that the growth series can be rational for some other choices of generating sets in the $5$-dimensional Heisenberg group. So rationality of the growth series depends on the generating set. \\

Another interesting feature is asymptotic invariance:

\begin{corollary}[Asymptotic invariance]\label{asyinv} Let $S$ be a simply connected solvable
Lie group with polynomial growth and $\rho $ a periodic pseudodistance on $S.
$ Let $*$ be the new Lie product on $S$ given by the nilshadow group structure (or the graded nilshadow group structure). Then $\rho (e,g*x)/\rho
(e,x)\rightarrow 1$ as $x\rightarrow \infty $ for every $g\in S.$
\end{corollary}

This follows immediately from Theorem \ref{main-theorem}, when $*$ is the graded nilshadow product, and from Theorem \ref{Pansu} below in the case $*$ is the nilshadow group structure.

It is worth observing that we may not in general replace $*$ by the ordinary product on $S$. Indeed, let for instance $S=\Bbb{R}\ltimes \Bbb{R}^{2}$ be the universal
cover of the group of motions of the Euclidean plane, then $S,$ like its
nilshadow $\Bbb{R}^{3}$, admits a lattice $\Gamma \simeq \Bbb{Z}^{3}$. The
quotient $S/\Gamma $ is diffeomorphic to the $3$-torus $\Bbb{R}^{3}/\Bbb{Z}%
^{3}$ and it is easy to find Riemannian metrics on this torus so that their
lift to $\Bbb{R}^{3}$ is not invariant under rotation around the $z$-axis.
Hence this metric, viewed on the Lie group $S$ will not be asymptotically
invariant under left translation by elements of $S$. Nevertheless, if the metric is left-invariant and not just periodic, then we have the following corollary of the proof of Theorem \ref{main-theorem}.

\begin{corollary}[Left-invariant pseudodistances are asymptotic to subFinsler metrics]\label{finsler} Let $S$ be a simply connected solvable Lie group of polynomial growth and $\rho$ be a periodic pseudodistance on $S$ which is invariant under all left-translations by elements of $S$ (e.g. a left-invariant coarsely geodesic metric on $S$). Then there is a left-invariant subFinsler metric $d$ on $S$ which is asymptotic to $\rho$ in the sense that $\frac{\rho(e,g)}{d(e,g)} \to 1$ as $g \to \infty$.
\end{corollary}

We already mentioned above that determining the exact limit shape of a word metric on $S$ is a difficult task. Consequently so is the task of telling when two distinct word metrics are asymptotic. The above statement says that in any case every word metric on $S$ is asymptotic to some left-invariant subFinsler metric. So the set of possible limit shapes is no richer for word metrics than for left-invariant subFinsler metrics.

We note that in the case of nilpotent Lie groups (where $K$ is trivial), Theorem \ref{main-theorem} shows that every periodic metric is asymptotic to a left-invariant metric. It is still an open problem to determine whether every coarsely geodesic periodic metric  is at a bounded distance from a left-invariant metric (this is Burago's theorem in $\R^n$, more about it below).\\

Theorems \ref{weaklycom} and \ref{main-theorem} allow us to describe the asymptotic cone of $(G,\rho)$ for any periodic pseudodistance $\rho$ on any locally compact group with polynomial growth.

\begin{corollary}[Asymptotic cone]\label{asycones} Let $G$ be a locally compact group with
polynomial growth and $\rho$ a periodic pseudodistance on $G$. Then the sequence of pointed metric spaces $\{(G,\frac{1}{n}\rho,e)\}_{n\geq 1}$ converges in the Gromov-Hausdorff topology. The limit is the metric space $(N,d_\infty,e)$, where $N$ is a graded simply connected nilpotent Lie group and $d_\infty$ a left invariant
subFinsler metric on $N$. Moreover the Lie group $N$ is (up to isomorphism) independent of $\rho$.
 The space $(N,d_{\infty })$ is isometric to ``the asymptotic cone'' associated to $(G,\rho ).$ This asymptotic cone is independent of the choice of ultrafilter used to define it.
\end{corollary}

This corollary is a generalization of Pansu's theorem ((10) in
\cite{Pan}). We refer the reader to the book \cite{gromov-pansu-lafontaine} for the definitions of the asymptotic cone and the Gromov-Hausdorff convergence. We discuss in Section \ref{speed} the speed of convergence (in the Gromov-Hausdorff metric) in this theorem and its corollaries about volume growth. In particular there is a major difference between the discrete nilpotent case and the solvable non nilpotent case. In the former, one can find a polynomial rate of convergence \cite{breuillard-ledonne}, while in the latter no such rate exist in general (see Theorem \ref{SmallSpeed}).

\subsection{Folner sets and ergodic theory}

A consequence of Corollary \ref{volasym} is that sequences of balls with
radius going to infinity are Folner sequences, namely:

\begin{corollary}
\label{Folner}Let $G$ be a locally compact group with polynomial growth and $%
\rho $ a periodic pseudodistance on $G$. Let $B_{\rho }(t)$ be the $\rho $%
-ball of radius $t$ in $G.$ Then $\{B_{\rho }(t)\}_{t>0}$ form a Folner
family of subsets of $G$ namely, for any compact set $F$ in $G,$ we have ($\Delta$ denotes the symmetric difference)
\begin{equation}
\lim_{t\rightarrow +\infty }\frac{vol_{G}(FB_{\rho }(t)\Delta B_{\rho }(t))}{%
vol_{G}(B_{\rho }(t))}=0  \label{folnerunif}
\end{equation}
\end{corollary}

\proof%
Indeed $FB_{\rho }(t)\Delta B_{\rho }(t)\subset B_{\rho }(t+c)\backslash
B_{\rho }(t)$ for some $c>$ depending on $F$. Hence $(\ref{folnerunif})$
follows from $(\ref{limitvol})$.
\endproof%

This settles the so-called ``localization problem'' of Greenleaf for locally
compact groups of polynomial growth (see \cite{Gre}), i.e. determining
whether the powers of a compact generating set $\{\Omega ^{n}\}_{n}$ form a
Folner sequence. At the same time it implies that the ergodic theorem for $G$%
-actions holds along any sequence of balls with radius going to infinity.

\begin{theorem}
\label{ergodic}(Ergodic Theorem) Let be given a locally compact group $G$
with polynomial growth together with a measurable $G$-space $X$ endowed with
a $G$-invariant ergodic probability measure $m$. Let $\rho $ be a periodic
pseudodistance on $G$ and $B_{\rho }(t)$ the $\rho $-ball of radius $t$ in $%
G.$ Then for any $p$, $1\leq p<\infty ,$ and any function $f\in \Bbb{L}%
^{p}(X,m)$ we have
\begin{equation*}
\lim_{t\rightarrow +\infty }\frac{1}{vol_{G}(B_{\rho }(t))}\int_{B_{\rho
}(t)}f(gx)dg=\int_{X}fdm
\end{equation*}
for $m$-almost every $x\in X$ and also in $\Bbb{L}^{p}(X,m).$
\end{theorem}

In fact, Corollary \ref{Folner} above, was the ``missing block'' in the proof
of the ergodic theorem on groups of polynomial growth. So far and to my
knowledge, Corollary \ref{Folner} and Theorem \ref{ergodic} were known only
along some subsequence of balls $\{B_{\rho }(t_{n})\}_{n}$ chosen so that (%
\ref{folnerunif}) holds (see for instance \cite{Cal} or \cite{Tem}). This
issue was drawn to my attention by A. Nevo and was my initial motivation for the present work.
We refer the reader to the A.
Nevo's survey paper \cite{Nevo} Section 5.

It later turned out that the mere fact that balls are Folner in a given polynomial growth locally compact group can also be derived from the fact these groups are doubling metric spaces (which is an easier result than the precise asymptotics $vol(\Omega^n) \sim c_\Omega n^{d(G)}$ proved in this paper and only requires lower and upper bounds of the form $c_1 n^{d(G)} \leq vol(\Omega^n) \leq c_2 n^{d(G)}$). This was observed by R. Tessera \cite{Tes} who rediscovered a cute argument of Colding and Minicozzi \cite[Lemma 3.3.]{colding-minicozzi} showing that the volume of spheres $\Omega^{n+1}\setminus \Omega^n$ is at most some $O(n^{-\delta})$ times the volume of the ball $\Omega^n$, where $\delta>0$ is a positive constant depending only on the doubling constant the word metric induced by $\Omega$ in $G$.

In \cite{breuillard-ledonne}, we give a better upper bound (which depends only on the nilpotency class and not on the doubling constant) for the volume of spheres in the case of finitely generated nilpotent groups. This is done by showing the following error term in the asymptotics of the volume of balls: we have $vol(\Omega^n)=c_\Omega n^{d(G)} + O(n^{d(G)-\alpha_r})$, where $\alpha_r>0$ depends only on the nilpotency class $r$ of $G$. We refer the reader to Section \ref{speed} and to the preprint \cite{breuillard-ledonne} for more information on this. We only note here that although the above Colding-Minicozzi-Tessera upper bound on the volume of spheres holds generally for all locally compact groups $G$ with polynomial growth, unless $G$ is nilpotent, there is no error term in general in the asymptotics of the volume of balls. An example with arbitrarily small speed is given in \S \ref{cones}.

\subsection{A conjecture of Burago and Margulis}
In \cite{Bur2} D. Burago and G. Margulis conjectured that any two word metrics on a finitely generated group which are asymptotic (in the sense that $\frac{\rho_1(e,\gamma)}{\rho_2(e,\gamma)}$ tends to $1$ at infinity) must be at a bounded distance from one another (in the sense that $|\rho_1(e,\gamma)-\rho_2(e,\gamma)|=O(1)$). This holds for abelian groups. An analogous result was proved by Abels and Margulis for word metrics on reductive groups \cite{abels-margulis}. S. Krat \cite{krat} established this property for word metrics on the Heisenberg group $H_3(\Z)$. However using Theorem \ref{main-theorem} (which in this particular case of finitely generated nilpotent groups is just Pansu's theorem \cite{Pan})  we will show in Section \ref{burbur}, that there are counter-examples and exhibit two word metrics on $H_3(\Z) \times \Z$ which are asymptotic and yet are not at a bounded distance. For more on this counter-example, and how to adequately modify the conjecture of Burago and Margulis, we refer the interested reader to \cite{breuillard-ledonne}.

\subsection{Organization of the paper}

Sections 2-4 are devoted to preliminaries. In Section \ref{quasinormsection} we present the basic nilpotent theory as can be found in Guivarc'h's thesis \cite{Gui}. In particular, a full proof of the Bass-Guivarc'h formula is given. In Section \ref{nilshadow}, we recall the construction of the nilshadow of a solvable Lie group. In Section \ref{periodic-metrics} we set up the axioms and basic properties of the (pseudo)distance functions that are studied in this paper.

Sections 5-7 contain the core of the proof of the main theorems. In Section 5, we assume
that $G$ is a simply connected solvable Lie group and reduce the problem to the nilpotent case. In
Section 6, we assume that $G$ is a simply connected nilpotent Lie group and prove Theorem
\ref{MetComp} in this case following the strategy used by Pansu in \cite{Pan}. In Section 7, we
prove Theorem \ref{general} for general locally compact groups and reduce the proof of the results of the introduction to the Lie case.

In the last section we make further comments about the speed of
convergence. In particular we give examples answering negatively the aforementioned question of
Burago and Margulis.

The Appendix is devoted to the discrete Heisenberg groups of dimension 3 and 5. We compute their
limit balls, explain Figure 1, and recover the main result of Stoll \cite{Sto2}.

The reader who is mainly interested in the nilpotent group case can read
directly Section 6 while keeping an eye on Sections 2 and 4 for background
notations and elementary facts.

Finally, let us mention that the results and methods of this paper were
largely inspired by the works of Y. Guivarc'h \cite{Gui} and P. Pansu \cite
{Pan}.

\subsection{Nota Bene} A version of this article circulated since 2007. The present version contains essentially the same material, only the exposition has been improved and several somewhat sketchy arguments have been replaced by full fledged proofs (in particular in Sections \ref{nilshadow} and  \ref{locg}). This delay is due to the fact that I was planning for a long time to improve Section \ref{nilpotent-case} and show an error term in the volume asymptotics of balls in nilpotent groups. E. Le Donne and I recently managed to achieve this and it has now become an independent joint paper \cite{breuillard-ledonne}.

\section{Quasi-norms and the geometry of nilpotent Lie groups}\label{quasinormsection}

In this section, we review the necessary background material on nilpotent
Lie groups. In paragraph \ref{Guiv}, we give some crucial properties of
homogeneous quasi norms and reproduce some lemmas originally due to Y.
Guivarc'h which will be used in the sequel. Meanwhile, we prove the
Bass-Guivarc'h formula for the degree of polynomial growth of nilpotent Lie
groups, following Guivarc'h's original argument.

\subsection{Carnot-Caratheodory metrics\label{CCmetrics}}

Let $G$ be a connected Lie group with Lie algebra $\frak{g}$ and let $m_{1}$
be a vector subspace of $\frak{g}.$ We denote by $\left\| \cdot \right\| $ a
norm on $m_{1}$.

We now recall the definition of a left-invariant \textit{Carnot-Carath\'{e}odory metric } also called \emph{subFinsler metric} on $G.$ Let $x,y\in G.$ We consider all possible piecewise smooth paths $\xi
:[0,1]\rightarrow G$ going from $\xi (0)=x$ to $\xi (1)=y.$ Let $\xi
^{\prime }(u)$ be the tangent vector which is pulled back to the identity by
a left translation, i.e.
\begin{equation}
\frac{d\xi }{du}=\xi (u)\cdot \xi ^{\prime }(u)  \label{CCequation}
\end{equation}
where $\xi ^{\prime }(u)\in \frak{g}$ and the notation $\xi (u)\cdot \xi
^{\prime }(u)$ means the image of $\xi ^{\prime }(u)$ under the differential
at the identity of the left translation by the group element $\xi (u).$ We
say that the path $\xi $ is \textit{horizontal} if the vector $\xi ^{\prime
}(u)$ belongs to $m_{1}$ for all $u\in [0,1].$ We denote by $\mathcal{H}$
the set of piecewise smooth horizontal paths. The Carnot-Carath\'{e}odory
metric associated to the norm $\left\| \cdot \right\| $ is defined by:
\begin{equation*}
d(x,y)=\inf \{\int_{0}^{1}\left\| \xi ^{\prime }(u)\right\| du,\xi \in
\mathcal{H},\text{ }\xi (0)=x,\xi (1)=y\}
\end{equation*}
where the infimum is taken over all piecewise smooth paths $\xi
:[0,1]\rightarrow N$ with $\xi (0)=x,\xi (1)=y$ that are horizontal in the
sense that $\xi ^{\prime }(u)\in m_{1}$ for all $u$. If $\|\cdot \|$ is a Euclidean norm, the metric $d(x,y)$ is also called \emph{subRiemannian}. In this paper however the norm $\|\cdot\|$ will typically not be Euclidean (it can be polyhedral like in the case of word metrics on finitely generated nilpotent groups) and $d(x,y)$ will only be \emph{subFinsler}. If $m_{1}=\frak{g},$
and $\left\| \cdot \right\| $ is a Euclidean (resp. arbitrary) norm on $%
\frak{g}$, then $d$ is simply the usual left-invariant Riemannian (resp.
Finsler) metric associated to $\left\| \cdot \right\| .$

Chow's theorem (e.g. see \cite{Gro2} or \cite{Monty}) tells us that $d(x,y)$ is finite for
all $x$ and $y$ in $G$ if and only if the vector subspace $m_{1}$, together
with all brackets of elements of $m_{1},$ generates the full Lie algebra $%
\frak{g}$. If this condition is satisfied, then $d$ is a distance on $G$
which induces the original topology of $G$.

In this paper, we will only be concerned with Carnot-Caratheodory metrics on
a simply connected nilpotent Lie group $N$. In the sequel, whenever we speak
of a Carnot-Carath\'{e}odory metric on $N,$ we mean one that is associated
to a norm $\left\| \cdot \right\| $ on a subspace $m_{1}$ such that $\frak{n}%
=m_{1}\oplus [\frak{n},\frak{n}]$ where $\frak{n}=Lie(N).$ It is easy to
check that any such $m_{1}$ generates the Lie algebra $\frak{n}$.

\begin{remark}
\label{CC-ball}Let us observe here that for such a metric $d$ on $N,$ we
have the following description of the unit ball for $\left\| \cdot \right\| $%
\begin{equation*}
\left\{ v\in m_{1},\left\| v\right\| \leq 1\right\} =\left\{ \frac{\pi
_{1}(x)}{d(e,x)},x\in N\backslash \{e\}\right\}
\end{equation*}
where $\pi _{1}$ is the linear projection from $\frak{n}$ (identified with $N
$ via $\exp$) to $m_{1}$ with kernel $[\frak{n},\frak{n}].$ Indeed, $\pi _{1}$ gives
rise to a homomorphism from $N$ to the vector space $m_{1}.$ And if $\xi (u)$
is a horizontal path from $e$ to $x,$ then applying $\pi _{1}$ to $(\ref
{CCequation})$ we get $\frac{d}{du}\pi _{1}(\xi (u))=\xi ^{\prime }(u),$
hence $\pi _{1}(x)=\int_{0}^{1}\xi ^{\prime }(u)du.$ Hence $%
\left\| \pi _{1}(x)\right\| \leq d(e,x)$ with equality if $x\in m_{1}.$
\end{remark}

\subsection{Dilations on a nilpotent Lie group and the \label{dilations}%
associated graded group\label{gradedLie}}

We now focus on the case of simply connected nilpotent Lie groups. Let $N$
be such a group with Lie algebra $\frak{n}$ and nilpotency class $r$. For
background about analysis on such groups, we refer the reader to the book
\cite{CG}. The exponential map is a diffeomorphism between $\frak{n}$ and $N$%
. Most of the time, if $x\in \frak{n}$, we will abuse notation and denote
the group element $\exp (x)$ simply by $x$. We denote by $\{C^{p}(\frak{n}%
)\}_{p}$ the central descending series for $\frak{n},$ i.e. $C^{p+1}(\frak{n}%
)=[\frak{n},C^{p}(\frak{n})]$ with $C^{0}(\frak{n})=\frak{n}$ and $C^{r}(%
\frak{n})=\{0\}.$

Let $(m_{p})_{p\geq 1}$ be a collection of vector subspaces of $\frak{n}$
such that for each $p\geq 1$,
\begin{equation}
C^{p-1}(\frak{n})=C^{p}(\frak{n})\oplus m_{p}.  \label{suppl}
\end{equation}
Then $\frak{n}=\oplus _{p\geq 1}m_{p}$ and in this decomposition, any
element $x$ in $\frak{n}$ (or $N$ by abuse of notation) will be written in
the form
\begin{equation*}
x=\sum_{p\geq 1}\pi _{p}(x)
\end{equation*}
where $\pi _{p}(x)$ is the linear projection onto $m_{p}$.

To such a decomposition is associated a one-parameter group of dilations $(\delta
_{t})_{t>0}$. These are the linear endomorphisms of $\frak{n}$ defined by
\begin{equation*}
\delta _{t}(x)=t^{p}x
\end{equation*}
for any $x\in m_{p}$ and for every $p$. Conversely, the one-parameter group $(\delta
_{t})_{t\geq 0}$ determines the $(m_{p})_{p\geq 1}$'s since they appear as
eigenspaces of each $\delta _{t}$, $t\neq 1$. The dilations $\delta _{t}$ do
not preserve \textit{a priori} the Lie bracket on $\frak{n}$. This is the
case if and only if
\begin{equation}
\lbrack m_{p},m_{q}]\subseteq m_{p+q}  \label{gradation}
\end{equation}
for every $p$ and $q$ (where $[m_{p},m_{q}]$ is the subspace spanned by
all commutators of elements of $m_{p}$ with elements of $m_{q}$). If (\ref
{gradation}) holds, we say that the $(m_{p})_{p\geq 1}$ form a \textit{%
stratification} of the Lie algebra $\frak{n}$, and that $\frak{n}$ is a \textit{%
stratified} (or homogeneous) Lie algebra. It is an exercise to check that $(\ref
{gradation})$ is equivalent to require $[m_{1},m_{p}]=m_{p+1}$ for all $p.$

If (\ref{gradation}) does not hold, we can however consider a new Lie
algebra structure on the real vector space $\frak{n}$ by defining the new
Lie bracket as $[x,y]_{\infty }=\pi _{p+q}([x,y])$ if $x\in m_{p}$ and $y\in
m_{q}$. This new Lie algebra $\frak{n}_{\infty }$ is stratified and has the same
underlying vector space as $\frak{n.}$ We denote by $N_{\infty }$ the
associated simply connected Lie group. Moreover the $(\delta _{t})_{t>0}$ form a one-parameter group
of automorphisms of $\frak{n}_{\infty }$. In fact the original Lie bracket $%
[x,y]$ on $\frak{n}$ can be deformed continuously to $[x,y]_{\infty }$
through a continuous family of Lie algebra structures by setting
\begin{equation}
\lbrack x,y]_{t}=\delta _{\frac{1}{t}}([\delta _{t}x,\delta _{t}y])
\label{deformation}
\end{equation}
and letting $t\rightarrow +\infty $. Note that conversely, if the $\delta
_{t}$'s are automorphisms of $\frak{n}$, then $[x,y]=\pi _{p+q}([x,y])$ for
all $x\in m_{p}$ and $y\in m_{q}$, and $\frak{n}=\frak{n}_{\infty }.$

The graded Lie algebra associated to $\frak{n}$ is by definition
\begin{equation*}
gr(\frak{n})=\bigoplus_{p\geq 0}C^{p}(\frak{n})/C^{p+1}(\frak{n})
\end{equation*}
endowed with the Lie bracket induced from that of $\frak{n}$. The quotient
map $m_{p}\rightarrow C^{p}(\frak{n})/C^{p+1}(\frak{n})$ gives rise to a
linear isomorphism between $\frak{n}$ and $gr(\frak{n})$, which is a Lie
algebra isomorphism between the new Lie algebra structure $\frak{n}_{\infty
} $ and $gr(\frak{n}).$ Hence stratified Lie algebra structures induced by a
choice of supplementary subspaces $(m_{p})_{p\geq 1}$ as in (\ref{suppl})
are all isomorphic to $gr(\frak{n}).$

On $N_\infty$ the left-invariant subFinsler metrics $d_{\infty }$ associated to a choice of norm on $m_{1}$ are of
special interest. The one-parameter group of dilations $\{\delta _{t}\}_t$ is an automorphism of $N_{\infty }$ and that
\begin{equation}
d_{\infty }(\delta _{t}x,\delta _{t}y)=td_{\infty }(x,y)  \label{dilation}
\end{equation}
for any $x,y\in N_{\infty }$. The metric space $(N_\infty,d_\infty)$ is called a \emph{Carnot group}.

If on the other hand the simply connected nilpotent Lie group $N$ is not stratified,  then the group of dilations $(\delta _{t})_{t}$ associated to a
choice of supplementary vector subspaces $m_{i}$'s as in (\ref{suppl}) will
not consist of automorphisms of $N$ and the relation (\ref{dilation}) will
not hold.

Note also that if we are given two different choices of supplementary
subspaces $m_{i}$'s and $m_{i}^{\prime }$'s as in (\ref{suppl}), then the left-invariant
Carnot-Caratheodory metrics on the corresponding stratified Lie groups are
isometric if and only if $(m_{1},\left\| \cdot \right\| )$ and $%
(m_{1}^{\prime },\left\| \cdot \right\| ^{\prime })$ are isometric (a linear isomorphism from $m_{1}$ to $m_{1}^{\prime }$ that
sends $\left\| \cdot \right\| $ to $\left\| \cdot \right\| ^{\prime }$ extends to an isometry of the two Carnot groups).

\subsection{The Campbell-Hausdorff formula}

The exponential map $\exp :\frak{n}\rightarrow N$ is a diffeomorphism.
In the sequel, we will often abuse notation and identify $N$ and $\frak{n}$
without further notice. In particular, for two elements $x$ and $y$ of $%
\frak{n}$ (or $N$ equivalently) $xy$ will denote their product in $N$, while
$x+y$ denotes the sum in $\frak{n}$. Let $(\delta _{t})_{t}$ be a one-parameter group
of dilations associated to a choice of supplementary subspaces $m_{i}$'s as
in (\ref{suppl}). We denote the corresponding stratified Lie algebra by $\frak{n}%
_{\infty }$ as above and the Lie group by $N_{\infty }.$ The product on $%
N_{\infty }$ is denoted by $x*y$. On $N_{\infty }$ the dilations $(\delta
_{t})_{t}$ are automorphisms.

The Campbell-Hausdorff formula (see \cite{CG}) allows to give a more precise
form of the product in $N.$ Let $(e_{i})_{1\leq i\leq d}$ be a basis of $%
\frak{n}$ adapted to the decomposition into $m_{i}$'s, that is $%
m_{i}=span\{e_{j},e_{j}\in m_{i}\}.$ Let $x=x_{1}e_{1}+...+x_{d}e_{d}$ the
corresponding decomposition of an element $x\in \frak{n}$. Then define the
degree $d_{i}=\deg (e_{i})$ to be the largest $j$ such that $e_{i}\in
C^{j-1}(\frak{n}).$ If $\alpha =(\alpha _{1},...,\alpha _{d})\in \Bbb{N}^{d}$
is a multi-index, then let $d_{\alpha }=\deg (e_{1})\alpha _{1}+...+\deg
(e_{d})\alpha _{d}.$

The Campbell-Hausdorff formula yields
\begin{equation}
(xy)_{i}=x_{i}+y_{i}+\sum C_{\alpha ,\beta }x^{\alpha }y^{\beta }
\label{CHF0}
\end{equation}
where $C_{\alpha ,\beta }$ are real constants and the sum is over all
multi-indices $\alpha $ and $\beta $ such that $d_{\alpha }+d_{\beta }\leq
\deg (e_{i})$, $d_{\alpha }\geq 1$ and $d_{\beta }\geq 1$.

From (\ref{deformation}), it is easy to give the form of the associated
stratified Lie group law:
\begin{equation}
(x*y)_{i}=x_{i}+y_{i}+\sum C_{\alpha ,\beta }x^{\alpha }y^{\beta }
\label{CHF}
\end{equation}
where the sum is restricted to those $\alpha $'s and $\beta $'s such that $%
d_{\alpha }+d_{\beta }=\deg (e_{i})$, $d_{\alpha }\geq 1$ and $d_{\beta
}\geq 1$.

\subsection{Homogeneous quasi-norms and Guivarc'h's theorem on polynomial
growth}\label{Guiv}

Let $\frak{n}$ be a finite dimensional real nilpotent Lie algebra and
consider a decomposition
\begin{equation*}
\frak{n}=m_{1}\oplus ...\oplus m_{r}
\end{equation*}
by supplementary vector subspaces as in (\ref{suppl}). Let $(\delta
_{t})_{t>0}$ be the one parameter group of dilations associated to this
decomposition, that is $\delta _{t}(x)=t^{i}x$ if $x\in m_{i}$. We now
introduce the following definition.

\begin{definition}[Homogeneous quasi-norm]
\label{quasinorm}A continuous function $|\cdot |:\frak{n}\rightarrow \Bbb{R}%
_{+}$ is called a homogeneous quasi-norm associated to the dilations $%
(\delta _{t})_{t}$, if it satisfies the following properties:

$(i)$ $|x|=0\Leftrightarrow x=0.$

$(ii)$ $|\delta _{t}(x)|=t|x|$ for all $t>0.$
\end{definition}

\begin{example}
\label{exampl}(1) Quasi-norms of supremum type, i.e. $|x|=\max_{p}\left\|
\pi _{p}(x)\right\| _{p}^{1/p}$ where $\left\| \cdot \right\| _{p}$ are
ordinary norms on the vector space $m_{p}$ and $\pi _{p}$ is the projection
on $m_{p}$ as above.

(2) $|x|=d_{\infty }(e,x)$, where $d_{\infty }$ is a Carnot-Carath\'{e}odory
metric on a stratified nilpotent Lie group (as the relation (\ref{dilation})
shows).
\end{example}

Clearly, a quasi-norm is determined by its sphere of radius $1$ and two
quasi-norms (which are homogeneous with respect to the same group of
dilations) are always equivalent in the sense that
\begin{equation}
\frac{1}{c}\left| \cdot \right| _{1}\leq \left| \cdot \right| _{2}\leq
c\left| \cdot \right| _{1}  \label{equivalent}
\end{equation}
for some constant $c>0$ (indeed, by continuity, $|\cdot |_{2}$ admits a
maximum on the ``sphere'' $\{|x|_{1}=1\}$). If the two quasi-norms are
homogeneous with respect to two distinct semi-groups of dilations, then the
inequalities (\ref{equivalent}) continue to hold outside a neighborhood of $%
0 $, but may fail near $0$.

Homogeneous quasi-norms satisfy the following properties:

\begin{proposition}
\label{propert}Let $|\cdot |$ be a homogeneous quasi-norm on $\frak{n}$,
then there are constants $C,C_{1},C_{2}>0$ such that

$(a)$ $|x_{i}|\leq C\cdot |x|^{\deg (e_{i})}$ if $x=x_{1}e_{1}+...+x_{n}e_{n}
$ in an adapted basis $(e_{i})_{i}.$

$(b)$ $|x^{-1}|\leq C\cdot |x|.$

$(c)$ $|x+y|\leq C\cdot (|x|+|y|)$

$(d)$ $|xy|\leq C_{1}(|x|+|y|)+C_{2}.$
\end{proposition}

Properties $(a)$, $(b)$ and $(c)$ are straightforward from the fact that $%
|x|=\max_{p}\left\| \pi _{p}(x)\right\| _{p}^{1/p}$ is a homogeneous
quasi-norm and from (\ref{equivalent}). Property $(d)$ justifies the term
``quasi-norm'' and follows from Lemma \ref{guivarch} below. It can be a
problem that the constant $C_{1}$ in $(d)$ may not be equal to $1$. In fact,
this is why we use the word quasi-norm instead of just norm, because we do
not require the triangle inequality axiom to hold. However the following
lemma of Guivarc'h is often a good enough remedy to this situation. Let $%
\left\| \cdot \right\| _{p}$ be an arbitrary norm on the vector space $m_{p}$%
.

\begin{lemma}
\label{guivarch}\label{GuivLem}\label{jauge}(Guivarc'h, \cite{Gui} lemme
II.1) Let $\varepsilon >0$. Up to rescaling each $\left\| \cdot \right\| _{p}
$ into a proportional norm $\lambda _{p}\left\| \cdot \right\| _{p}$ ($%
\lambda _{p}>0$) if necessary, the quasi-norm $|x|=\max_{p}\left\| \pi
_{p}(x)\right\| _{p}^{1/p}$ satisfies
\begin{equation}
|xy|\leq |x|+|y|+\varepsilon   \label{almostriangle}
\end{equation}
for all $x,y\in N$. If $N$ is stratified with respect to $(\delta _{t})_{t}$ we
can take $\varepsilon =0$.
\end{lemma}

This lemma is crucial also for computing the coarse asymptotics of volume
growth. For the reader's convenience, we reproduce here Guivarc'h's
argument, which is based on the Campbell-Hausdorff formula (\ref{CHF0}).

\proof%
We fix $\lambda _{1}=1$ and we are going to give a condition on the $\lambda
_{i}$'s so that (\ref{almostriangle}) holds. The $\lambda _{i}$'s will be
taken to be smaller and smaller as $i$ increases. We set $%
|x|=\max_{p}\left\| \pi _{p}(x)\right\| _{p}^{1/p}$ and let $|x|_{\lambda
}=\max_{p}\left\| \lambda _{p}\pi _{p}(x)\right\| _{p}^{1/p}$ for any $r$%
-tuple of $\lambda _{i}$'s. We want that for any index $p\leq r,$%
\begin{equation}
\lambda _{p}\left\| \pi _{p}(xy)\right\| _{p}\leq \left( |x|_{\lambda
}+|y|_{\lambda }+\varepsilon \right) ^{p}  \label{zooo}
\end{equation}
By (\ref{CHF0}) we have $\pi _{p}(xy)=\pi _{p}(x)+\pi _{p}(y)+P_{p}(x,y)$
where $P_{p}$ is a polynomial map into $m_{p}$ depending only on the $\pi
_{i}(x)$ and $\pi _{i}(y)$ with $i\leq p-1$ such that
\begin{equation*}
\left\| P_{p}(x,y)\right\| _{p}\leq C_{p}\cdot \sum_{l,m\geq 1,l+m\leq
p}M_{p-1}(x)^{l}M_{p-1}(y)^{m}
\end{equation*}
where $M_{k}(x):=\max_{i\leq k}\left\| \pi _{i}(x)\right\| _{i}^{1/i}$ and $%
C_{p}>0$ is a constant depending on $P_{p}$ and on the norms $\left\| \cdot
\right\| _{i}$'s. Since $\varepsilon >0,$ when expanding the right hand side
of (\ref{zooo}) all terms of the form $|x|_{\lambda }^{l}|y|_{\lambda }^{m}$
with $l+m\leq p$ appear with some positive coefficient, say $\varepsilon
_{l,m}$. The terms $|x|_{\lambda }^{p}$ and $|y|_{\lambda }^{p}$ appear with
coefficient $1$ and cause no trouble since we always have $\lambda
_{p}\left\| \pi _{p}(x)\right\| _{p}\leq |x|_{\lambda }^{p}$ and $\lambda
_{p}\left\| \pi _{p}(y)\right\| _{p}\leq |y|_{\lambda }^{p}$. Therefore, for
(\ref{zooo}) to hold, it is sufficient that
\begin{equation*}
\lambda _{p}C_{p}M_{p-1}(x)^{l}M_{p-1}(y)^{m}\leq \varepsilon
_{l,m}|x|_{\lambda }^{l}|y|_{\lambda }^{m}
\end{equation*}
for all remaining $l$ and $m.$ However, clearly $M_{k}(x)\leq \Lambda
_{k}\cdot |x|_{\lambda }$ where $\Lambda _{k}:=\max_{i\leq k}\{1/\lambda
_{i}^{1/i}\}\geq 1.$ Hence a sufficient condition for (\ref{zooo}) to hold
is
\begin{equation*}
\lambda _{p}\leq \frac{\overline{\varepsilon }}{C_{p}\Lambda _{p-1}^{p}}
\end{equation*}
where $\overline{\varepsilon }=\min \varepsilon _{l,m}$. Since $\Lambda
_{p-1}$ depends only on the first $p-1$ values of the $\lambda _{i}$'s, it
is obvious that such a set of conditions can be fulfilled by a suitable $r$%
-tuple $\lambda .$
\endproof

\begin{remark}
The constant $C_{2}$ in Property $(d)$ above can be taken to be $0$ when $N$
is stratified with respect to the $m_{i}$'s (i.e. the $\delta _{t}$'s are
automorphisms), as is easily seen after changing $x$ and $y$ into their
image under $\delta _{t}$. And conversely, if $C_{2}=0$ for some $\delta _{t}
$-homogeneous quasi-norm on $N,$ then $N$ admits a stratification. Indeed, from (\ref{CHF0}%
) and (\ref{CHF}), we see that if the $\delta _{t}$'s are not automorphisms, then one can find $%
x,y\in N$ such that, when $t$ is small enough, $|\delta _{t}(xy)-\delta
_{t}(x)\delta _{t}(y)|\geq ct^{(r-1)/r}$ for some $c>0.$ However, combining
Properties $(c)$ and Property $(d)$ with $C_{2}=0$ above we must have $%
|\delta _{t}(xy)-\delta _{t}(x)\delta _{t}(y)|=O(t)$ near $t=0.$ A
contradiction.
\end{remark}

Guivarc'h's lemma enables us to show:

\begin{theorem}
\label{guivthm}\label{coarsecomp}(Guivarc'h ibid.) Let $\Omega $ be a
compact neighborhood of the identity in a simply connected nilpotent Lie group $N$ and $\rho _{\Omega }(x,y)=\inf
\{n\geq 1,x^{-1}y\in \Omega ^{n}\}$. Then for any homogeneous quasi-norm $%
|\cdot |$ on $N,$ there is a constant $C>0$ such that
\begin{equation}
\frac{1}{C}|x|\leq \rho _{\Omega }(e,x)\leq C|x|+C  \label{zoop}
\end{equation}
\end{theorem}

\proof%
Since any two homogeneous quasi-norms (w.r.t the same one-parameter group of dilations) are equivalent, it is enough to do the
proof for one of them, so we consider the quasi-norm obtained in Lemma \ref
{guivarch} with the extra property (\ref{almostriangle}). The lower bound in
(\ref{zoop}) is a direct consequence of (\ref{almostriangle}) and one can
take there $C$ to be $\max \{|x|,x\in \Omega \}+\varepsilon .$ For the upper
bound, it suffices to show that there is $C\in \Bbb{N}$ such that for all $%
n\in \Bbb{N}$, if $|x|\leq n$ then $x\in \Omega ^{Cn}.$ To achieve this, we
proceed by induction of the nilpotency length of $N.$ The result is clear
when $N$ is abelian. Otherwise, by induction we obtain $C_{0}\in \Bbb{N}$
such that $x=\omega _{1}\cdot ...\cdot \omega _{C_{0}n}\cdot z$ where $%
\omega _{i}\in \Omega $ and $z\in C^{r-1}(N)$ whenever $|x|\leq n.$ Hence $%
|z|\leq |x|+C_{0}n\cdot \max |\omega _{i}^{-1}|+\varepsilon C_{0}\cdot n\leq
C_{1}n$ for some other constant $C_{1}\in \Bbb{N}.$ So we have reduced the
problem to $x=z\in m_{r}=C^{r-1}(N)$ which is central in $N.$ We have $%
z=z_{1}^{n^{r}}$ where $|z_{1}|=|z|/n\leq C_{1}.$ Since $\Omega $ is a
neighborhood of the identity in $N,$ the set $\mathcal{U}$ of all products
of at most $\dim (m_{r})$ simple commutators of length $r$ of elements in $%
\Omega $ is a neighborhood of the identity in $C^{r-1}(N)$ (e.g. see \cite
{Gro2}, p113). It follows that there is a constant $C_{2}\in \Bbb{N}$ such
that $z_{1}$ is in $\mathcal{U}^{C_{2}}$, hence the product of at most $%
C_{2}\dim (m_{r})$ simple commutators. Then we are done because $z$ itself
will be equal to the same product of commutators where each letter $x_{i}\in
\Omega $ is replaced by $x_{i}^{n}.$ This last fact follows from the
following lemma:

\begin{lemma}
Let $G$ be a nilpotent group of nilpotency class $r$ and $n_{1},...,n_{r}$
be positive integers. Then for any $x_{1},...,x_{r}\in G$%
\begin{equation*}
\lbrack
x_{1}^{n_{1}},[x_{2}^{n_{2}},[...,x_{r}^{n_{r}}]...]=[x_{1},[x_{2},[...,x_{r}]...]^{n_{1}\cdot ...\cdot n_{r}}
\end{equation*}
\end{lemma}

To prove the lemma it suffices to use induction and the following obvious
fact: if $[x,y]$ commutes to $x$ and $y$ then $[x^{n},y]=[x,y]^{n}.$%
\endproof%

Finally, we obtain:

\begin{corollary}
Let $\Omega $ be a compact neighborhood of the identity in $N.$ Then there
are positive constants $C_{1}$ and $C_{2}$ such that for all $n\in \Bbb{N},$%
\begin{equation*}
C_{1}n^{d}\leq vol_{N}(\Omega ^{n})\leq C_{2}n^{d}
\end{equation*}
where $d$ is given by the Bass-Guivarc'h formula:
\begin{equation}
d=\sum_{i\geq 1}i\cdot \dim m_{i}  \label{GuiBass}
\end{equation}
\proof%
By Theorem \ref{guivthm}, it is enough to estimate the volume of the
quasi-norm balls. By homogeneity of the quasi-norm, we have $%
vol_{N}\{x,|x|\leq t\}=t^{d}vol_{N}\{x,|x|\leq 1\}$.
\endproof%
\end{corollary}

\begin{remark}
The use of Malcev's embedding theorem allows, as Guivarc'h observed, to
deduce immediately that the analogous result holds for virtually nilpotent
finitely generated groups. This fact that was also proven independently by H.
Bass \cite{Bass} by a direct combinatorial argument. See also Tits' appendix to Gromov's paper \cite{Gro}. In fact Guivarc'h's Theorem \ref{guivthm} seems to have been rediscovered several times in the past 40 years, including by Pansu in his thesis \cite{Pan}, the latest example of that being \cite{Kar}.
\end{remark}

\section{The nilshadow\label{nilsh}\label{nilshadow}}\label{nilshadowsection}

The goal of this section is to introduce the nilshadow of a simply connected solvable Lie
group $G$. We will assume that $G$ has polynomial growth, although this last assumption is not necessary for almost everything we do in this section. The only statement which will be used afterwards in the paper (in Section \ref{nilpotent-reduction}) is Lemma \ref{Kinvdec} below. The reader familiar with the nilshadow can jump directly to the statement of this lemma and skip the forthcoming discussion.\\

\subsection{Construction of the nilshadow}
The nilshadow of $G$ is a simply connected nilpotent Lie group $G_{N}$, which is
associated to $G$ in a natural way. This notion was first introduced by
Auslander and Green in \cite{AG} in their study of flows on solvmanifolds.
They defined it as the unipotent radical of a \textit{semi-simple splitting}
of $G.$ However, we are going to follow a different approach for its
construction by working first at the Lie algebra level. We refer the reader
to the book \cite{DR} where this approach is taken up.

Let $\frak{g}$ be a solvable real Lie algebra and $\frak{n}$ the nilradical
of $\frak{g.}$ We have $[\frak{g},\frak{g}]\subset \frak{n}.$ If $x\in \frak{%
g},$ we write $ad(x)=ad_{s}(x)+ad_{n}(x)$ the Jordan decomposition of $ad(x)$
in $GL(\frak{g}).$ Since $ad(x)\in Der(\frak{g}),$ the space of derivations
of $\frak{g}$, and $Der(\frak{g})$ is the Lie algebra of the \textit{%
algebraic} group $Aut(\frak{g)},$ the Jordan components $ad_{s}(x)$ and $%
ad_{n}(x)$ also belong to $Der(\frak{g}).$ Moreover, for each $x\in \frak{g}$%
, $ad_{s}(x)$ sends $\frak{g}$ into $\frak{n}$ (because so does $ad(x)$ and $%
ad_{s}(x)$ is a polynomial in $ad(x)$). Let $\frak{h}$ be a Cartan
subalgebra of $\frak{g}$, namely a nilpotent self-normalizing subalgebra. Recall that the image of a Cartan subalgebra by a surjective Lie algebra homomorphism is again a Cartan subalgebra. Now since $\frak{g}/\frak{n}$ is abelian, it follows that $\frak{h}$ maps onto $\g/\n$, i.e.  $\frak{h}+\frak{n}=\frak{g}$. Moreover $ad_{s}(x)_{|\frak{h}}=0$ if $%
x\in \frak{h}$, because $\frak{h}$ is nilpotent. \\

Now pick any real
vector subspace $\frak{v}$ of $\frak{h}$ in direct sum with $\n$.
Then the following two conditions hold:

$(i)$ $\frak{v}\oplus \frak{n}=\frak{g}$ .

$(ii)$ $ad_{s}(x)(y)=0$ for all $x,y\in \frak{v}.$\\

From $(i)$ and $(ii)$, it follows easily that $ad_{s}(x)$ commutes with $%
ad(y)$, $ad_{s}(y)$ and $ad_{n}(y)$, for all $x,y$ in $\frak{v}$. We have:

\begin{lemma}
\label{linear}The map $\frak{v}\rightarrow Der(\frak{g})$ defined by $%
x\mapsto ad_{s}(x)$ is a Lie algebra homomorphism.
\end{lemma}

\proof%
First let us check that this map is linear. Let $x,y\in \frak{v}$. By the above $ad_{s}(y)$ and $ad_{s}(x)$ commute with
each other (hence their sum is semi-simple) and commute with $%
ad_{n}(x)+ad_{n}(y).$ From the uniqueness of the Jordan decomposition it
remains to check that $ad_{n}(x)+ad_{n}(y)$ is nilpotent if $x,y$ in $\frak{%
v.}$ To see this, apply the following obvious remark twice to $a=ad_{n}(x)$
and $V=ad(\frak{n})$ first and then to $a=ad_{n}(y)$ and $V=span\{ad_{n}(x),$
$ad((ad(y))^{n}x),n\geq 1\}$ : \textit{Let }$V$\textit{\ be a nilpotent
subspace of }$GL(\frak{g})$\textit{\ and }$a\in GL(\frak{g})$\textit{\
nilpotent, i.e. }$V^{n}=0$\textit{\ and }$a^{m}=0$\textit{\ for some }$%
n,m\in \Bbb{N}$\textit{\ and assume }$[a,V]\subset V.$\textit{\ Then }$%
(a+V)^{nm}=0.$

The fact that this map is a Lie algebra homomorphism follows  easily from the fact that
all $ad_{s}(x),$ $x\in \frak{v}$ commute with one another and with $[\frak{g},%
\frak{g}]\subset \frak{n}$.

\endproof%

We define a new Lie bracket on $\frak{g}$ by setting:
\begin{equation}
\lbrack x,y]_{N}=[x,y]-ad_{s}(x_{v})(y)+ad_{s}(y_{v})(x)  \label{NewBracket}
\end{equation}
where $x_{v}$ is the linear projection of $x$ on $\frak{v}$ according to the
direct sum $\frak{v}\oplus \frak{n}=\frak{g}.$ The Jacobi identity is
checked by a straightforward computation where the following fact is needed:
$ad_{s}\left( ad_{s}(x)(y)\right) =0$ for all $x,y\in \frak{g}.$ This holds
because, as we just saw, $ad_{s}(x)(\frak{g})\subset \frak{n}$ for all $x\in
\frak{g}$, and $ad_{s}(a)=0$ if $a\in \frak{n}$.

\begin{definition}
Let $\frak{g}_{N}$ be the vector space $\frak{g}$ endowed with the new Lie
algebra structure $[\cdot ,\cdot ]_{N}$ given by (\ref{NewBracket}). The%
\textbf{\ nilshadow} $G_{N}$ of $G$ is defined to be the simply connected
Lie group with Lie algebra $\frak{g}_{N}.$
\end{definition}

It is easy to check that $\frak{g}_{N}$\textit{\ is a nilpotent Lie algebra}%
. To see this, note first that $[\frak{g}_{N},\frak{g}_{N}]_{N}\subset \frak{%
n}$, and if $x\in \frak{g}_{N}$ and $y\in \frak{n}$ then $%
[x,y]_{N}=(ad_{n}(x_{v})+ad(x_{n}))(y).$ However, $ad_{n}(x_{v})+ad(x_{n})$
is a nilpotent endomorphism of $\frak{n}$ as follows from the same remark
used in the proof of Lemma \ref{linear}. Hence $\frak{g}_{N}$\textit{\ }is a
nilpotent.\\

The nilshadow Lie product on $G_N$ will be denoted by $*$ in order to distinguish it from the original Lie product on $G$. In the sequel, we will often identify $G$ (resp. $G_N$) with its Lie algebra $\g$ (resp. $\g_N$) via their respective exponential map. Since the underlying space of $\g_N$ was $\g$ itself, this gives an identification (although not a group isomorphism) between $G$ and $G_N$. Then the nilshadow Lie product can be expressed in terms of the original product as follows:

$$g*h=g \cdot (T(g^{-1})h)$$

Here $T$ is the Lie group homomorphism $G \rightarrow Aut(G)$ induced by the above choice of supplementary subspace $\vu$ as follows.

\begin{equation}\label{expl}
T(e^{a})(e^{b})=\exp (e^{ad_{s}(a_{v})}b) \textnormal{ }\forall a,b\in \frak{g}.\
\end{equation}

In other words, $T$ is the unique Lie group homomorphism whose differential at the identity is the Lie algebra homomorphism $d_e T: \g \rightarrow Der(\g)$ given by $d_e T(a)(b)=ad_s(a_v)b$, that is the composition of the map $\vu \rightarrow Der(\g)$ from Lemma \ref{linear} with the linear projection $\g \rightarrow \g/\n \simeq \vu$.

It is easy to check that this definition of the new product is compatible with the definition of the new Lie bracket.

It can also be checked that two choices of supplementary spaces $\vu$ as above yield isomorphic Lie structures (see \cite[Chap. III]{DR}). Hence by abuse of language, we speak of \textit{the} nilshadow of $\frak{g}$, when we mean the Lie structure on $G$ induced by a choice of $\vu$ as above.

The following example shows several of the features of a typical solvable
Lie group of polynomial growth.

\begin{example}[Nilshadow of a semi-direct product]
\label{exple}Let $G=\Bbb{R}\ltimes_{\phi }\Bbb{R}^{n}$ where $\phi
_{t}\in GL_{n}(\Bbb{R})$ is some one parameter subgroup given by $\phi
_{t}=\exp (tA)=k_{t}u_{t}$ where $A$ is some matrix in $M_{n}(\Bbb{R})$ and $%
A=A_{s}+A_{u}$ is its Jordan decomposition, giving rise to $k_{t}=\exp
(tA_{s})$ and $u_{t}=\exp (tA_{u}).$ The group $G$ is diffeomorphic to $\Bbb{%
R}^{n+1},$ hence simply connected. If all eigenvalues of $A_{s}$ are purely
imaginary, then $G$ has polynomial growth. However $G$ is not nilpotent
unless $A_{s}=0.$ So let us assume that neither $A_{s}$ nor $A_{u}$ is zero.
Then the nilshadow $G_{N}$ is the semi-direct product $\Bbb{R}%
\ltimes_{u}\Bbb{R}^{n}$ where $u_{t}$ is the unipotent part of $%
\phi _{t}.$

It is easy to compute the homogeneous dimension of $G$ (or $G_{N}$) in terms
of the dimension of the Jordan blocs of $A_{u}.$ If $n_{k}$ is the number of
Jordan blocks of $A_{u}$ of size $k,$ then
\begin{equation*}
d(G)=1+\sum_{k\geq 1}\frac{k(k+1)}{2}n_{k}
\end{equation*}
\end{example}

\subsection{Basic properties of the nilshadow}
We now list in the form of a few lemmas some basic properties of the nilshadow.

\begin{lemma}\label{abelian} The image of $T: G \rightarrow Aut(G)$ is abelian and relatively compact. Moreover $T(T(g)h)=T(h)$ for any $g,h\in G.$
\end{lemma}

\begin{proof} Since $G$ has polynomial growth it is of type $(R)$ by Guivarc'h's theorem. Hence all $ad_s(x)$ have purely imaginary eigenvalues. It follows that $K$ is compact. Since $T$ factors through the nilradical, its image is abelian. The last equality follows from $(\ref{expl})$ and the fact that $\forall x,y \in \g, \textnormal{ } ad_s(ad_s(x)(y))=0 $.
\end{proof}

\begin{lemma}\label{aut} $T(G)$ also belongs to $Aut(G_N)$ and $T$ is a group homomorphism $G_N \rightarrow Aut(G_N)$.
\end{lemma}

\begin{proof} The first assertion follows from $(\ref{expl})$ and the fact that $d_eT$ is a derivation of $\g_N$ as one can check from $(\ref{NewBracket})$ and the fact that $\forall x,y \in \g, \textnormal{ } ad_s(ad_s(x)(y))=0 $. The second assertion then follows from Lemma \ref{abelian}.
\end{proof}

We denote by $K$ the closure of $T(G)$ in $Aut(G) = Aut(\g)$.

\begin{lemma}[K-action on $\g_N$]\label{invofK} $K$ preserves $\vu$ and acts trivially on it. It also preserves the ideals $\n$ and the central descending series $\{C^i(\g_N)\}_i$ of $\g_N$.
\end{lemma}

\begin{proof} It suffices to check that $ad_s(\vu)$ preserves $\n$ and $C^i(\g_N)$. It preserves $\n$ because $ad(x)$ preserves $\n$ for all $x \in \g$. It preserves $C^i(\g_N)$ because it acts as a derivation of $\g_N$ as we have already checked in the proof of Lemma \ref{aut}.
\end{proof}

\begin{remark}[Well-definedness of $\pi_1$]\label{indept} It is also easy to check from the definition of the nilshadow bracket that the commutator subalgebra $[\g_N,\g_N]$ and in fact each term of the central descending series $C^i(\g_N)$ is an ideal in $\g$ and \emph{does not depend} on the choice of supplementary subspace $\vu$ used to defined the nilshadow bracket. In particular the projection map $\pi_1: \g_N \rightarrow \g_N/[\g_N,\g_N]$ is a well defined linear map on $\g=\g_N$ (i.e. independently of the choice involved in the construction of the nilshadow Lie bracket).
\end{remark}

\begin{lemma}[Exponential map] \label{exponent} The respective exponential maps $\exp: \g \rightarrow G$ and $\exp_N : \g_N \rightarrow G_N$ coincide on $\n$ and on $\vu$.
\end{lemma}

\begin{proof} Since the two Lie products coincide on $N=\exp(\n)$, so do their exponential map. For the second assertion, note that $T(e^{-tv})v=v$ for every $v \in \vu$ because $ad_s(x)(y)=0$ for all $x,y \in \nu$. It follows that $\{e^{tv}\}_t$ is a one-parameter subgroup for both Lie structures, hence it is equal to $\{\exp_N(tv)\}_t$.
\end{proof}

\begin{remark}[Surjectivity of the exponential map]
The exponential map is not always a diffeomorphism, as the example of the
universal cover $\widetilde{E}$ of the group $E$ of motions of the plane
shows (indeed any $1$-parameter subgroup of $E$ is either a translation
subgroup or a rotation subgroup, but the rotation subgroup is compact hence
a torus, so its lift will contain the (discrete) center of $E$, hence will
miss every lift of a non trivial translation). In fact, it is easy to see
that if $\frak{g}$ is the Lie algebra of a solvable (non-nilpotent) Lie
group of polynomial growth, then $\frak{g}$ maps surjectively on the Lie
algebra of $E.$ Hence,\textit{\ for a simply connected solvable and
non-nilpotent Lie group of polynomial growth, the exponential map is never
onto}. Nevertheless its image is easily seen to be dense.
\end{remark}

However, exponential coordinates of the second kind behave nicely. Note that $[\g_N,\g_N] \subset \n$.

\begin{lemma}[Exponential coordinates of the second kind]\label{coordo} Let $\{C^i(\g_N)\}_{i\geq 0}$ be the central descending series of $\g_N$ (with $C^1(\g_N)=[\g_N,\g_N]$) and pick linear subspaces $m_i$ in $\g_N$ such that $C^i(\g_N)=m_i \oplus C^{i-1}(\g_N)$ for $i \geq 2$. Let $\ell$ be a supplementary subspace of $C^1(\g_N)$ in $\n$. Define
exponential coordinates of the second kind by setting
\begin{eqnarray*}
m_{r}\oplus ...\oplus m_{2}\oplus \ell \oplus \frak{v} &\rightarrow &G \\
(\xi _{r},...,\xi _{1},v) &\mapsto &\exp _{N}(\xi _{r})*\ldots
*\exp _{N}(\xi _{1})*\exp _{N}(v)
\end{eqnarray*}
This map is a diffeomorphism. Moreover $\exp _{N}(\xi _{r})* \ldots *\exp
_{N}(\xi _{1})*\exp _{N}(v)=e^{\xi _{r}}\cdot ...\cdot e^{\xi
_{1}}\cdot e^{v}$ for all choices of $v\in \frak{v}$ and $\xi _{i}\in m_{i}.$
\end{lemma}

\begin{proof} By Lemma \ref{exponent} the exponential maps of $G$ and $G_N$ coincide on $\n$ and on $\vu$. Moreover $g*h=g\cdot h$ whenever $g$ belongs to the nilradical $\exp(\n)$ of $G$. Hence $\exp _{N}(\xi _{r})* \ldots *\exp
_{N}(\xi _{1})*\exp _{N}(v)=\exp _{N}(\xi _{r})\cdot \ldots \cdot \exp
_{N}(\xi _{1})\cdot \exp _{N}(v)=e^{\xi _{r}}\cdot ...\cdot e^{\xi
_{1}}\cdot e^{v}$. The restriction of the map to $\n$ is a diffeomorphism onto $\exp(\n)$, because this map and its inverse are explicit polynomial maps (the $\xi_i$'s are coordinates of the second kind, see the book \cite{CG}). Now the map $\n \oplus \vu \to G$ sending $(n,v)$ to $e^n \cdot e^v$ is a diffeomorphism, because $G$ is simply connected and hence the quotient group $G/\exp(\n)$ isomorphic to a vector space and hence to $\exp(\vu)$.
\end{proof}

\begin{lemma}[``Bi-invariant'' Riemannian metric]\label{biinv} There exists a Riemannian metric on $G$ which is left invariant under
both Lie structures.
\end{lemma}

\begin{proof}Indeed it suffices to pick a scalar product on $\frak{g}$
which is invariant under the compact subgroup $K=\overline{T(G)}\subset Aut(\frak{g}).$
\end{proof}

We identify $K=\overline{\{T(g),g\in G\}}$ with its image in $Aut(\frak{g})$
under the canonical isomorphism between $Aut(G)$ and $Aut(\frak{g}).$ Recall
that, according to Lemma \ref{invofK}, the central descending series
of $\frak{g}_{N}$ is invariant under $ad_{s}(x)$ for all $x\in \frak{v}$ and
consists of ideals of $\frak{g}.$ The same holds for $\n$.
It follows that these linear subspaces also invariant under $K$. However since $K$ is compact, its action on $\frak{g}$
is completely reducible. Therefore we have proved:

\begin{lemma}[K-invariant stratification of the nilshadow]\label{Kinvdec} Let $\g$ be the Lie algebra of a simply connected Lie group $G$ with polynomial growth. Let $\g_N$ be the nilshadow Lie algebra obtained from a splitting $\g=\n \oplus \vu$ as above (i.e. $\n$ is the nilradical and $\vu$ satisfies $ad_s(x)(y)=0$ for every $x,y \in \vu$). Let $K:=\overline{\{T(g),g \in G\}} \subset Aut(G)$, where $T$ is defined by $(\ref{expl})$. Then there is a choice of linear subspaces $m_i$'s and $\ell$ such that
\begin{equation}\label{decompo}
\g_N=m_r \oplus \ldots m_2 \oplus \ell \oplus \vu,
\end{equation}
where each term is $K$-invariant, $m_1:=\ell\oplus \vu$ and  the central descending series of $\g_N$ satisfies $C^i(\g_N)=m_i \oplus C^{i-1}(\g_N)$. Moreover the action on $K$ can be read off on the exponential coordinates of second kind in this decomposition, namely:
\begin{eqnarray*}
k\left( e^{\xi _{r}}\cdot ...\cdot e^{\xi _{0}}\right) &=&k(e^{\xi
_{r}})\cdot ...\cdot k(e^{\xi _{0}})=e^{k(\xi _{r})}\cdot ...\cdot e^{k(\xi
_{0})} \\
&=&\exp _{N}(k(\xi _{r}))*...*\exp _{N}(k(\xi _{0}))
\end{eqnarray*}
\end{lemma}




\section{Periodic metrics}\label{periodic-metrics}

In this section, unless otherwise stated, $G$ will denote an arbitrary
locally compact group.

\subsection{Definitions}

By a \textit{pseudodistance} (or metric) on a topological space $X$, we mean
a function $\rho :X\times X\rightarrow \Bbb{R}_{+}$ satisfying $\rho
(x,y)=\rho (y,x)$ and $\rho (x,z)\leq \rho (x,y)+\rho (y,z)$ for any triplet
of points of $X$. However $\rho (x,y)$ may be equal to $0$ even if $x\neq y$.

We will require our pseudodistances to be\textit{\ locally bounded},
meaning that the image under $\rho $ of any compact subset of $G\times G$ is
a bounded subset of $\Bbb{R}_{+}$. To avoid irrelevant cases (for instance $%
\rho \equiv 0$) we will also assume that $\rho $ is \textit{proper}, i.e.
the map $y\mapsto \rho (e,y)$ is a proper map, namely the preimage of a bounded set is bounded (we do not ask that the map be continuous). When $\rho $ is locally
bounded then it is proper if and only if $y\mapsto \rho (x,y)$ is proper for
any $x\in G$.

A pseudodistance $\rho $ on $G$ is said to be \textit{asymptotically geodesic%
} if for every $\varepsilon >0$ there exists $s>0$ such that for any $x,y\in
G$ one can find a sequence of points $x_{1}=x,$ $x_{2},...,x_{n}=y$ in $G$
such that
\begin{equation}
\sum_{i=1}^{n-1}\rho (x_{i},x_{i+1})\leq (1+\varepsilon )\rho (x,y)
\label{asymgeo}
\end{equation}
and $\rho (x_{i},x_{i+1})\leq s$ for all $i=1,...,n-1$.

We will consider exclusively pseudodistances on a group $G$ that are \textit{%
invariant} under left translations by all elements of a fixed closed and
co-compact subgroup $H$ of $G$, meaning that for all $x,y\in G$ and all $%
h\in H,$ $\rho (hx,hy)=\rho (x,y).$

Combining all previous axioms, we set the following definition.

\begin{definition}
\label{periodic}\label{IAGDef}Let $G$ be a locally compact group. A
pseudodistance $\rho $ on $G$ will be said to be a $\mathbf{periodic}$ $%
\mathbf{metric}$ (or $H$-periodic metric) if it satisfies the following
properties:

$(i)$ $\rho $ is invariant under left translations by a closed co-compact
subgroup $H$.

$(ii)$ $\rho $ is locally bounded and proper.

$(iii)$ $\rho $ is asymptotically geodesic.
\end{definition}

\begin{remark}
The assumption that $\rho $ is symmetric, i.e. $\rho (x,y)=\rho (y,x)$ is
here only for the sake of simplicity, and most of what is proven in this
paper can be done without this hypothesis.
\end{remark}

\subsection{Basic properties}

Let $\rho $ be a periodic metric on $G$ and $H$ some co-compact subgroup of $%
G$. The following properties are straighforward.

$(1)$ \label{restrict0} $\rho $ is at a bounded distance from its restriction to $H.$ This
means that if $F$ is a bounded fundamental domain for $H$ in $G$ and for an
arbitrary $x\in G$, if $h_{x}$ denotes the element of $H$ such that $x\in
h_{x}F,$ then $\left| \rho (x,y)-\rho (h_{x},h_{y})\right| \leq C$ for some
constant $C>0$.

$(2)$ \label{properness}$\forall t>0$ there exists a compact subset $K_{t}$
of $G$ such that, $\forall x,y\in G,$ $\rho (x,y)\leq t\Rightarrow
x^{-1}y\in K_{t}$. And conversely, if $K$ is a compact subset of $G$, $%
\exists t(K)>0$ s.t. $x^{-1}y\in K\Rightarrow \rho (x,y)\leq t(K).$

$(3)$ If $\rho (x,y)\geq s$, the $x_{i}$'s in (\ref{asymgeo}) can be chosen
in such a way that $s\leq \rho (x_{i},x_{i+1})\leq 2s$ (one can take a
suitable subset of the original $x_{i}$'s).

$(4)$\label{restrict} The restriction of $\rho $ to $H\times H$ is a periodic pseudodistance
on $H$. This means that the $x_{i}$'s in (\ref{asymgeo}) can be chosen in $H$%
.

$(5)$ Conversely, given a periodic pseudodistance $\rho _{H}$ on $H,$ it is
possible to extend it to a periodic pseudodistance on $G$ by setting $\rho
(x,y)=\rho _{H}(h_{x},h_{y})$ where $x=h_{x}F$ for some bounded fundamental
domain $F$ for $H$ in $G$.

\subsection{Examples\label{examples}}

Let us give a few \smallskip examples of periodic pseudodistances.

$(1)$ Let $\Gamma $ be a finitely generated torsion free nilpotent group
which is embedded as a co-compact discrete subgroup of a simply connected
nilpotent Lie group $N$. Given a finite symmetric generating set $S$ of $%
\Gamma ,$ we can consider the corresponding word metric $d_{S}$ on $\Gamma $
which gives rise to a periodic metric on $N$ given by $\rho
(x,y)=d_{S}(\gamma _{x},\gamma _{y})$ where $x\in \gamma _{x}F$ and $y\in
\gamma _{y}F$ if $F$ is some fixed fundamental domain for $\Gamma $ in $N.$

$(2)$ Another example, given in \cite{Pan}, is as follows. Let $N/\Gamma $
be a nilmanifold with universal cover $N$ and fundamental group $\Gamma $.
Let $g$ be a Riemannian metric on $N/\Gamma .$ It can be lifted to the
universal cover and thus gives rise to a Riemannian metric $\widetilde{g}$
on $N$. This metric is $\Gamma $-invariant, proper and locally bounded.
Since $\Gamma $ is co-compact in $N,$ it is easy to check that it is also
asymptotically geodesic hence periodic$.$

$(3)$ Any word metric on $G$. That is, if $\Omega $ is a compact symmetric
generating subset of $G$, let $\Delta _{\Omega }(x)=\inf \{n\geq 1,x\in
\Omega ^{n}\}$. Then define $\rho (x,y)=\Delta _{\Omega }(x^{-1}y).$ Clearly
$\rho $ is a pseudodistance (although not a distance) and it is $G$%
-invariant on the left, it is also proper, locally bounded and
asymptotically geodesic, hence periodic$.$

$(4)$ If $G$ is a connected Lie group, any left invariant Riemannian metric
on $G.$ Here again $H=G$ and we obtain a periodic distance. Similarly, any
left invariant Carnot-Carath\'{e}odory metric on $G$ will do.

\begin{remark}[Berestovski's theorem] According to a result of Berestovski \cite{berestovski} every left-invariant geodesic distance on a connected Lie group is a subFinsler metric as defined in Paragraph \ref{CCmetrics}.
\end{remark}

\subsection{Coarse equivalence between invariant pseudodistances}

The following proposition is basic:

\begin{prop}
\label{coarseq}\label{roughcomp}Let $\rho _{1}$ and $\rho _{2}$ be two
periodic pseudodistances on $G.$ Then there is a constant $C>0$ such that
for all $x,y\in G$
\begin{equation}
\frac{1}{C}\rho _{2}(x,y)-C\leq \rho _{1}(x,y)\leq C\rho _{2}(x,y)+C
\label{uppercomparison}
\end{equation}
\end{prop}

\proof%
Clearly it suffices to prove the upper bound. Let $s>0$ be the number
corresponding to the choice $\varepsilon =1$ in (\ref{asymgeo}) for $\rho
_{2}.$ From \ref{properness} $(2)$, there exists a compact subset $K_{s}$ in
$G$ such that $\rho _{2}(x,y)\leq 2s\Rightarrow x^{-1}y\in K_{2s},$ and
there is a constant $t=t(K_{2s})>0$ such that $x^{-1}y\in K_{2s}\Rightarrow
\rho _{1}(x,y)\leq t$. Let $C=\max \{2t/s,t\},$ and let $x,y\in G.$ If $\rho
_{2}(x,y)\leq s$ then $\rho _{1}(x,y)\leq t$ so the right hand side of (\ref
{uppercomparison}) holds. If $\rho _{2}(x,y)\geq s$ then, from (\ref{asymgeo}%
) and \ref{properness} $(3)$, we get a sequence of $x_{i}$'s in $G$ from $x$
to $y$ such that $s\leq \rho _{2}(x_{i},x_{i+1})\leq 2s$ and $%
\sum_{1}^{N}\rho _{2}(x_{i},x_{i+1})\leq 2\rho _{2}(x,y)$. It follows that $%
\rho _{1}(x_{i},x_{i+1})\leq t$ for all $i.$ Hence $\rho _{1}(x,y)\leq \sum
\rho _{1}(x_{i},x_{i+1})\leq Nt\leq \frac{2}{s}t\rho _{2}(x,y)$ and the
right hand side of (\ref{uppercomparison}) holds.
\edpf%

In the particular case when $G=N$ is a simply connected nilpotent Lie group,
the distance to the origin $x\mapsto \rho (e,x)$ is also coarsely equivalent
to any homogeneous quasi-norm on $N.$ We have,

\begin{prop}
\label{coarsecomp1}Suppose $N$ is a simply connected nilpotent Lie group.
Let $\rho _{1}$ be a periodic pseudodistance on $N$ and $|\cdot |$ be a
homogeneous quasi-norm, then there exists $C>0$ such that for all $x\in N$%
\begin{equation}
\frac{1}{C}|x^{-1}y|-C\leq \rho _{1}(x,y)\leq C|x^{-1}y|+C
\label{uppercomparison2}
\end{equation}
Moreover, if $\rho _{2}$ is a periodic pseudodistance on the stratified
nilpotent group $N_{\infty }$ associated to $N,$ then again, there is a
constant $C>0$ such that
\begin{equation}
\frac{1}{C}\rho _{2}(e,x)-C\leq \rho _{1}(e,x)\leq C\rho _{2}(e,x)+C
\label{GG0}
\end{equation}
\end{prop}

The proposition follows at once from Guivarc'h's theorem (see Corollary \ref
{coarsecomp} above), the equivalence of homogeneous quasi-norms, and the
fact that left-invariant Carnot-Caratheodory metrics on $N_{\infty }$ are homogeneous
quasi norms. However, since the group structures on $N$ and $N_{\infty }$
differ, (\ref{GG0}) cannot in general be replaced by the stronger relation (%
\ref{uppercomparison}) as simple examples show.

The next proposition is of fundamental importance for the study of metrics
on Lie groups of polynomial growth:

\begin{proposition}
\label{coarsecomp2}Let $G$ be a simply connected solvable Lie group of
polynomial growth and $G_{N}$ its nilshadow. Let $\rho $ and $\rho _{N}$ be
arbitrary periodic pseudodistances on $G$ and $G_{N}$ respectively. Then
there is a constant $C>0$ such that for all $x,y\in G$%
\begin{equation}
\frac{1}{C}\rho _{N}(x,y)-C\leq \rho (x,y)\leq C\rho _{N}(x,y)+C
\label{uppercomparison3}
\end{equation}
\end{proposition}

\proof%
According to Proposition \ref{coarseq}, it is enough to show (\ref
{uppercomparison3}) for \textit{some} choice of periodic metrics on $G$ and $%
G_{N}.$ But in Lemma \ref{biinv} we constructed a Riemannian metric
on $G$ which is left invariant for both $G$ and $G_{N}.$ We are done.
\endproof%

\subsection{Right invariance under a compact subgroup}

Here we verify that, given a compact subgroup of $G,$ any periodic metric is
at bounded distance from another periodic metric which is invariant on the
right by this compact subgroup. Let $K$ be a compact subgroup of $G$ and $%
\rho $ a periodic pseudodistance on $G.$ We average $\rho $ with the help of
the normalized Haar measure on $K$ to get:
\begin{equation}
\rho^{K}(x,y)=\int_{K\times K}\rho (xk_1,yk_2)dk_1dk_2  \label{average}
\end{equation}
Then the following holds:

\begin{lemma}
\label{bdedistance}There is a constant $C_{0}>0$ depending only on $\rho $
and $K$ such that for all $k_1,k_2\in K$ and all $x,y\in G$%
\begin{equation}
|\rho (xk_1,yk_2)-\rho (x,y)|\leq C_{0}  \label{zob}
\end{equation}
\end{lemma}

\proof
From \ref{properness} $(2)$, $\exists t=t(K)>0$ s.t. $\forall x\in G$, $\rho
(x,xk)\leq t$. Applying the triangle inequality, we are done.%
\endproof

Hence we obtain:

\begin{proposition}
The pseudodistance $\rho^{K}$ is periodic and lies at a bounded distance
from $\rho .$ In particular, as $x$ tends to infinity in $G$ the following
limit holds
\begin{equation}
\lim_{x\rightarrow \infty }\frac{\rho^{K}(e,x)}{\rho (e,x)}=1  \label{zobo}
\end{equation}
\end{proposition}

\proof%
From Lemma \ref{bdedistance} and \ref{properness} $(3),$ it is easy to check
that $\rho^{K}$ must be asymptotically geodesic, and periodic. Integrating (%
\ref{zob}) we get that $\rho^{K}$ is at a bounded distance from $\rho $ and
(\ref{zobo}) is obvious.%
\endproof%

If $K$ is normal in $G,$ we thus obtain a periodic metric $\rho^{K}$ on $%
G/K $ such that $\rho^{K}(p(x),p(y))$ is at a bounded distance from $\rho
(x,y)$, where $p$ is the quotient map $G\rightarrow G/K.$

\section{Reduction to the nilpotent case}\label{nilpotent-reduction}

In this section, $G$ denotes a \textit{simply connected} solvable Lie group
of polynomial growth. We are going to reduce the proof of the theorems of
the Introduction to the case of a nilpotent $G.$ This is performed by
showing that any $periodic$ pseudodistance $\rho $ on $G$ is asymptotic to
some associated $periodic$ pseudodistance $\rho _{N}$ on the nilshadow $%
G_{N}.$ We state this in Proposition \ref{pop} below.

The key step in the proof is Proposition \ref{invK} below, which shows the
asymptotic invariance of $\rho $ under the ``semisimple part'' of $G.$ The
crucial fact there is that the displacement of a distant point under a fixed
unipotent automorphism is negligible compared to the distance from the
identity (see Lemmas \ref{lem2}, \ref{lem4}), so that the action of the
semisimple part of large elements can be simply approximated by their action
by left translation.

\subsection{Asymptotic invariance under a compact group of automorphisms of $%
G\label{asyminv}$}

The main result of this section is the following. Let $G$ be a connected and simply connected solvable Lie group with polynomial growth and $G_N$ its nilshadow (see Section \ref{nilshadow}).

\begin{proposition}
\label{pop}Let $H$ be a closed co-compact subgroup of $G$ and $\rho $ an $H$%
-periodic pseudodistance (see Definition \ref{periodic}) on $G.$ There exist a closed subset $H_{K}$
containing $H$ which is a co-compact subgroup for both $G$ and $G_{N},$ and
an $H_{K}$-periodic (for both Lie structures) pseudodistance $\rho _{K}$ such
that
\begin{equation}
\lim_{x\rightarrow \infty }\frac{\rho _{K}(e,x)}{\rho (e,x)}=1  \label{Keq}
\end{equation}
\end{proposition}

The closed subgroup $H_K$ will be taken to be the closure of the group generated by all elements of the form $k(h)$, where $h$ belongs to $H$ and $k$ belongs to the closure $K$ in the group $Aut(G)$ of the image of $H$ under the homomorphism $T: G\rightarrow Aut(G)$ introduced in Section \ref{nilsh}. It is easy to check from the definition of the nilshadow product $(\ref{nilproduct})$ that this is indeed a subgroup in both $G$ and its nilshadow $G_N$.

The new pseudodistance $\rho_K$ is defined as follows, using a double averaging procedure:

\begin{equation}
\rho _{K}(x,y):=\int_{H\backslash H_{K}}\int_{K}\rho (gk(x),gk(y))dkd\mu (g)
\label{rhok}
\end{equation}

Here the measure $\mu$ is the normalized Haar measure on the coset space $H\backslash H_K$ and $dk$ is the normalized Haar measure on the compact group $K$. Recall that all closed subgroups of $S$ are unimodular (since they have polynomial growth by \cite{Gui}[Lemme I.3.]). Hence the existence of invariant measures on the coset spaces.\\

An essential part of the proof of Proposition \ref{pop} is enclosed in the
following statement:

\begin{proposition}
\label{invK}\label{convptw}Let $\rho $ be a periodic pseudodistance on $G$
which is invariant under a co-compact subgroup $H.$ Then $\rho $ is
asymptotically invariant under the action of $K=\overline{\{T(h),h\in H\}}%
\subseteq Aut(G).$ Namely, (uniformly) for all $k\in K$,
\begin{equation}
\lim_{x\rightarrow \infty }\frac{\rho (e,k(x))}{\rho (e,x)}=1  \label{Kinv}
\end{equation}
\end{proposition}

The proof of Proposition \ref{invK} splits into two steps. First we show
that it is enough to prove (\ref{Kinv}) for a dense subset of $k$'s$.$ This
is a consequence of the following continuity statement:

\begin{lemma}
\label{lem1}Let $\varepsilon >0,$ then there is a neighborhood $U$ of the
identity in $K$ such that, for all $k\in U,$%
\begin{equation*}
 \overline{\lim }_{x\rightarrow \infty }\frac{\rho (x,k(x))}{%
\rho (e,x)}<\varepsilon
\end{equation*}
\end{lemma}

Then we show that the action of $T(g)$ can be approximated by the
conjugation by $g$, essentially because the unipotent part of this
conjugation does not move $x$ very much when $x$ is far. This is the content
of the following lemma:

\begin{lemma}
\label{lem2} Let $\rho $ be a periodic pseudodistance on $G$ which is
invariant under a co-compact subgroup $H$. Then for any $\varepsilon >0,$
and any compact subset $F$ in $H$ there is $s_{0}>0$ such that
\begin{equation*}
|\rho (e,T(h)x)-\rho (e,hx)|\leq \varepsilon \rho (e,x)
\end{equation*}
for any $h\in F$ and as soon as $\rho (e,x)>s_{0}.$
\end{lemma}

\noindent \textit{Proof of Proposition \ref{invK} modulo Lemmas (\ref{lem1}) and (\ref
{lem2}):} As $\rho $ is assumed to be $H$-invariant, for every $h\in H,$ we
have $\rho (e,h^{-1}x)/\rho (e,x)\rightarrow 1.$ The proof of the
proposition then follows immediately from the combination of the last two
lemmas.
\endproof%

\subsection{Proof of Lemmas (\ref{lem1}) and (\ref{lem2})}

We choose $K$-invariant subspaces $m_i$'s and $\ell$ of the nilshadow $\g_N$ of $\g$ as in Lemma \ref{Kinvdec} from Section \ref{nilshadow}. In particular

$$\g_N=m_r \oplus \ldots \oplus m_2 \oplus \ell \oplus \vu,$$
where each term is $K$-invariant, $\frak{n}=[\frak{g}_{N},\frak{g}_{N}]\oplus \frak{l}$ and $C^i(\g_N)=m_i \oplus C^{i-1}(\g_N)$. Moreover $\delta_t(x)=t^ix$ if $x \in m_i$ (here $m_1=\ell \oplus \vu$).

We also set $v(x)=\max_{i}\left\|
\xi _{i}\right\| _{i}^{1/d_{i}}$ if $x=\exp _{N}(\xi _{r})*\ldots
*\exp _{N}(\xi _{0})$ and $d_{i}=i$ if $i>0$ and $d_{0}=1.$ And we let $|x|:=\max_i\|x_i\|^{1/d_i}$ if $x=x_r+\ldots+x_1+x_0$ in the above direct sum decomposition. 

Note that $|\cdot|$ is a $\delta_t$-homogeneous quasi-norm. Moreover, it is straightforward to verify (using the Campbell-Hausdorff formula $(\ref{CHF})$ and Proposition \ref{propert}) that $v(x) \leq C|x|+C$ for some constant $C>0$. In particular $\xi_i/|x|^{d_i}$ remains bounded as $|x|$ becomes large.

\vspace{.5cm}

\noindent \textit{Proof of Lemma \ref{lem1}. } Combining Propositions \ref{coarsecomp1} and \ref
{coarsecomp2}, there is a constant $C>0$ such that for all $x,y\in G,$ $\rho
(x,y)\leq C|x^{*-1}*y|+C.$ Therefore we have reduced to prove the statement for $|\cdot|$ instead of $\rho$, namely it is enough to show that $|x^{*-1}*k(x)|$ becomes negligible compared to $|x|$ as $|x|$ goes to infinity and $k$ tends to $1$.

It follows from the Campbell-Baker-Hausdorff formula $(\ref{CHF0})$ and $(\ref{CHF})$ that, if $x,y\in
G_{N}$ and $|x|,|y|$ are $O(t),$ then $|\delta_{\frac{1}{t}}(x*y)-\delta _{\frac{1}{t}}(x)*\delta _{\frac{1}{t}}(y)|=O(t^{-1/r}),$ and similarly $|\delta _{\frac{1}{t}}(x_1*\ldots*x_m)-\delta _{\frac{1}{t}}(x_1)*\ldots *\delta _{\frac{1}{t}}(x_m)|=O_m(t^{-1/r}),$ for $m$ elements $x_i$ with $|x_i|=O(t)$. Hence when writing $x=\exp _{N}(\xi _{r})*...%
*\exp _{N}(\xi _{0}),$ and setting $t=|x|,$ we thus obtain that
the following quantity
\begin{equation*}
\left| \delta _{\frac{1}{t}}(x^{*-1}*k(x))-\overset{*}{%
\prod_{0\leq i\leq r}}\exp _{N}(-t^{-d_{i}}\xi _{i})*\overset{*}{%
\prod_{0\leq i\leq r}}\exp _{N}(t^{-d_{r-i}}k(\xi _{r-i}))\right|
\end{equation*}
is a $O(t^{-1/r}).$ Indeed recall from Lemma \ref{Kinvdec} that $k(x)=\exp _{N}(k(\xi _{r}))*...%
*\exp _{N}(k(\xi _{0}))$. As $x$ gets larger, each $t^{-d_{i}}\xi _{i}$ remains in a
compact subset of $m_{i}.$ Therefore, as $k$ tends to the identity in $K$, each $t^{-d_{i}}k(\xi _{i})$ becomes uniformly close to $
t^{-d_{i}}\xi _{i}$ independently of the choice of $x\in G_{N}$ as long as $t=|x|$ is large. The result follows. \edpf

\vspace{.5cm}

\noindent \textit{Proof of Lemma \ref{lem2}.} Recall that $hx=h*T(h)x$ for all $x,h\in G$ (see $(\ref{nilproduct})$. By the triangle inequality it is enough to bound $\rho (y,h*y)$, where $y=T(h)x$. From Propositions \ref{coarsecomp1} and \ref{coarsecomp2}, $\rho$ is comparable (up to multiplicative and additive constants to the homogeneous quasi-norm $|\cdot|$. Hence the Lemma follows from the following:

\begin{lemma}
\label{lem4}Let $N$ be a simply connected nilpotent Lie group and let $%
|\cdot |$ be a homogeneous quasi norm on $N$ associated to some $1$%
-parameter group of dilations $(\delta _{t})_{t}$. For any $\varepsilon >0$
and any compact subset $F$ of $N,$ there is a constant $s_{2}>0$ such that
\begin{equation*}
|x^{-1}gx|\leq \varepsilon |x|
\end{equation*}
for all $g\in F$ and as soon as $|x|>s_{2}.$
\end{lemma}

\proof%
Recall, as in the proof of the last lemma, that for any $c_{1}>0$ there is a
$c_{2}>0$ such that if $t>1$ and $x,y\in N$ are such that $|x|,|y|\leq
c_{1}t,$ then $|\delta _{\frac{1}{t}}(xy)-\delta _{\frac{1}{t}}(x)*\delta _{%
\frac{1}{t}}(y)|\leq c_{2}t^{-1/r}$. In particular, if
we set $t=|x|,$ then
\begin{equation*}
\left| \delta _{\frac{1}{t}}(x^{-1}gx)-\delta _{\frac{1}{t}}(x)^{-1}*\delta
_{\frac{1}{t}}(g)*\delta _{\frac{1}{t}}(x)\right| \leq c_{2}t^{-1/r}
\end{equation*}
On the other hand, as $g$ remains in the compact set $F,$ $\delta _{\frac{1}{%
t}}(g)$ tends uniformly to the identity when $t=|x|$ goes to infinity, and $%
\delta _{\frac{1}{t}}(x)$ remains in a compact set. By continuity, we see
that $\delta _{\frac{1}{t}}(x)^{-1}*\delta _{\frac{1}{t}}(g)*\delta _{\frac{1%
}{t}}(x)$ becomes arbitrarily small as $t$ increases. We are done.
\endproof%
\edpf

\subsection{Proof of Proposition \ref{pop}}

First we prove the following continuity statement:

\begin{lemma}
\label{lem5}Let $\rho $ be a periodic pseudodistance on $G$ and $\varepsilon
>0$. Then there exists a neighborhood of the identity $U$ in $G$ and $s_{3}>0
$ such that
\begin{equation*}
1-\varepsilon \leq \frac{\rho (e,gx)}{\rho (e,x)}\leq 1+\varepsilon
\end{equation*}
as soon $g\in U$ and $\rho (e,x)>s_{3}.$
\end{lemma}

\proof%
Let $\rho_N$ be a left invariant Riemannian metric on the nilshadow $G_N$.
\begin{equation*}
|\rho (e,x)-\rho (e,gx)|\leq \rho (x,gx)\leq \rho(x,g*x)+ \rho(g*x,gx)
\end{equation*}
However $\rho(a,b) \leq C\rho_N(a,b) +C $ for some $C>0$ by Proposition \ref{coarsecomp2}. Moreover by $(\ref{nilproduct})$ we have $gx=g * T(g)x$. Hence
\begin{equation*}
|\rho (e,x)-\rho (e,gx)|\leq C \rho_N(x,g*x) + C \rho_N(x,T(g)x) +2C
\end{equation*}
To complete the proof, we apply Lemmas \ref{lem4} and \ref{lem1} to the right hand side
above. \edpf

We proceed with the proof of Proposition \ref{pop}. Let $L$ be the set of
all $g\in G$ such that $\rho (e,gx)/\rho (e,x)$ tends to $1$ as $x$ tends to
infinity in $G$. Clearly $L$ is a subgroup of $G$. Lemma \ref{lem5} shows
that $L$ is closed. The $H$-invariance of $\rho $ insures that $L$ contains $%
H$. Moreover, Proposition \ref{invK}\textit{\ }implies that $L$ is invariant
under $K$. Consequently $L$ contains $H_{K},$ the closed subgroup generated
by all $k(h),$ $k\in K,$ $h\in H.$ This, together with Proposition \ref{invK}%
, grants pointwise convergence of the integrand in (\ref{Keq}). Convergence
of the integral follows by applying Lebesgue's dominated convergence
theorem. 

The fact that $\rho _{K}$ is invariant under left multiplication by $H$ and
invariant under precomposition by automorphisms from $K$ insures that $\rho
_{K}$ is invariant under $*$-left multiplication by any element $%
h\in H$, where $*$ is the multiplication in the nilshadow $G_{N}.$
Moreover we check that $T(g)\in K$ if $g\in H_{K},$ hence $H_{K}$ is a
\textit{subgroup} of $G_{N}.$ It is clearly co-compact in $G_{N}$ too (if $F$
is compact and $HF=G$ then $H*F_{K}=G$ where $F_{K}$ is the union
of all $k(F)$, $k\in K$).

Clearly $\rho _{K}$ is proper and locally bounded, so in order to finish the
proof, we need only to check that $\rho _{K}$ is asymptotically geodesic. By
$H$-invariance of $\rho _{K}$ and since $H$ is co-compact in $G$, it is
enough to exhibit a pseudogeodesic between $e$ and a point $x\in H.$ Let $%
x=z_{1}\cdot ...\cdot z_{n}$ with $z_{i}\in H$ and $\sum \rho (e,z_{i})\leq
(1+\varepsilon )\cdot \rho (e,x).$ Fix a compact fundamental domain $F$ for $%
H$ in $H_{K}$ so that integration in (\ref{Keq}) over $H\backslash H_{K}$ is
replaced by integration over $F.$ Then for some constant $C_{F}>0$ we have $%
|\rho (g,gz)-\rho (e,gz)|\leq C_{F}$ for $g\in F$ and $z\in H.$ Moreover, it
follows from Proposition \ref{invK}, Lemma \ref{lem5} and the fact that $%
H_{K}\subset L,$ that
\begin{equation}
\rho (e,gk(z))\leq (1+\varepsilon )\cdot \rho (e,z)  \label{ineq00}
\end{equation}
for all $g\in F,$ $k\in K$ and as soon as $z\in G$ is large enough. Fix $s$
large enough so that $C_{F}\leq \varepsilon s$ and so that (\ref{ineq00})
holds when $\rho (e,z)\geq s$. As already observed in the discussion
following Definition \ref{IAGDef} (property \ref{properness} (3)) we may
take the $z_{i}$'s so that $\frac{s}{2}\leq \rho (e,z_{i})\leq s.$ Then $%
nC_{F}\leq ns\varepsilon \leq 3\varepsilon \rho (e,x).$ Finally we get for $%
\varepsilon <1$ and $x$ large enough
\begin{eqnarray*}
\sum \rho _{K}(e,z_{i}) &\leq &C_{F}n+(1+\varepsilon )^{2}\rho (e,x) \\
&\leq &C_{F}n+(1+\varepsilon )^{3}\rho _{K}(e,x) \\
&\leq &(1+10\varepsilon )\cdot \rho _{K}(e,x)
\end{eqnarray*}
where we have used the convergence $\rho _{K}/\rho \rightarrow 1$ that we
just proved.
\edpf


\section{The nilpotent case}\label{nilpotent-case}

In this section, we prove Theorem \ref{main-theorem} and its corollaries stated in the Introduction
for a simply connected nilpotent Lie group. We essentially follow Pansu's argument from \cite{Pan}, although our approach differs somewhat in its presentation. Throughout the section, the nilpotent Lie
group will be denoted by $N,$ and its Lie algebra by $\frak{n}$.

Let $m_{1}$ be any vector subspace of $\frak{n}$ such that $\frak{n}%
=m_{1}\oplus [\frak{n},\frak{n}]$. Let $\pi _{1}$ the associated linear
projection of $\frak{n}$ onto $m_{1}$. Let $H$ be a closed co-compact
subgroup of $N$. To every $H$-periodic pseudodistance $\rho $ on $N$ we associate a
norm $\left\| \cdot \right\| _{0}$ on $m_{1}$ which is the norm whose unit
ball is defined to be the closed convex hull of all elements $\pi
_{1}(h)/\rho (e,h)$ for all $h\in H\backslash \{e\}.$ In other words,
\begin{equation}
E:=\{x\in m_{1},\left\| x\right\| _{0}\leq 1\}=\overline{CvxHull}\left\{
\frac{\pi _{1}(h)}{\rho (e,h)},h\in H\backslash \{e\}\right\}
\label{statement}
\end{equation}
The set $E$ is clearly a convex subset of $m_{1}$ which is symmetric around $%
0$ (since $\rho $ is symmetric). To check that $E$ is indeed the unit ball
of a norm on $m_{1}$ it remains to see that $E$ is bounded and that $0$ lies
in its interior. The first fact follows immediately from $(\ref
{uppercomparison2})$ and Example \ref{exampl}. If $0$ does not lie in the
interior of $E,$ then $E$ must be contained in a proper subspace of $m_{1},$
contradicting the fact that $H$ is co-compact in $N$.

Taking large powers $h^n$, we see that we can replace the set $H\setminus\{e\}$ in the above definition by any neighborhood of infinity in $H$. Similarly, it is easy to see that
the following holds:

\begin{prop}
\label{sphere}For $s>0$ let $E_{s}$ be the closed convex hull of all $\pi
_{1}(x)/\rho (e,x)$ with $x\in N$ and $\rho (e,x)>s$. Then $%
E=\bigcap_{s>0}E_{s}$.
\end{prop}

\proof%
Since $\rho $ is $H$-periodic, we have $\rho (e,h^{n})\leq n\rho (e,h)$ for
all $n\in \Bbb{N}$ and $h\in H$. This shows $E\subset \bigcap_{s>0}E_{s}.$
The opposite inclusion follows easily from the fact that $\rho $ is at a
bounded distance from its restriction to $H,$ i.e. from \ref{properness} $%
(1) $.
\endproof%

We now choose a set of supplementary subspaces $(m_{i})$ starting with $%
m_{1} $ as in Paragraph \ref{dilations}. This defines a new Lie product $*$
on $N$ so that $N_{\infty }=(N,*)$ is stratified. We can then consider the $*$%
-left invariant Carnot-Carath\'{e}odory metric associated to the norm $%
\left\| \cdot \right\| _{0}$ as defined in Paragraph \ref{CCmetrics} on the
stratified nilpotent Lie group $N_{\infty }.$  In this section, we will prove Theorem \ref{main-theorem} for nilpotent groups in the following form:

\begin{theorem}
\label{Pansu}Let $\rho $ be a periodic pseudodistance on $N$ and $d_{\infty }$ the
Carnot-Carath\'{e}odory metric defined above, then as $x$ tends to
infinity in $N$
\begin{equation}
\lim \frac{\rho (e,x)}{d_{\infty }(e,x)}=1  \label{asym}
\end{equation}
\end{theorem}

Note that $d_{\infty }$ is left-invariant for the $N_{\infty }$ Lie product,
but not the original Lie product on $N$.

Before going further, let us draw some simple consequences.

$(1)$ In Theorem \ref{Pansu} we may replace $d_{\infty }(e,x)$ by $%
d(e,x)$, where $d$ is the left invariant Carnot-Caratheodory metric on $N$
(rather than $N_{\infty }$) defined by the norm $\left\| \cdot \right\| _{0}$
(as opposed to $d_{\infty }$ which is $*$-left invariant). Hence $\rho ,d$
and $d_{\infty }$ are asymptotic. This follows from the combination of
Theorem \ref{Pansu} and Remark \ref{CC-ball}.

$(2)$ Observe that the choice of $m_{1}$ was arbitrary. Hence two
Carnot-Carath\'{e}odory metrics corresponding to two different choices of a
supplementary subspace $m_{1}$ with the same induced norm on $\frak{n}/[%
\frak{n},\frak{n}]$, are asymptotically equivalent (i.e. their ratio tends
to $1$), and in fact isometric (see Remark \ref{CC-ball}). Conversely, if
two Carnot-Carath\'{e}odory metrics are associated to the same supplementary
subspace $m_{1}$ and are asymptotically equivalent, they must be equal. This
shows that the set of all possible norms on the quotient vector space $\frak{%
n}/[\frak{n},\frak{n}]$ is in bijection with the set of all classes of
asymptotic equivalence of Carnot-Carath\'{e}odory metrics on $N_{\infty }$.

$(3)$ As another consequence we see that if a locally bounded proper and
asymptotically geodesic left-invariant pseudodistance on $N$ is also
homogeneous with respect to the $1$-parameter group $(\delta _{t})_{t}$
(i.e. $\rho (e,\delta _{t}x)=t\rho (e,x)$) then it has to be of the form $%
\rho (x,y)=d_{\infty }(e,x^{-1}y)$ where $d_{\infty }$ is a
Carnot-Carath\'{e}odory metric on $N_{\infty }$.

\subsection{Volume asymptotics}

Theorem \ref{Pansu} also yields a formula for the asymptotic volume of $\rho
$-balls of large radius. Let us fix a Haar measure on $N$ (for example
Lebesgue measure on $\frak{n}$ gives rise to a Haar measure on $N$ under $%
\exp $). Since $d_{\infty }$ is homogeneous, it is straightforward to compute
the volume of a $d_{\infty }$-ball:
\begin{equation*}
vol(\{x\in N,d_{\infty }(e,x)\leq t\})=t^{d(N)}vol(\{x\in N,d_{\infty
}(e,x)\leq 1\})
\end{equation*}
where $d(N)=\sum_{i\geq 1}\dim (C^{i}(\frak{n}))$ is the \textit{homogeneous
dimension} of $N.$ For a pseudodistance $\rho $ as in the statement of
Theorem \ref{Pansu}, we can define the \textit{asymptotic volume of }$\rho $
to be the volume of the unit ball for the associated Carnot-Carath\'{e}odory
metric $d_{\infty }$.
\begin{equation*}
AsVol(\rho )=vol(\{x\in N,d_{\infty }(e,x)\leq 1\})
\end{equation*}
Then we obtain as an immediate corollary of Theorem \ref{Pansu}:

\begin{corollary}
Let $\rho $ be periodic pseudodistance on $N.$ Then
\begin{equation*}
\lim_{t\rightarrow +\infty }\frac{1}{t^{d(N)}}vol(\{x\in N,\rho (e,x)\leq
t\})=AsVol(\rho )>0
\end{equation*}
\end{corollary}

Finally, if $\Gamma $ is an arbitrary finitely generated nilpotent group, we
need to take care of the torsion elements$.$ They form a normal finite
subgroup $T$ and applying Theorem \ref{Pansu} to $\Gamma /T$, we obtain:

\begin{corollary}
\label{counting}Let $S$ be a finite symmetric generating set of $\Gamma $
and $S^{n}$ the ball of radius $n$ is the word metric $\rho _{S}$
associated to $S,$ then
\begin{equation*}
\lim_{n\rightarrow +\infty }\frac{1}{n^{d(N)}}\#S^{n}=\#T\cdot \frac{%
AsVol(\rho _{\overline{S}})}{vol(N/\overline{\Gamma })}>0
\end{equation*}
where $N$ is the Malcev closure of $\overline{\Gamma }=\Gamma /T$, the
torsion free quotient of $\Gamma ,$ and $d_{\overline{S}}$ is the word
pseudodistance associated to $\overline{S}$, the projection of $S$ in $%
\overline{\Gamma }.$
\end{corollary}

Moreover, it is possible to be a bit more precise about $AsVol(\rho _{%
\overline{S}}).$ In fact, the norm $\left\| \cdot \right\| _{0}$ on $m_{1}$
used to define the limit Carnot-Carath\'{e}odory distance $d_{\infty }$
associated to $\rho _{\overline{S}}$ is a simple polyhedral norm defined by
\begin{equation*}
\left\{ \left\| x\right\| _{0}\leq 1\right\} =CvxHull\left( \pi _{1}(%
\overline{s}),s\in S\right)
\end{equation*}

More generally the following holds. Let $H$ be any closed, co-compact
subgroup of $N.$ Choose a Haar measure on $H$ so that $vol_{N}(N/H)=1$.
Theorem \ref{Pansu} yields:

\begin{corollary}
Let $\Omega $ be a compact symmetric (i.e. $\Omega =\Omega ^{-1}$)
neighborhood of the identity, which generates $H$. Let $\left\| \cdot
\right\| _{0}$ be the norm on $m_{1}$ whose unit ball is $\overline{CvxHull}%
\{\pi _{1}(\Omega )\}$ and let $d_{\infty }$ be the corresponding
Carnot-Carath\'{e}odory metric on $N_{\infty }.$ Then we have the following
limit in the Hausdorff topology
\begin{equation*}
\lim_{n\rightarrow +\infty }\delta _{\frac{1}{n}}(\Omega ^{n})=\left\{ g\in
N,d_{\infty }(e,g)\leq 1\right\}
\end{equation*}
and
\begin{equation*}
\lim_{n\rightarrow +\infty }\frac{vol_{H}(\Omega ^{n})}{n^{d(N)}}%
=vol_{N}\left( \left\{ g\in N,d_{\infty }(e,g)\leq 1\right\} \right)
\end{equation*}
\end{corollary}


\subsection{Outline of the proof}

We first devise some standard lemmas about piecewise approximations of
horizontal paths (Lemmas \ref{RiemIneq}, \ref{compageod}, \ref{fewdiscon}).
Then it is shown (Lemma \ref{largeproducts}) that the original product on $N$
and the product in the associated graded Lie group are asymptotic to each
other, namely, if $(\delta _{t})_{t}$ is a $1$-parameter group of dilations
of $N,$ then after renormalization by $\delta _{\frac{1}{t}},$ the product
of $O(t)$ elements lying in some bounded subset of $N,$ is very close to the
renormalized product of the same elements in the graded Lie group $N_{\infty
}$. This is why all complications due to the fact that $N$ may not be
\textit{a priori} graded and the $\delta _{t}$'s may not be automorphisms
disappear when looking at the large scale geometry of the group. Finally, we
observe (Lemma \ref{sphereapprox}), as follows from the very definition of
the unit ball $E$ for the limit norm $\left\| \cdot \right\| _{0},$ that any
vector in the boundary of $E$, can be approximated, after renormalizing by $%
\delta _{\frac{1}{s}}$ by some element $x\in N$ lying in a fixed annulus $%
s(1-\varepsilon )\leq \rho (e,x)\leq s(1+\varepsilon ).$ This enables us to
assert that any $\rho $-quasi geodesic gives rise, after renormalization, to
a $d_{\infty }$-geodesic (this gives the lower bound in Theorem \ref{Pansu}). And vice-versa, that any $d_{\infty
} $-geodesic can be approximated uniformly by some renormalized $\rho $%
-quasi geodesic (this gives the upper bound in Theorem \ref{Pansu}).

\subsection{Preliminary lemmas}

\begin{lemma}
\label{RiemIneq}Let $G$ be a Lie group and let $\left\| \cdot \right\| _{e}$
be a Euclidean norm on the Lie algebra of $G$ and $d_{e}(\cdot ,\cdot )$ the
associated left invariant Riemannian metric on $G$. Let $K$ be a compact subset of $G$. Then there is a constant
$C_{0}=C_{0}(d_{e},K)>0$ such that whenever $d_{e}(e,u)\leq 1$ and $x,y\in K$%
\begin{equation*}
\left| d_{e}(xu,yu)-d_{e}(x,y)\right| \leq C_{0}d_{e}(x,y)d_{e}(e,u)
\end{equation*}
\end{lemma}

\proof%
The proof reduces to the case when $u$ and $x^{-1}y$ are in a small
neighborhood of $e.$ Then the inequality boils down to the following $%
\left\| [X,Y]\right\| _{e}\leq c\left\| X\right\| _{e}\left\| Y\right\| _{e}$
for some $c>0$ and every $X,Y$ in $Lie(G).$
\endproof%

\begin{lemma}
\label{compageod}Let $G$ be a Lie group, let $\left\| \cdot \right\| $ be
some norm on the Lie algebra of $G$ and let $d_{e}(\cdot ,\cdot )$ be a left
invariant Riemannian metric on $G$. Then for every $L>0$ there is a constant
$C=C(d_{e},\left\| \cdot \right\| ,L)>0$ with the following property. Assume
$\xi _{1},\xi _{2}:[0,1]\rightarrow G$ are two piecewise smooth paths in the
Lie group $G$ with $\xi _{1}(0)=\xi _{2}(0)=e.$ Let $\xi _{i}^{\prime }\in
Lie(G)$ be the tangent vector pulled back at the identity by a left
translation of $G$. Assume that $\sup_{t\in [0,1]}\left\| \xi _{1}^{\prime
}(t)\right\| \leq L$, and that $\int_{0}^{1}\left\| \xi _{1}^{\prime
}(t)-\xi _{2}^{\prime }(t)\right\| dt\leq \varepsilon $. Then
\begin{equation*}
d_{e}(\xi _{1}(1),\xi _{2}(1))\leq C\varepsilon
\end{equation*}
\end{lemma}

\proof%
The function $f(t)=d_{e}(\xi _{1}(t),\xi _{2}(t))$ is piecewise smooth. For
small $dt$ we may write, using Lemma \ref{RiemIneq}
\begin{eqnarray*}
f(t+dt)-f(t) &\leq &d_e(\xi _{1}(t)\xi _{1}^{\prime }(t)dt,\xi _{1}(t)\xi
_{2}^{\prime }(t)dt)+d_e(\xi _{1}(t)\xi _{2}^{\prime }(t)dt,\xi _{2}(t)\xi
_{2}^{\prime }(t)dt)-f(t)+o(dt) \\
&\leq &\left\| \xi _{1}^{\prime }(t)-\xi _{2}^{\prime }(t)\right\|
_{e}dt+C_{0}f(t)\left\| \xi _{2}^{\prime }(t)dt\right\| _{e}+o(dt) \\
&\leq &\varepsilon (t)dt+C_0Lf(t)dt+o(dt)
\end{eqnarray*}
where $\varepsilon (t)=\left\| \xi _{1}^{\prime }(t)-\xi _{2}^{\prime
}(t)\right\| _{e}.$ In other words,
\begin{equation*}
f^{\prime }(t)\leq \varepsilon (t)+C_{0}Lf(t)
\end{equation*}
Since $f(0)=0,$ Gronwall's lemma implies that $f(1)\leq
e^{C_{0}L}\int_{0}^{1}\varepsilon (s)e^{-C_{0}Ls}ds\leq $ $C\varepsilon .$
\edpf%

From now on, we will take \label{compRiemCC}$G$ to be the stratified nilpotent Lie group $%
N_{\infty }$, and $d_{e}(\cdot ,\cdot )$ will denote a left invariant
Riemannian metric on $N_{\infty }$ while $d_{\infty }(\cdot ,\cdot )$ is a
left invariant Carnot-Caratheodory Finsler metric on $N_{\infty }$
associated to some norm $\left\| \cdot \right\| $ on $m_{1}.$

\begin{remark}
There is $c_{0}>0$ such that $c_{0}^{-1}d_{e}(e,x)\leq d_{\infty }(e,x)\leq
c_{0}d_{e}(e,x)^{\frac{1}{r}}$ in a neighborhood of $e$. Hence in the
situation of the lemma we get $d_{\infty }(\xi _{1}(1),\xi _{2}(1))\leq
C_{1}\varepsilon ^{\frac{1}{r}}$ for some other constant $%
C_{1}=C_{1}(L,d_{\infty },d_{e}).$
\end{remark}

\begin{lemma}
\label{linearapprox}Let $N\in \Bbb{N}$ and $d_{N}(x,y)$ be the function in $%
N_{\infty }$ defined in the following way:
\begin{equation*}
d_{N}(x,y)=\inf \{\int_{0}^{1}\left\| \xi ^{\prime }(u)\right\| du,\xi \in
\mathcal{H}_{PL(N)},\text{ }\xi (0)=x,\xi (1)=y\}
\end{equation*}
where $\mathcal{H}_{PL(N)}$ is the set of horizontal paths $\xi $ which are
piecewise linear with at most $N$ possible values for $\xi ^{\prime }.$ Then
we have $d_{N}\rightarrow d_{\infty }$ uniformly on compact subsets of $%
N_{\infty }.$
\end{lemma}

\proof%
Note that it follows from Chow's theorem (e.g. see \cite{Monty} or \cite{Gro2}) that there exists $K_{0}\in \Bbb{N}$
such that $A:=\sup_{d_{\infty }(e,x)=1}d_{K_{0}}(e,x)<\infty .$ Moreover,
since piecewise linear paths are dense in $L^{1},$ it follows for example from Lemma \ref
{compageod} that for each fixed $x$, $d_{n}(e,x)\rightarrow d_{\infty }(e,x)$%
. We need to show that $d_{N}(e,x)\rightarrow d_{\infty }(e,x)$ uniformly in
$x$ satisfying $d_{\infty }(e,x)=1.$ By contradiction, suppose there is a
sequence $(x_{n})_{n}$ such that $d_{\infty }(e,x_{n})=1$ and $%
d_{n}(e,x_{n})\geq 1+\varepsilon _{0}$ for some $\varepsilon _{0}>0.$ We may
assume that $(x_{n})_{n}$ converges to say $x.$ Let $y_{n}=x^{-1}*x_{n}$ and
$t_{n}=d_{\infty }(e,y_{n}).$ Then $d_{K_{0}}(e,y_{n})=t_{n}d_{K_{0}}(e,%
\delta _{\frac{1}{t_{n}}}(y_{n}))\leq At_{n}$. Thus $d_{n}(e,x_{n})\leq
d_{n}(e,x)+d_{n}(e,y_{n})\leq d_{n}(e,x)+At_{n}$ as soon as $n\geq K_{0}.$
As $n$ tends to $\infty $, we get a contradiction.
\endproof%

This lemma prompts the following notation. For $\varepsilon >0$, we let $%
N_{\varepsilon }\in \Bbb{N}$ be the first integer such that $1\leq
d_{N_{\varepsilon }}(e,x)\leq 1+\varepsilon $ for all $x$ with $d_{\infty
}(e,x)=1.$ Then we have:

\begin{lemma}
\label{fewdiscon}For every $x\in N_{\infty }$ with $d_{\infty }(e,x)=1,$ and
all $\varepsilon >0$ there exists a path $\xi :[0,1]\rightarrow N_{\infty }$
in $\mathcal{H}_{PL(N_{\varepsilon })}$ with unit speed (i.e. $\left\| \xi
^{\prime }\right\| =1$) such that $\xi (0)=e$ and $d_{\infty }(x,\xi
(1))\leq C_{2}\varepsilon $ and $\xi ^{\prime }$ has at most one
discontinuity on any subinterval of $[0,1]$ of length $\varepsilon
^{r}/N_{\varepsilon }$.
\end{lemma}

\proof%
We know that there is a path in $\mathcal{H}_{PL(N_{\varepsilon })}$
connecting $e$ to $x$ with length $\ell \leq 1+\varepsilon .$
Reparametrizing the path so that it has unit speed, we get a path $\xi
_{0}:[0,\ell ]\rightarrow N_{\infty }$ in $\mathcal{H}_{PL(N_{\varepsilon
})} $ with $d_{\infty }(x,\xi _{0}(1))=d_{\infty }(\xi _{0}(\ell ),\xi
_{0}(1))\leq \varepsilon .$ The derivative $\xi _{0}^{\prime }$ is constant
on at most $N_{\varepsilon }$ different intervals say $[u_{i},u_{i+1}).$ Let
us remove all such intervals of length $\leq \varepsilon ^{r}/N_{\varepsilon
}$ by merging them to an adjacent interval and let us change the value of $%
\xi _{0}^{\prime }$ on these intervals to the value on the adjacent interval
(it doesn't matter if we choose the interval on the left or on the right).
We obtain a new path $\xi :[0,1]\rightarrow N_{\infty }$ in $\mathcal{H}%
_{PL(N_{\varepsilon })}$ with unit speed and such that $\xi ^{\prime }$ has
at most one discontinuity on any subinterval of $[0,1]$ of length $%
\varepsilon ^{r}/N_{\varepsilon }.$ Moreover $\int_{0}^{1}\left\| \xi
^{\prime }(t)-\xi _{0}^{\prime }(t)\right\| dt\leq \varepsilon ^{r}.$ By
Lemma \ref{compageod} and Remark \ref{compRiemCC}, we have $d_{\infty }(\xi
(1),\xi _{0}(1))\leq C_{1}\varepsilon ,$ hence
\begin{equation*}
d_{\infty }(\xi (1),x)\leq d_{\infty }(x,\xi _{0}(1))+d_{\infty }(\xi
_{0}(1),\xi (1))\leq (C_{1}+1)\varepsilon
\end{equation*}
\endproof%

\begin{lemma}[Piecewise horizontal approximation of paths]
\label{largeproducts}Let $x*y$ denote the product inside the stratified Lie
group $N_{\infty }$ and $x\cdot y$ the ordinary product in $N$. Let $n\in
\Bbb{N}$ and $t\geq n.$ Then for any compact subset $K$ of $N$, and any $%
x_{1},...,x_{n}$ elements of $K$, we have
\begin{equation*}
d_{e}(\delta _{\frac{1}{t}}(x_{1}\cdot ...\cdot x_{n}),\delta _{\frac{1}{t}%
}(x_{1}*...*x_{n}))\leq c_{1}\frac{1}{t}
\end{equation*}
and
\begin{equation*}
d_{e}(\delta _{\frac{1}{t}}(x_{1}*...*x_{n}),\delta _{\frac{1}{t}}(\pi
_{1}(x_{1})*...*\pi _{1}(x_{n})))\leq c_{2}\frac{1}{t}
\end{equation*}
where $c_{1},c_{2}$ depend on $K$ and $d_{e}$ only.
\end{lemma}

\proof%
Let $\left\| \cdot \right\| $ be a norm on the Lie algebra of $N.$ For $%
k=1,...,n$ let $z_{k}=$ $x_{1}\cdot ...\cdot x_{k-1}$ and $%
y_{k}=x_{k+1}*...*x_{n}.$ Since all $x_{i}$'s belong to $K,$ it follows from
$(\ref{GG0})$ that as soon as $t\geq n$, all $\delta _{\frac{1}{t}}(z_{k})$
and $\delta _{\frac{1}{t}}(y_{k})$ for $k=1,...,n$ remain in a bounded set
depending only on $K.$ Comparing $(\ref{CHF})$ and $(\ref{CHF0})$, we see
that whenever $y=O(1)$ and $\delta _{\frac{1}{t}}(x)=O(1)$, we have
\begin{equation}
\left\| \delta _{\frac{1}{t}}(xy)-\delta _{\frac{1}{t}}(x*y)\right\| =O(%
\frac{1}{t^{2}})  \label{plug}
\end{equation}
On the other hand, from $(\ref{CHF})$ it is easy to verify that right $*$%
-multiplication by a bounded element is Lipschitz for $\left\| \cdot
\right\| $ and the Lipschitz constant is locally bounded. It follows that
there is a constant $C_{1}>0$ (depending only on $K$ and $\left\| \cdot
\right\| $) such that for all $k\leq n$
\begin{equation*}
\left\| \delta _{\frac{1}{t}}((z_{k}\cdot x_{k})*y_{k})-\delta _{\frac{1}{t}%
}(z_{k}*x_{k}*y_{k})\right\| \leq C_{1}\left\| \delta _{\frac{1}{t}%
}(z_{k}\cdot x_{k})-\delta _{\frac{1}{t}}(z_{k}*x_{k})\right\|
\end{equation*}
Applying $n$ times the relation $(\ref{plug})$ with $x=x_{1}\cdot ...\cdot
x_{k-1}$ and $y=x_{k},$ we finally obtain
\begin{equation*}
\left\| \delta _{\frac{1}{t}}(x_{1}\cdot ...\cdot x_{n})-\delta _{\frac{1}{t}%
}(x_{1}*...*x_{n})\right\| =O(\frac{n}{t^{2}})=O(\frac{1}{t})
\end{equation*}
where $O()$ depends only on $K$. On the other hand, using $(\ref{CHF0}),$ it
is another simple verification to check that if $x,y$ lie in a bounded set,
then $\frac{1}{c_{2}}d_{e}(x,y)\leq \left\| x-y\right\| $ $\leq
c_{2}d_{e}(x,y)$ for some constant $c_{2}>0.$ The first inequality follows.

For the second inequality, we apply Lemma \ref{compageod} to the paths $\xi
_{1}$ and $\xi _{2}$ starting at $e$ and with derivative equal on $[\frac{k}{%
n},\frac{k+1}{n})$ to $n\delta _{\frac{1}{t}}(x_{k})$ for $\xi _{1}$ and to $%
n\frac{\pi _{1}(x_{k})}{t}$ for $\xi _{2}.$ We get
\begin{equation*}
d_{e}(\delta _{\frac{1}{t}}(x_{1}*...*x_{n}),\delta _{\frac{1}{t}}(\pi
_{1}(x_{1})*...*\pi _{1}(x_{n}))=O(\frac{1}{t}).
\end{equation*}
\edpf%

\begin{remark}
From Remark \ref{compRiemCC} we see that if we replace $d_{e}$ by $d_{\infty
}$ in the above lemma, we get the same result with $\frac{1}{t}$ replaced by
$t^{-\frac{1}{r}}.$
\end{remark}


\begin{lemma}[Approximation in the abelianized group]
\label{sphereapprox}Recall that $\left\| \cdot \right\| _{0}$ is the norm on
$m_{1}$ defined in ($\ref{statement})$. For any $\varepsilon >0,$ there
exists $s_{0}>0$ such that for every $s>s_{0}$ and every $v\in m_{1}$ such
that $\left\| v\right\| _{0}=1,$ there exists $h\in H$ such that
\begin{equation*}
(1-\varepsilon )s\leq \rho (e,h)\leq (1+\varepsilon )s
\end{equation*}
and
\begin{equation*}
\left\| \frac{\pi _{1}(h)}{\rho (e,h)}-v\right\| _{0}\leq \varepsilon
\end{equation*}
\end{lemma}

\proof%
Let $\varepsilon >0$ be fixed. Considering a finite $\varepsilon $-net in $E$%
, we see that there exists a finite symmetric subset $\{g_{1},...,g_{p}\}$
of $H\backslash \{e\}$ such that, if we consider the closed convex hull of $%
\frak{F}=\{f_{i}=\pi _{1}(g_{i})/\rho (e,g_{i})|i=1,...,p\}$ and $\left\|
\cdot \right\| _{\varepsilon }$ the associated norm on $m_{1},$ then $%
\left\| \cdot \right\| _{0}\leq \left\| \cdot \right\| _{\varepsilon }\leq
(1+2\varepsilon )\left\| \cdot \right\| _{0}$. Up to shrinking $\frak{F}$ if
necessary, we may assume that $\left\| f_{i}\right\| _{\varepsilon }=1$ for
all $i$'s. We may also assume that the $f_{i}$'s generate $m_{1}$ as a
vector space. The sphere $\{x,\left\| x\right\| _{\varepsilon }=1\}$ is a
symmetric polyhedron in $m_{1}$ and to each of its facets corresponds $%
d=\dim (m_{1})$ vertices lying in $\frak{F}$ and forming a vector basis of $%
m_{1}$. Let $f_{1},...,f_{d}$, say, be such vertices for a given facet. If $%
x\in m_{1}$ is of the form $x=\sum_{i=1}^{d}\lambda _{i}f_{i}$ with $\lambda
_{i}\geq 0$ for $1\leq i\leq d$ then we see that $\left\| x\right\|
_{\varepsilon }=\sum_{i=1}^{d}\lambda _{i}$, because the convex hull of $%
f_{1},...,f_{d}$ is precisely that facet, hence lies on the sphere $%
\{x,\left\| x\right\| _{\varepsilon }=1\}$.

Now let $v\in m_{1},$ $\left\| v\right\| _{0}=1,$ and let $s>0.$ The half
line $tv,$ $t>0$, hits the sphere $\{x,\left\| x\right\| _{\varepsilon }=1\}$
in one point. This point belongs to some facet and there are $d$ linearly
independent elements of $\frak{F}$, say $f_{1},...,f_{d}$, the vertices of
that facet, such that this point belongs to the convex hull of $%
f_{1},...,f_{d}$. The point $sv$ then lies in the convex cone generated by $%
\pi _{1}(g_{1}),...,\pi _{1}(g_{d})$. Moreover, there is a constant $%
C_{\varepsilon }>0$ ($C_{\varepsilon }\leq \frac{d}{2}\max_{1\leq i\leq
p}\rho (e,g_{i})$) such that
\begin{equation*}
\left\| sv-\sum_{i=1}^{d}n_{i}\pi _{1}(g_{i})\right\| _{\varepsilon }\leq
C_{\varepsilon }
\end{equation*}
for some non-negative integers $n_{1},...,n_{d}$ depending on $s>0.$ Hence
\begin{eqnarray*}
\frac{1}{s}\sum_{i=1}^{d}n_{i}\rho (e,g_{i}) &=&\frac{1}{s}\left\|
\sum_{i=1}^{d}n_{i}\pi _{1}(g_{i})\right\| _{\varepsilon }\leq \frac{1}{s}%
(\left\| sv\right\| _{\varepsilon }+C_{\varepsilon }) \\
&\leq &1+2\varepsilon +\frac{C_{\varepsilon }}{s}\leq 1+3\varepsilon
\end{eqnarray*}
where the last inequality holds as soon as $s>C_{\varepsilon }/\varepsilon .$

Now let $h=g_{1}^{n_{1}}\cdot ...\cdot g_{d}^{n_{d}}\in H$. We have $\pi
_{1}(h)=\sum_{i=1}^{d}n_{i}\pi _{1}(g_{i})$%
\begin{equation*}
\rho (e,h)\geq \left\| \pi _{1}(h)\right\| _{0}\geq s-C_{\varepsilon }\geq
s(1-\varepsilon )
\end{equation*}
Moreover
\begin{equation*}
\rho (e,h)\leq \sum_{i=1}^{d}n_{i}\rho (e,g_{i})\leq s(1+3\varepsilon )
\end{equation*}
Changing $\varepsilon $ into say $\frac{\varepsilon }{5}$ and for say $%
\varepsilon <\frac{1}{2},$ we get the desired result with $s_{0}(\varepsilon
)=\frac{d}{\varepsilon }\max_{1\leq i\leq p}\rho (e,g_{i})$.
\edpf%

\subsection{Proof of Theorem \ref{Pansu}}

We need to show that as $x\rightarrow \infty $ in $N$
\begin{equation*}
1\leq \underline{\lim }\frac{\rho (e,x)}{d_{\infty }(e,x)}\leq \overline{%
\lim }\frac{\rho (e,x)}{d_{\infty }(e,x)}\leq 1
\end{equation*}

First note that it is enough to prove the bounds for $x\in H.$ This follows
from (\ref{properness}) $(1)$.

Let us begin with the lower bound. We fix $\varepsilon >0$ and $%
s=s(\varepsilon )$ as in the definition of an asymptotically geodesic metric
(see $(\ref{asymgeo})$). We know by \ref{properness} $(3)$ and $(4)$ that as
soon as $\rho (e,x)\geq s$ we may find $x_{1},...,x_{n}$ in $H$ with $s\leq
\rho (e,x_{i})\leq 2s$ such that $x=\prod x_{i}$ and $\sum \rho
(e,x_{i})\leq (1+\varepsilon )\rho (e,x).$ Let $t=d_{\infty }(e,x),$ then $%
n\leq \frac{1+\varepsilon }{s}\rho (e,x)$, hence $n\leq \frac{C}{%
s(\varepsilon )}t$ where $C$ is a constant depending only on $\rho $ (see $(%
\ref{uppercomparison2})$). We may then apply Lemma \ref{largeproducts} (and
the remark following it) to get, as $t\geq n$ as soon as $s(\varepsilon
)\geq C,$%
\begin{equation*}
d_{\infty }(\delta _{\frac{1}{t}}(x),\delta _{\frac{1}{t}}(\pi
_{1}(x_{1})*...*\pi _{1}(x_{n})))\leq c_{1}^{\prime }t^{-\frac{1}{r}}
\end{equation*}
But for each $i$ we have $\left\| \pi _{1}(x_{i})\right\| _{0}\leq \rho
(e,x_{i})$ by definition of the norm, hence
\begin{equation*}
t=d_{\infty }(e,x)\leq \sum \left\| \pi _{1}(x_{i})\right\| _{0}+d_{\infty
}(x,\pi _{1}(x_{1})*...*\pi _{1}(x_{n}))\leq (1+\varepsilon )\rho
(e,x)+c_{1}^{\prime }t^{1-\frac{1}{r}}
\end{equation*}
Since $\varepsilon $ was arbitrary, letting $t\rightarrow \infty $ we obtain
\begin{equation*}
\underline{\lim }\frac{\rho (e,x)}{d_{\infty }(e,x)}\geq 1
\end{equation*}

We now turn to the upper bound. Let $t=d_{\infty }(e,x)$ and $\varepsilon
>0. $ According to Lemma \ref{fewdiscon}, there is a horizontal piecewise
linear path $\{\xi (u)\}_{u\in [0,1]}$ with unit speed such that $d_{\infty
}(\delta _{\frac{1}{t}}(x),\xi (1))\leq C_{2}\varepsilon $ and no interval
of length $\geq \frac{\varepsilon ^{r}}{N_{\varepsilon }}$contains more than
one change of direction. Let $s_{0}(\varepsilon )$ be given by Lemma \ref
{sphereapprox} and assume $t>s_{0}(\varepsilon ^{r})N_{\varepsilon
}/\varepsilon ^{r}.$ We split $[0,1]$ into $n$ subintervals of length $%
u_{1},...,u_{n}$ such that $\xi ^{\prime }$ is constant equal to $y_{i}$ on
the $i$-th subinterval and $s_{0}(\varepsilon ^{r})\leq tu_{i}\leq
2s_{0}(\varepsilon ^{r})$. We have $\xi (1)=u_{1}y_{1}*...*u_{n}y_{n}.$
Lemma \ref{sphereapprox} yields points $x_{i}\in H$ such that
\begin{equation*}
\left\| y_{i}-\frac{\pi _{1}(x_{i})}{tu_{i}}\right\| \leq \varepsilon ^{r}
\end{equation*}
and $\rho (e,x_{i})\in [(1-\varepsilon ^{r})tu_{i},(1+\varepsilon
^{r})tu_{i}]$ (note that $tu_{i}>s_{0}(\varepsilon ^{r})$). Let $\overline{%
\xi }$ be the piecewise linear path $[0,1]\rightarrow N_{\infty }$ with the
same discontinuities as $\xi $ and where the value $y_{i}$ is replaced by $%
\frac{\pi _{1}(x_{i})}{tu_{i}}.$ Then according to Lemma \ref{compageod}, $%
d_{\infty }(\xi (1),\overline{\xi }(1))\leq C\varepsilon .$ Since $\rho
(e,x_{i})\leq 4s_{0}(\varepsilon ^{r})$ for each $i$, we may apply Lemma \ref
{largeproducts} (and the remark following it) and see that if $y=x_{1}\cdot
...\cdot x_{n},$
\begin{equation*}
d_{\infty }(\overline{\xi }(1),\delta _{\frac{1}{t}}(y))\leq c_{1}^{\prime
}(\varepsilon )t^{-\frac{1}{r}}
\end{equation*}
Hence $d_{\infty }(\delta _{\frac{1}{t}}(x),\delta _{\frac{1}{t}}(y))\leq
(C_{2}+C)\varepsilon +c_{1}^{\prime }(\varepsilon )t^{-\frac{1}{r}}$ and $%
\rho (e,y)\leq \sum \rho (e,x_{i})\leq (1+\varepsilon ^{r})t$ while $\rho
(x,y)\leq C^{\prime }td_{\infty }(e,\delta _{\frac{1}{t}}(x^{-1}y))+C^{%
\prime }\leq t(Cd_{\infty }(\delta _{\frac{1}{t}}(x),\delta _{\frac{1}{t}%
}(y))+o_{\varepsilon }(1)).$ Hence
\begin{equation*}
\rho (e,x)\leq t+o_{\varepsilon }(t)
\end{equation*}
\endproof%

\begin{remark}
In the last argument we used the fact that $\left\| \delta _{\frac{1}{t}%
}(xu)-\delta _{\frac{1}{t}}(x*u)\right\| =O(\frac{1}{t^{\frac{1}{r}}})$ if $%
\delta _{\frac{1}{t}}(x)$ and $\delta _{\frac{1}{t}}(u)$ are bounded, in
order to get for $y=xu,$%
\begin{eqnarray*}
d_{\infty }(e,\delta _{\frac{1}{t}}(u)) &\leq &d_{\infty }(\delta _{\frac{1}{%
t}}(x),\delta _{\frac{1}{t}}(xu))+d_{\infty }(\delta _{\frac{1}{t}%
}(xu),\delta _{\frac{1}{t}}(x*u)) \\
&\leq &d_{\infty }(\delta _{\frac{1}{t}}(x),\delta _{\frac{1}{t}}(y))+o(1).
\end{eqnarray*}
\end{remark}

\section{Locally compact $G$ and proofs of the main results}\label{locg}

In this section, we prove Theorem \ref{general} and complete the proof of Theorem \ref{main-theorem} and its corollaries. We begin with the latter.

\begin{proof}[Proof of Theorem \ref{main-theorem}]
It is the combination of Proposition \ref{pop}, which reduces the problem to nilpotent Lie groups, and Theorem \ref{Pansu}, which treats the nilpotent case. It only remains to justify the last assertion that $d_\infty$ is invariant under $T(H)$.

Since $K=\overline{T(H)}$ stabilizes $m_1$ (see Lemma \ref{Kinvdec} for the definition of $m_1$) and acts by automorphisms of the nilpotent (nilshadow) structure (Lemma \ref{aut}), given any $k \in K$, the metric $d_\infty(k(x),k(y))$ is nothing else but the left invariant subFinsler metric on the nilshadow associated to the norm $\|k(v)\|$ for $v \in m_1$ (if $\|\cdot\|$ denotes the norm associated to $d_\infty$).

 However, $d_\infty$ is asymptotically invariant under $K$, because of Proposition $\ref{pop}$. Namely $d_\infty(e,k(x))/d_\infty(e,x)$ tends to $1$ as $x$ tends to infinity. Finally $d_\infty(e,v)=\|v\|$ and $d_\infty(e,k(v))=\|k(v)\|$ for all $v \in m_1$. Two asymptotic norms on a vector space are always equal. It follows that the norms $\|\cdot\|$ and $\|k(\cdot)\|$ on $m_1$ coincide. Hence $d_\infty(e,k(x))=d_\infty(e,k(x))$ for all $x \in S$ as claimed.
\end{proof}

\begin{proof}[Proof of Corollary \ref{finsler}] First some initial remark (see also Remark \ref{CC-ball}). If $d$ is a left-invariant subFinlser metric on a simply connected nilpotent Lie group $N$ induced by a norm $\|\cdot\|$ on a supplementary subspace $m_1$ of the commutator subalgebra, then it follows from the very definition of subFinsler metrics (see Paragraph \ref{CCmetrics}) that $\pi_1$ is $1$-Lipschitz between the Lie group and the abelianization of it endowed with the norm $\|\cdot\|$, namely $\|\pi_1(x)\|\leq d(e,x)$, with equality if $x \in m_1$. From this and considering the definition of the limit norm in $(\ref{statement})$, we conclude that $\|\cdot\|$ coincides with the limit norm of $d$. In particular Theorem \ref{Pansu} implies that $d$ is asymptotic to the $*$-left invariant subFinsler metric $d_\infty$ induced by the same norm $\|\cdot\|$ on the graded Lie group $(N_\infty,*)$.

We can now prove Corollary \ref{finsler}. By the above remark, the limit metric $d_\infty$ on the  graded nilshadow of $S$ is asymptotic to the subFinsler metric $d$ induced by the same norm $\|\cdot \|$ on the same ($K$-invariant) supplementary subspace $m_1$ of the commutator subalgebra of the nilshadow, and which is left invariant for the nilshadow structure on $S$. However, it follows from Theorem \ref{main-theorem} that $d_\infty$ and the norm $\|\cdot\|$ are $K$-invariant. This implies that $d$ is also left-invariant with respect to the original Lie group structure of $S$. Indeed, by $(\ref{nilproduct})$, we can write $d(gx,gy)=d(g*(T(g)x),g*(T(g)y))=d(T(g)x,T(g)y)=d(x,y)$, where $*$ denotes this time the nilshadow product structure. We are done.
\end{proof}

\begin{proof}[Proof of Corollary \ref{asyinv}]
This follows immediately from Theorem \ref{main-theorem}, when $*$ denotes the graded nilshadow product. If $*$ denotes the nilshadow group structure, then it follows from Theorem \ref{Pansu} and the remark we just made in the proof of Corollary \ref{finsler} (see also Remark \ref{CC-ball}).
\end{proof}

\subsection{Proof of Theorem \ref{general}.}\label{pfgen}

Let $G$ be a locally compact group of polynomial growth. We will show that $G$ has a compact normal subgroup $K$ such that $G/K$ contains a closed co-compact subgroup, which can be realized as a closed co-compact subgroup of a connected and simply connected solvable Lie group of type $(R)$ (i.e. of polynomial growth). The proof will follow in several steps.\\

(a) First we show that \emph{up to moding out by a normal compact subgroup, we may assume that $G$ is a Lie group whose connected component of the identity has no compact normal subgroup.} 
Indeed, it follows from Losert's refinement of Gromov's theorem (\cite{Los} Theorem 2)
that there exists a normal compact subgroup $K$ of $G$ such that $G/K$ is a
Lie group. So we may now assume that $G$ is a Lie group (not necessarily
connected) of polynomial growth. The connected component $G_{0}$ of $G$ is a
connected Lie group of polynomial growth. Recall the following classical fact:

\begin{lemma} Every connected Lie group has a unique maximal compact normal subgroup. By uniqueness it must be a characteristic Lie subgroup.
\end{lemma}

\begin{proof} Clearly if $K_1$ and $K_2$ are compact normal subgroups, then $K_1K_2$ is again a compact normal subgroup. Considering $G/K$, where $K$ is a compact normal subgroup of maximal dimension, we may assume that $G$ has no compact normal subgroup of positive dimension. But every finite normal subgroup of a connected group is central. Hence the closed group generated by all finite normal subgroups is contained in the center of $G$. The center is an abelian Lie subgroup, i.e. isomorphic to a product of a vector space $\R^n$, a torus $\R^m/\Z^m$, a free abelian group $\Z^k$ and a finite abelian group. In such a group, there clearly is a unique maximal compact subgroup (namely the product of the finite group and the torus). It is also normal, and maximal in $G$.
\end{proof}

The maximal compact normal subgroup of $G_0$ is a characteristic Lie subgroup of of $G_0$. It is therefore normal in $G$ and we may mod out by it. We therefore have shown that every locally compact (compactly generated) group with polynomial growth admits a quotient by a compact normal subgroup, which is a Lie group $G$ whose connected component of the identity $G_0$ has polynomial growth and contains no compact normal subgroup. We will now show that a certain co-compact subgroup of $G$ has the embedding property of Theorem \ref{general}. \\

(b) Second we show that, \emph{up to passing to a co-compact subgroup, we may assume that the connected component $G_0$ is solvable.} For this purpose, let $Q$ be the solvable radical of $G_0$, namely the maximal connected normal Lie subgroup of $G_0$. Note that it is a characteristic subgroup of $G_0$ and therefore normal in $G$. Moreover $G_0/Q$ is a semisimple Lie group. Since $G_0$ has polynomial growth, it follows that $G_0/Q$ must be compact. Consider the action of $G$ by conjugation on $G_0/Q$, namely the map $\phi: G \to Aut(G_0/Q)$. Since $G_0/Q$ is compact semisimple, its group of automorphisms is also a compact Lie group. In particular, the kernel $\ker \phi$ is a co-compact subgroup of $G$.

The connected component of the identity of $Aut(G_0/Q)$ is itself semisimple and hence has finite center. However the image of the connected component $(\ker \phi)_0$ of $\ker \phi$ in $G_0/Q$ modulo $Q$ is central. Therefore it must be trivial. We have shown that $(\ker \phi)_0$  is contained in $Q$ and hence is solvable. Moreover $(\ker \phi)_0$  has no compact normal subgroup, because otherwise its maximal normal compact subgroup, being characteristic in $(\ker \phi)_0$, would be normal in $G$ (note that $(\ker \phi)_0$ is normal in $G$).

Changing $G$ into the co-compact subgroup $\ker \phi$, we can therefore assume that $G_0$ is solvable, of polynomial growth, and has no non trivial compact normal subgroup. The group $G/G_0$ is discrete, finitely generated, and has polynomial growth. By Gromov's theorem, it must be virtually nilpotent, in particular virtually polycyclic.\\

(c) We finally prove the following proposition.

\begin{proposition}\label{empol} Let $G$ be a Lie group such that its connected component of the identity $G_0$ is solvable, admits no compact normal subgroup, and with $G/G_0$ virtually polycyclic. Then $G$ has a closed co-compact subgroup, which can be embedded as a closed co-compact subgroup of a connected and simply connected solvable Lie group.
\end{proposition}

The proof of this proposition is mainly an application of a theorem of H.C. Wang, which is a vast generalization of Malcev's embedding theorem for torsion free finitely generated nilpotent groups. Wang's theorem \cite{Wan} states that any $\mathcal{S}$-group can be embedded as a closed co-compact subgroup of a simply connected real linear solvable Lie group with only finitely many connected components. Wang defines a $\mathcal{S}$-group to be any real Lie group $G$, which admits a normal subgroup $A$ such that $G/A$ is finitely generated abelian and $A$ is a torsion-free nilpotent Lie group whose connected components group is finitely generated. In particular any $\mathcal{S}$-group has a finite index (hence co-compact) subgroup which embeds as a co-compact subgroup in a connected and simply connected solvable Lie group. In order to prove Proposition \ref{empol}, it therefore suffices to establish that $G$ has a co-compact $\mathcal{S}$-group.

We first recall the following simple fact:

\begin{lemma}\label{topfinite} Every closed subgroup $F$ of a connected solvable Lie group $S$ is topologically finitely generated.
\end{lemma}

\begin{proof} We argue by induction on the dimension of $S$. Clearly there is an epimorphism $\pi: S \to \R$. By induction hypothesis $F \cap \ker \pi$ is topologically finitely generated. The image of $F$ is a subgroup of $\R$. However every subgroup of $\R$ contains either one or two elements, whose subgroup they generate has the same closure as the original subgroup. We are done.
\end{proof}

Next we show the existence of a nilradical.

\begin{lemma} Let $G$ be as in Proposition \ref{empol}. Then $G$ has a unique maximal normal nilpotent subgroup $G_N$.
\end{lemma}

\begin{proof} The subgroup generated by any two normal nilpotent subgroups of any given group is itself nilpotent (Fitting's lemma, see e.g. \cite{robinson}[5.2.8]). Let $G_N$ be the closure of the subgroup generated by all nilpotent subgroups of $G$. We need to show that $G_N$ is nilpotent. For this it is clearly enough to prove that it is topologically finitely generated (because any finitely generated subgroup of $G_N$ is nilpotent by the remark we just made). Since $G/G_0$ is virtually polycyclic, every subgroup of it is finitely generated (\cite{Rag}[4.2]). Hence it is enough to prove that $G_N \cap G_0$ is topologically finitely generated. This follows from Lemma \ref{topfinite}.
\end{proof}

Incidently, we observe that the connected component of the identity $(G_N)_0$ coincides with the nilradical $N$ of $G_0$ (it is the maximal normal nilpotent connected subgroup of $G_0$).

We now claim the following:

\begin{lemma}\label{malcev} The quotient group $G/G_N$ is virtually abelian.
\end{lemma}

The proof of this lemma is inspired by the proof of the fact, due to Malcev, that polycyclic groups have a finite index subgroup with nilpotent commutator subgroup (e.g. see \cite{robinson}[ 15.1.6]).

\begin{proof} We will show that $G$ has a finite index normal subgroup whose commutator subgroup is nilpotent. This clearly implies the lemma, for this nilpotent subgroup will be normal, hence contained in $G_N$. 

First we observe that the group $G$ admits a finite normal series $G_m \leq G_{m-1} \leq \ldots \leq  G_1=G$, where each $G_i$ is a closed normal subgroup of $G$ such that $G_i/G_{i+1}$ is either finite, or isomorphic to either $\Z^n$, $\R^n$ or $\R^n/\Z^n$. This see it pick one of the $G_i$'s to be the connected component $G_0$ and then treat $G/G_0$ and $G_0$ separately. The first follows from the definition of a polycyclic group ($G/G_0$ has a normal polycyclic subgroup of finite index). While for $G_0$, observe that its nilradical $N$ is a connected and simply connected nilpotent Lie group and it admits such a series of characteristic subgroups (pick the central descending series), and $G_0/N$ is an abelian connected Lie group, hence isomorphic to the direct product of a torus $\R^n/\Z^n$ and a vector group $\R^n$. The torus part is characteristic in $G_0/N$, hence its preimage in $G_0$ is normal in $G$.

The group $G$ acts by conjugation on each partial quotient $Q_i:=G_i/G_{i+1}$. This yields a map $G \to Aut(Q_i)$. Now note that in order to prove our lemma, it is enough to show that for each $i$, there is a finite index subgroup of $G$ whose commutator subgroup maps to a nilpotent subgroup of $Aut(Q_i)$. Indeed, taking the intersection of those finite index subgroup, we get a finite index normal subgroups whose commutator subgroup acts nilpotently on each $Q_i$, hence is itself nilpotent (high enough commutators will all vanish).

Now $Aut(Q_i)$ is either finite (if $Q_i$ is finite), or isomorphic to $GL_n(\Z)$ (in case $Q_i$ is either $\Z^n$ or $\R^n/\Z^n$) or to $GL_n(\R)$ (when $Q_i \simeq \R^n$). The image of $G$ in $Aut(Q_i)$ is a solvable subgroup. However, every solvable subgroup of $GL_n(\R)$ contains a finite index subgroup, whose commutator subgroup is unipotent (hence nilpotent). This follows from Kolchin's theorem for example, that a connected solvable algebraic subgroup of $GL_n(\C)$ is triangularizable. We are done.
\end{proof}

In the sequel we assume that $G/G_0$ is torsion-free polycyclic. It is legitimate to do so in the proof of Proposition \ref{empol}, because every virtually polycyclic group has a torsion-free polycyclic subgroup of finite index (see e.g. \cite{Rag}[Lemma 4.6]).

We now claim the following:

\begin{lemma}\label{tfree} $G_N$ is torsion-free.
\end{lemma}

\begin{proof} Since $G/G_0$ is torsion-free, it is enough to prove that $G_N \cap G_0$ is torsion-free. However the set of torsion elements in $G_N$ forms a subgroup of $G_N$ (if $x,y$ are torsion, then $xy$ is too because $\langle x,y \rangle$ is nilpotent). Clearly it is a characteristic subgroup of $G_N$. Hence its intersection with $G_0$ is normal in $G_0$. Taking the closure, we obtain a nilpotent closed normal subgroup $T$ of $G_0$ which contains a dense set of torsion elements. Recall that $G_0$ has no normal compact subgroup. From this it quickly follows that $T$ is trivial, because first it must be discrete (the connected component $T_0$ is compact and normal in $G_0$), hence finitely generated (by Lemma \ref{topfinite}), hence made of torsion elements. But a finitely generated torsion nilpotent group is finite. Again since $G_0$ has no compact normal subgroup, $T$ must be trivial, and $G_N$ is torsion-free.
\end{proof}

Now observe that the group of connected components of $G_N$, namely $G_N/(G_N)_0$ is finitely generated. Indeed, since $G/G_0$ is finitely generated (as any polycyclic group), it is enough to prove that $(G_0 \cap G_N)/(G_N)_0$ is finitely generated, but this follows from the fact that $G_0 \cap G_N$ is topologically finitely generated (Lemma \ref{topfinite}).

Now we are almost done. Note that $G$ is topologically finitely generated (Lemma \ref{topfinite}), therefore so is $G/G_N$. By Lemma \ref{malcev} $G/G_N$ is virtually abelian, hence has a finite index normal subgroup isomorphic to $\Z^n \times \R^m$. It follows that $G/G_N$ has a co-compact subgroup isomorphic to a free abelian group $\Z^{n+m}$. Hence after changing $G$ by a co-compact subgroup, we get that $G$ is an extension of $G_N$ (a torsion-free nilpotent Lie group with finitely generated group of connected components) by a finitely generated free abelian group. Hence it is an $\mathcal{S}$-group in the terminology of Wang \cite{Wan}. We apply Wang's theorem and this ends the proof of Proposition \ref{empol}.\\

(d) We can now conclude the proof of Theorem \ref{general}. By (a) and (b) $G$ has a quotient by a compact group which admits a co-compact subgroup satisfying the assumptions of Proposition \ref{empol}. Hence to conclude the proof it only remains to verify that the simply connected solvable Lie group in which a co-compact subgroup of $G/K$ embeds has polynomial growth (i.e. is of type $(R)$). But this follows from the following lemma (see \cite{Gui}[Thm. I.2]).

 \begin{lemma}\label{cocomp} Let $G$ be a locally compact group. Then $G$ has polynomial growth if and only if some (resp. any) co-compact subgroup of it has polynomial growth.
 \end{lemma}

 \begin{proof}First one checks that $G$ is compactly generated if and only if some (resp. any) co-compact subgroup is. This is by the same argument which shows that finite index subgroups of a finitely generated group are finitely generated. In particular, if $\Omega$ is a compact symmetric generating set of $G$ and $H$ is a co-compact subgroup, then there is $n_0 \in \N$ such that $\Omega^{n_0}H=G$. Then $H \cap \Omega^{3n_0}$ generates $H$.
 
If $G$ has polynomial growth and $H$ is any compactly generated closed subgroup, then $H$ has polynomial growth. Indeed (see \cite{Gui}[Thm I.2]), if $\Omega_H$ denotes a compact generating set for $H$, and $K$ a compact neighborhood of the identity in $G$, then $$vol_G(K) vol_H(\Omega_H^n) \leq vol_H(KK^{-1} \cap H) vol_G(\Omega_H^nK).$$
 This inequality follows by integrating over a left Haar measure of $G$ the function $\phi(x):=\int_{\Omega_H^n} 1_{K}(h^{-1}x)dh$, where $dh$ is a left Haar measure on $H$. This integral equals the left handside of the above displayed equation, while it is pointwise bounded by $vol_H(xK^{-1} \cap H)$ inside $HK$ and by zero outside $HK$.
 
 In the other direction, if $H$ has polynomial growth, then $G$ also has, because one can write $\Omega^n \subset \Omega_H^nK$ for some compact generating set $\Omega_H$ of $H$ and some compact neighborhood $K$ of the identity in $G$ (see Proposition \ref{roughcomp}). Then the result follows from the following inequality
  $$vol_H(\Omega_H) vol_G(\Omega_H^n K) \leq vol_H(\Omega_H^{n+1})vol_G(\Omega_H^{-1}K),$$
 which itself is a direct consequence of the fact that the function $$\psi(x):=\int_{\Omega_H^{n+1}} 1_{\Omega_H^{-1}K}(h^{-1}x)dh,$$ where $dh$ is a left Haar measure on $H$, satisfies $\int_G \psi(x) dx = vol_H(\Omega_H^{n+1})vol_G(\Omega_H^{-1}K)$ on the one hand and is bounded below by $vol_H(\Omega_H)$ for every $x \in \Omega_H^nK$ on the other hand.
 \end{proof}
 
 Note that the above proof would be slightly easier if we already knew that both $G$ and $H$ were unimodular, in which case $G/H$ has an invariant measure. But we know this only a posteriori, because the polynomial growth condition implies unimodularity (\cite{Gui}).
 
 Similar considerations show that $G$ has polynomial growth if and only if $G/K$ has polynomial growth, given any normal compact subgroup $K$ (e.g. see \cite{Gui}).

\endproof%

We end this paragraph with a remark and an example, which we mentioned in the Introduction.

\begin{remark}[Discrete subgroups are virtually nilpotent]\label{polattice} Suppose $\Gamma$ is a discrete subgroup of a connected solvable Lie group of type $(R)$ (i.e. of polynomial growth). Then $\Gamma$ is virtually nilpotent. Indeed, a similar argument as in Lemma \ref{topfinite} shows that every subgroup of $\Gamma$ is finitely generated. It follows that $\Gamma$ is polycyclic. However Wolf \cite{wolf} proved that polycyclic groups with polynomial growth are virtually nilpotent.
\end{remark}

\begin{example}[A group with no nilpotent co-compact subgroup]\label{exx} Let $G$ be the connected solvable Lie group $G=\R \ltimes (\R^2 \times \R^2)$, where $\R$ acts as a dense one-parameter subgroup of $SO(2,\R) \times SO(2,\R)$. Then $G$ is of type $(R)$. It has no compact subgroup. And it has no nilpotent co-compact subgroup. Indeed suppose $H$ is a closed co-compact nilpotent subgroup. Then it has a non-trivial center. Hence there is a non identity element whose centralizer is co-compact in $G$. However a simple examination of the possible centralizers of elements of $G$ shows that none of them is co-compact.
\end{example}

\subsection{Proof of Corollary \ref{volasym} and Theorem \ref{firsthm}.}

Let $G$ be an arbitrary locally compact group of polynomial growth and $\rho
$ a periodic pseudodistance on $G.$\\

\textbf{Claim 1:} \textit{Corollary \ref{volasym} holds for a co-compact
subgroup }$H$\textit{\ of }$G$\textit{, if and only if it holds for }$G$. 
By Lemma \ref{cocomp}, the groups $G$ and $H$ are unimodular, and hence $G/H$ bears a $G$-invariant Radon measure  $vol_{G/H}$, which is finite since $H$ is co-compact. Now let $F$ be a bounded Borel fundamental domain for $H$ inside $G.$
And let $\overline{\rho }$ be the periodic pseudodistance on $G$ induced by the
restriction of $\rho $ to $H,$ that is $\overline{\rho }(x,y):=\rho
(h_{x},h_{y})$ where $h_{x}$ is the unique element of $H$ such that $x\in
h_{x}F.$ By \ref{properness} $(1)$ and $(4)$, $\rho $ and $\overline{\rho }$
are at a bounded distance from each other. In particular, $B_{\overline{\rho
}}(r-C)\subset B_{\rho }(r)\subset B_{\overline{\rho }}(r+C)$. Hence if the
limit (\ref{limitvol}) holds for $\overline{\rho },$ it also holds for $\rho
$ with the same limit. However, $B_{\overline{\rho }}(r)=\{x\in G,\rho
(e,h_{x})\leq r\}=B_{\rho _{H}}(r)F$ where $\rho _{H}$ is the restriction of
$\rho $ to $H.$ Hence $vol_{G}(B_{\overline{\rho }}(r))=vol_{H}(B_{\rho
_{H}}(r))\cdot vol_{G/H}(F).$ By \ref{properness} $(4)$, $\rho _{H}$ is a
periodic pseudodistance on $H.$ So the result holds for $(H,\rho _{H})$ if and only
if it holds for $(G,\rho )$. Conversely, if $\rho _{0}$ is a periodic pseudodistance
on $H,$ then $\overline{\rho _{0}}(x,y):=\rho _{0}(h_{x},h_{y})$ is a
periodic pseudodistance on $G,$ hence again $vol_{G}(B_{\overline{\rho }%
_{0}}(r))=vol_{H}(B_{\rho _{0}}(r))\cdot vol_{G}(F)$ and the result will
hold for $(H,\rho _{0})$ if and only if it holds for $(G,\overline{\rho _{0}}%
).$\\

\textbf{Claim 2:} \textit{If Corollary \ref{volasym} holds for }$G/K$\textit{%
, where }$K$\textit{\ is some compact normal subgroup,} \textit{then it
holds for }$G$\textit{\ as well}. Indeed, if $\rho $ is a periodic
pseudodistance on $G,$ then the $K$-average $\rho^{K}$, as defined in (\ref
{average}), is at a bounded distance from $G$ according to Lemma \ref
{bdedistance}. Now $\rho^{K}$ induces a periodic pseudodistance $\overline{
\rho^{K}}$ on $G/K$ and $B_{\rho^{K}}(r)=B_{\overline{\rho^{K}}}(r)K.$
Hence, $vol_{G}(B_{\rho^{K}}(r))=vol_{G/K}(B_{\overline{{\rho}^{K}}}(r))\cdot vol_{K}(K).$ And if
the
limit (\ref{limitvol}) holds for $%
\overline{{\rho}^{K}},$ it also holds for $\rho^{K}$, hence for $\rho $ too.\\

Thus the discussion above combined with Theorem \ref{general} reduces
Corollary \ref{volasym} to the case when $G$ is simply connected and
solvable, which was treated in Section 5 and 6.
\endproof

\subsection{Proof of Proposition \ref{bdeddistance} and Corollary \ref{asycones}}

\begin{proof}[Proof of Proposition \ref{bdeddistance}] We say that two metric spaces $(X,d_X)$ and $(Y,d_Y)$ are at a bounded distance if they are $(1,C)$-quasi-isometric for some finite $C$. This is an equivalence relation. Now if $\rho$ is $H$-periodic with $H$ co-compact, then $(G,\rho)$ is at a bounded distance from $(H,\rho{|H})$. Hence we may assume that $H=G$, i.e. that $\rho$ is left invariant on $G$. 

Now Theorem \ref{general} gives the existence of a normal compact subgroup $K$, a co-compact subgroup $H$ containing $K$ and a simply connected solvable Lie group $S$ such that $H/K$ is isomorphic to a co-compact subgroup of $S$.

Lemma \ref{bdedistance} shows that $(G,\rho)$ is at a bounded distance from $(G,\rho^K)$, where $\rho^K$ is defined as in $(\ref{average})$. Now $\rho^K$ induces a left invariant periodic metric on $G/K$, and $(G/K,\rho^K)$ is clearly at a bounded distance from $(G,\rho^K)$. Now by \ref{restrict}, its restriction to $H/K$ is at a bounded distance and is left invariant. Now we set $\rho_S(s_1,s_2)=\rho^K(h_1,h_2)$, where (given a bounded fundamental domain $F$ for the left action of $H/K$ on $S$) $h_i$ is the unique element of $H/K$ such that $s_i \in h_iF$. Clearly then $(S,\rho_S)$ is at a bounded distance from $(H/K,\rho^K)$. We are done.
\end{proof}

We note that our construction of $S$ here depends on the stabilizer of $\rho$ in $G$. Certainly not every choice of Lie shadow can be used for all periodic metrics (think that $\R^3$ is a Lie shadow of the universal cover of the group of motions of the plane). Perhaps a single one can be chosen for all, but we have not checked that.

\begin{proof}[Proof of Corollary \ref{asycones}]
Proposition \ref{bdeddistance} reduces the proof to a periodic metric  $\rho$ on  a simply connected solvable Lie group $S$. Let $d_\infty$ the subFinsler metric on $S$ (left invariant for the graded nilshadow group structure $S_N$) as given by Theorem \ref{main-theorem}. Let $\{\delta_t\}_t$ is the group of dilations in the graded nilshadow $S_N$ of $S$ as defined in Section \ref{nilshadow}. By definition of the pointed Gromov-Hausdorff topology (see \cite{gromov-pansu-lafontaine}), it is enough to prove the
 
 \noindent { \bf Claim.} The following quantity $$|\frac{1}{n}\rho(s_1,s_2) - d_\infty(\delta_{\frac{1}{n}}(s_1),\delta_{\frac{1}{n}}(s_2))|$$ converges to zero as $n$ tends to $+\infty$ uniformly for all $s_1,s_2$ in a ball of radius $O(n)$ for the metric $\rho$. 

Now this follows in three steps. First $\rho$ is at a bounded distance from its restriction to the (co-compact) stabilizer $H$ of $\rho$ (cf. \ref{restrict0} (1), \ref{restrict} (4)). Then for $h_1,h_2 \in H$, we can write $\rho(h_1,h_2)=\rho(e,h_1^{-1}h_2)$. However Proposition \ref{pop} implies the existence of another periodic distance $\rho_K$ on $S$, which is invariant under left translations by elements of $H$ for both the original Lie structure and the nilshadow Lie structure on $S$, such that $\frac{\rho(e,x)}{\rho_K(e,x)}$ tends to $1$ as $x$ tends to $\infty$. Hence  $\rho_K(e,h_1^{-1}h_2)=\rho_K(h_1,h_2)=\rho_K(e,h_1^{*-1}h_2)$, where $*$ is the nilshadow product on $S$. Hence $|\frac{1}{n}\rho(h_1,h_2) - \frac{1}{n}\rho_K(e,h_1^{*-1}h_2)|$ tends to zero uniformly as $h_1$ and $h_2$ vary in a ball of radius $O(n)$ for $\rho$.

Finally Theorem \ref{Pansu} implies that $|\frac{1}{n}\rho_K(e,h_1^{*-1}h_2)- \frac{1}{n}d_\infty(e,h_1^{*-1}h_2)|$ tends to zero and the claim follows, as one verifies from the Campbell Hausdorff formula by comparing $(\ref{CHF0})$ and $(\ref{CHF})$ as we did in $(\ref{plug})$, that $$|d_\infty( \delta_{\frac{1}{n}}(h_1),\delta_{\frac{1}{n}}(h_2)) - d_\infty(e,\delta_{\frac{1}{n}}(h_1^{*-1}h_2)|$$ converges to zero.

The fact that the graded nilpotent Lie group does not depend (up to isomorphism) on the periodic metric $\rho$ but only on the locally compact group $G$ follows from Pansu's theorem \cite{Pan2} that if two Carnot groups (i.e. a graded simply connected nilpotent Lie group endowed with left-invariant subRiemannian metric induced by a norm on a supplementary subspace to the commutator subalgebra) are bi-Lipschitz, the underlying Lie groups must be isomorphic. This deep fact relies on Pansu's generalized Rademacher theorem, see \cite{Pan2}. Indeed, two different periodic metrics $\rho_1$ and $\rho_2$ on $G$ are quasi-isometric (see Proposition \ref{roughcomp}), and hence their asymptotic cones are bi-Lipschitz (and bi-Lipschitz to any Carnot group metric on the same graded group, by $(\ref{equivalent})$).
\end{proof}

\section{Coarsely geodesic distances and speed of convergence\label{speed}}

Under no further assumption on the periodic pseudodistance $\rho ,$ the speed of
convergence in the volume asymptotics can be made arbitrarily small. This is
easily seen if we consider examples of the following type: define $\rho
(x,y)=|x-y|+|x-y|^{\alpha }$ on $\Bbb{R}$ where $\alpha \in (0,1)$. It is
periodic and $vol(B_{\rho }(t))=t-t^{\alpha }+o(t^{\alpha }).$

However, many natural examples of periodic metrics, such as word metrics or
Riemannian metrics, are in fact coarsely geodesic. A pseudodistance on $G$
is said to be \textit{coarsely geodesic}, if there is a constant $C>0$ such
that any two points can be connected by a $C$-coarse geodesic, that is, for
any $x,y\in G$ there is a map $g:[0,t]\rightarrow G$ with $t=\rho (x,y),$ $%
g(0)=x$ and $g(t)=y$, such that
\begin{equation*}
\left| \rho (g(u),g(v))-|u-v|\right| \leq C
\end{equation*}
for all $u,v\in [0,t]$.

This is a stronger requirement than to say that $\rho $ is asymptotically
geodesic (see \ref{asymgeo}). This notion is invariant under coarse
isometry. In the case when $G$ is abelian, D. Burago \cite
{Bur} proved the beautiful fact that any coarsely geodesic periodic metric on $G$ is at a bounded
distance from its asymptotic norm. In particular $vol_{G}(B_{\rho
}(t))=c\cdot t^{d}+O(t^{d-1})$ in this case. In the remarkable paper \cite{Sto}, M. Stoll
 proved that such an error term in $O(t^{d-1})$ holds for any finitely generated $2$-step
nilpotent group. Whether $O(t^{d-1})$ is the right error term for any finitely generated
nilpotent group remains an open question.

The example below shows on the contrary that in an arbitrary Lie
group of polynomial growth no universal error term can be expected.

\begin{theorem}
\label{SmallSpeed}Let $\varepsilon _{n}>0$ be an arbitrary sequence of
positive numbers tending to $0.$ Then there exists a group $G$ of polynomial
growth of degree $3$ and a compact generating set $\Omega $ in $G$ and $c>0$
such that
\begin{equation}
\frac{vol_{G}(\Omega ^{n})}{c\cdot n^{3}}\leq 1-\varepsilon _{n}
\label{badspeed}
\end{equation}
holds for infinitely many $n$, although $\frac{1}{c\cdot n^{3}}%
vol_{G}(\Omega ^{n})\rightarrow 1$ as $n\rightarrow +\infty .$
\end{theorem}

The example we give below is a semi-direct product of $\Bbb{Z}$ by $\Bbb{R}%
^{2}$ and the metric is a word metric. However, many similar examples can be
constructed as soon as the map $T:G\rightarrow K$ defined in Paragraph \ref
{asyminv} in not onto. For example, one can consider left invariant
Riemannian metrics on $G=\Bbb{R}\cdot (\Bbb{R}^{2}\times \Bbb{R}^{2})$ where
$\Bbb{R}$ acts by via a dense one-parameter subgroup of the $2$-torus $%
S^{1}\times S^{1}.$ Incidently, this group $G$ is known as the \textit{Mautner group}
and is an example of a \textit{wild} group in representation theory.

\subsection{An example with arbitrarily small speed\label{cones}}

In this paragraph we describe the example of Theorem \ref{SmallSpeed}. Let $%
G_{\alpha }=\Bbb{Z}\cdot \Bbb{R}^{2}$ where the action of $\Bbb{Z}$ is given
by the rotation $R_{\alpha }$ of angle $\pi \alpha ,$ $\alpha \in [0,1).$
The group $G_{\alpha }$ is quasi-isometric to $\Bbb{R}^{3}$ and hence of
polynomial growth of order $3$ and it is co-compact in the analogously
defined Lie group $\widetilde{G_{\alpha }}=\Bbb{R} \ltimes \Bbb{R}^{2}.$ Its
nilshadow is isomorphic to $\Bbb{R}^{3}.$ The point is that if $\alpha $ is
a suitably chosen Liouville number, then the balls in $G_{\alpha }$ will not
be well approximated by the limit norm balls.

Elements of $G_{\alpha }$ are written $(k,x)$ where $k\in \Bbb{Z}$ and $x\in
\Bbb{R}^{2}.$ Let $\left\| x\right\| ^{2}=\frac{1}{4}x_{1}^{2}+x_{2}^{2}$ be
a Euclidean norm on $\Bbb{R}^{2}$, and let $\Omega $ be the symmetric
compact generating set given by $\{(\pm 1,0)\}\cup \{(0,x),\left\| x\right\|
\leq 1\}.$ It induces a word metric $\rho _{\Omega }$ on $G$. It follows
from Theorem \ref{MetComp} and the definition of the asymptotic norm that $%
\rho _{\Omega }(e,(k,x))$ is asymptotic to the norm on $\Bbb{R}^{3}$ given
by $\rho _{0}(e,(k,x)):=|k|+\left\| x\right\| _{0}$ where $\left\| x\right\|
_{0}$ is the rotation invariant norm on $\Bbb{R}^{2}$ defined by $\left\|
x\right\| _{0}^{2}=\frac{1}{4}(x_{1}^{2}+x_{2}^{2}).$ The unit ball of $%
\left\| \cdot \right\| _{0}$ is the convex hull of the union of all images
of the unit ball of $\left\| \cdot \right\| $ under all rotations $%
R_{k\alpha },$ $k\in \Bbb{Z}.$

\begin{figure}\label{fig2}
\begin{center}
\setlength{\belowcaptionskip}{0pt}
\includegraphics[scale=.4]{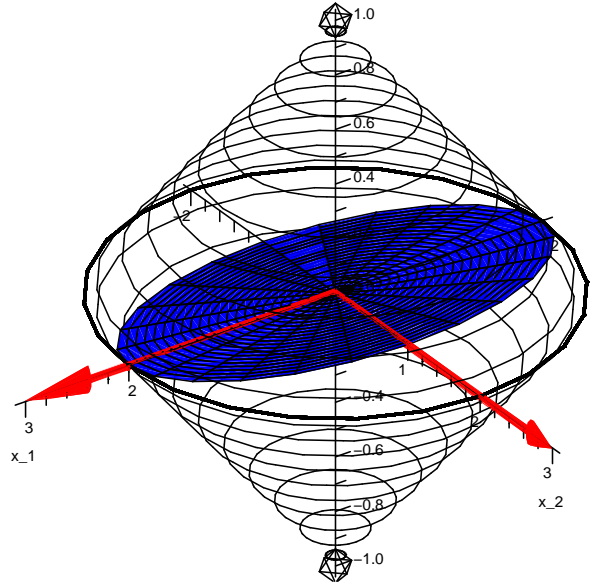}
\caption{The union of the two cones,
with basis the disc of radius $2$, represents the limit shape of the balls $
\Omega ^{n}$ in the group $\Z \ltimes \R^{2}$, where $\Bbb{Z}$ acts by an
irrational rotation, with generating set $\Omega =\{(\pm 1,0,0)\}\cup
\{(0,x_{1},x_{2}),\frac{1}{4}x_{1}^{2}+x_{2}^{2}\leq 1\}$.}
\end{center}
\end{figure}

We are going to choose $\alpha $ as a suitable Liouville number so that (\ref
{badspeed}) holds. Let $\delta _{n}=(4\varepsilon _{n})^{1/3}$ and choose $%
\alpha $ so that the following holds for infinitely many $n$'s:
\begin{equation}
d(k\alpha ,\Bbb{Z}+\frac{1}{2})\geq 2\delta _{n}  \label{Liouville}
\end{equation}
for all $k\in \Bbb{Z},$ $|k|\leq n.$ This is easily seen to be possible if
we choose $\alpha $ of the form $\sum 1/3^{n_{i}}$ for some suitable
lacunary increasing sequence of $(n_{i})_{i}.$

Note that, since $\left\| x\right\| _{0}\geq \left\| x\right\| ,$ we have $%
\rho _{\Omega }\geq \rho _{0}.$ Let $S_{n}$ be the piece of $\Bbb{R}^{2}$
defined by $S_{n}=\{|\theta |\leq \delta _{n}\}$ where $\theta $ is the
angle between the point $x$ and the vertical axis $\Bbb{R}e_{2}.$ We \textit{%
claim} that if $x\in S_{n}$, $\rho _{0}(e,(k,x))\leq n$ and $n$ satisfies (%
\ref{Liouville}), then
\begin{equation*}
\rho _{\Omega }(e,(k,x))\geq |k|+(1+\frac{\delta _{n}^{2}}{4})\left\|
x\right\| _{0}
\end{equation*}
It follows easily from the claim that $vol_{G}(\Omega ^{n})\leq
(1-\varepsilon _{n})\cdot vol_{G}(B_{\rho _{0}}(n)).$ Moreover $%
vol_{G}(B_{\rho _{0}}(n))=c\cdot n^{3}+O(n^{2}),$ where $c=\frac{4\pi }{3}$
if $vol_{G}$ is given by the Lebesgue measure.

\textit{Proof of claim.} Here is the idea to prove the claim. To find a
short path between the identity and a point on the vertical axis, we have to
rotate by a $R_{k\alpha }$ such that $k\alpha $ is close to $\frac{1}{2},$
hence go up from $(0,0)$ to $(k,0)$ first, thus making the vertical
direction shorter. However if (\ref{Liouville}) holds, the vertical
direction cannot be made as short as it could after rotation by any of the $%
R_{k\alpha }$ with $|k|\leq n.$

Note that if $\rho _{0}(e,(k,x))\leq n$ then $|k|\leq n$ and $\rho _{\Omega
}(e,(k,x))\geq |k|+\inf \sum \left\| R_{k_{i}\alpha }x_{i}\right\| $ where
the infimum is taken over all paths $x_{1},...,x_{N}$ such that $x=\sum x_{i}
$ and all rotations $R_{k_{i}\alpha }$ with $|k_{i}|\leq n.$ Note that if $%
\delta _{n}$ is small enough and (\ref{Liouville}) holds then for every $%
x\in S_{n}$ we have $\left\| R_{k\alpha }x\right\| \geq (1+\delta
_{n}^{2})\left\| x\right\| _{0}.$ On the other hand $\left\| x\right\|
_{0}=\sum \left\| x_{i}\right\| _{0}\cos (\theta _{i})$ where $\theta _{i}$
is the angle between $x_{i}$ and the $x$. Hence
\begin{eqnarray*}
\sum \left\| R_{k_{i}\alpha }x_{i}\right\|  &\geq &\sum_{|\theta _{i}|\leq
\delta _{n}}\left\| R_{k_{i}\alpha }x_{i}\right\| +\sum_{|\theta
_{i}|>\delta _{n}}\left\| R_{k_{i}\alpha }x_{i}\right\|  \\
&\geq &(1+\delta _{n}^{2})\sum_{|\theta _{i}|\leq \delta _{n}}\left\|
x_{i}\right\| _{0}\cos (\theta _{i})+\frac{1}{\cos (\delta _{n})}%
\sum_{|\theta _{i}|>\delta _{n}}\left\| x_{i}\right\| _{0}\cos (\theta _{i})
\\
&\geq &(1+\frac{\delta _{n}^{2}}{4})\cdot \left\| x\right\| _{0}
\end{eqnarray*}
\edpf


\subsection{Limit shape for more general word metrics on solvable Lie groups of polynomial growth}\label{solshape}
The determination of the limit shape of the word metric in Paragraph \ref{cones} was possible due to the rather simple nature of the generating set. In general, using the identity (see $(\ref{nilproduct})$)
\begin{equation}\label{iterated}\omega_1\cdot \ldots \cdot \omega_m=\omega_1*(T(\omega_1)\omega_2)*\ldots*(T(\omega_{m-1}\cdot \ldots \cdot \omega_1)\omega_m)
\end{equation}
it is easy to check that the unit ball of the limit norm $\|\cdot\|_\infty$ inducing the limit subFinsler metric $d_\infty$ on the nilshadow associated to a given word metric with generating set $\Omega$ is contained in the $K$-orbit of the convex hull of the projection of $\Omega$ to the abelianized nilshadow, namely the convex hull of $K\cdot \pi_1(\Omega)$.

In the example of Paragraph \ref{cones}, we even had equality between the two. However this is not the case in general. For example, the limit shape is always $K$-invariant, but clearly the limit shape associated to a generating set $\Omega$ coincides with the one associated with a conjugate $g\Omega g^{-1}$ of it, while the convex hull of the respective $K$-orbits may not be the same.

Of course if the generating set $\Omega$ is $K$-invariant to begin with, then $\Omega^n=\Omega^{*n}$ and we are back in the nilpotent case, where we know that the unit ball of the limit norm is just the convex hull of the projection of the generating set to the abelianization. In general however it is a challenging problem to determine the precise asymptotic shape of a word metric on a general solvable Lie group with polynomial growth, and there seems to be no simple description analogous to what we have in the nilpotent case.

Even in the above example $G_\alpha=\Z \ltimes_\alpha \R^2$, or in the universal cover of the group of motions of the plane (in which $G_\alpha$ embeds co-compactly), it is not that simple. In general the shape is determined by solving an optimization problem in which one has to find the path which maximizes the coordinates of the endpoint. In order to illustrate this, we treat without proof the following simple example.

Suppose $\Omega$ is a symmetric compact neighborhood of the identity in $G_\alpha=\Z \ltimes_\alpha \R^2$ of the form $\Omega=(0,\Omega_0) \cup (1,\Omega_1) \cup (1,\Omega_1)^{-1}$, where $\Omega_0,\Omega_1 \subset \R^2$. Then the limit shape of the word metric $\rho_\Omega$ associated to $\Omega$ is the solid body (rotationally symmetric around the vertical axis as in Figure 2) made of two copies (upper and lower) of a truncated cone  with base a disc on $(0,\R^2)$ of radius $\max\{r_0,r_1\}$ and top (resp. bottom) a disc on the plane $(1,\R^2)$ (resp. $(-1,\R^2)$) of radius $r_2$, where the radii are given by
$$r_0=\max\{\|x\|,x \in \Omega_0\}, r_1=\frac{1}{2}diam(\Omega_1),$$
where $diam(\Omega_1)$ is the diameter of $\Omega_1$ and $r_2$ is given by the integral
\begin{equation}\label{top-part}
r_2=\int_0^{2\pi} \max \{\pi_\theta(\Omega_1)\} \frac{d\theta}{2\pi},
\end{equation}
where $\pi_\theta(\Omega_1)$ is the orthogonal projection on the $x$-axis of image of $\Omega_1 \subset \R^2$ by a rotation of angle $\theta$ around the origin. It is indeed convex (note that $r_2\leq r_1$).


For example if $\Omega_1$ is made of only one point, then the limit shape is the same as in the previous paragraph and as in Figure 2, namely two copies of a cone. However if $\Omega_1$ is made of two points $\{a,b\}$, then the upper part of the limit shape will be a truncated cone with an upper disc of radius $r_2=\frac{\|a-b\|}{\pi}$ (which is the result of the computation of the above integral).

Let us briefly explain the formula $(\ref{top-part})$. A path of length $n$ reaching the highest $z$-coordinate in $G_\alpha$ is a word of the form $(1,\omega_1)\cdot \ldots \cdot (1,\omega_n)$, with $\omega_i \in \Omega_1$. By  $(\ref{iterated})$ this word equals
$$(n,\sum_1^n R_\alpha^{i-1}\omega_i).$$
Here $\omega_i$ can take any value in $\Omega_1$. In order to maximize the norm of the second coordinate, or equivalently (by rotation invariance) its $x$-coordinate, one has to choose $\omega_i \in \Omega_1$ at each stage in such a way that the $x$-coordinate of $R_\alpha^{i-1}\omega_i$ is maximized. Formula $(\ref{top-part})$ now follows from the fact that $\{R_\alpha^{i-1}\}_{1\leq i \leq n}$ becomes equidistributed in $SO(2,\R)$ as $n$ tends to infinity.

In order to show that $\max\{r_0,r_1\}$ is the radius of the base disc and more generally that the limit shape is no bigger than this double truncated cone, one needs to argue further by considering all possible paths of the form $(\eps_1,\omega_1)\cdot \ldots \cdot (\eps_n,\omega_n)$ where $\eps_i \in \{0, \pm 1\}$ and $\sum \eps_i$ is prescribed.

\subsection{Bounded distance versus asymptotic metrics}\label{burbur}

In this paragraph we answer a question of D. Burago and G. Margulis (see \cite{Bur2}). Based on the abelian case and the reductive case (Abels-Margulis \cite{abels-margulis}), Burago and Margulis had conjectured that every two asymptotic word metrics should be at a bounded distance. We give below a counterexample to this. We first give an example ($A$) of a
 nilpotent Lie group endowed with two left invariant subFinsler metrics $d_\infty$ and $d'_\infty$ that are asymptotic to each other, i.e. $d_\infty(e,x)/d'_\infty(e,x)\rightarrow 1$ as $x\rightarrow \infty $ but such that $|d_\infty(e,x)-d'_\infty(e,x)|$ is not uniformly bounded. Then we exhibit ($B$) a word metric that is
not at a bounded distance from any homogeneous quasi-norm. Finally these examples also yield ($C$) two word metrics $\rho_1$ and $\rho_2$ on the same finitely generated nilpotent group which are asymptotic but not at a bounded distance.

Note that the group $G_{\alpha }$ with $\rho _{0}$ and $\rho _{\Omega }$ from the last paragraph also
provides an example of asymptotic metrics which are not at a bounded distance (but this group was not
discrete).

$(A)$ Let $N=\Bbb{R}\times H_{3}(\Bbb{R})$ where $H_{3}$ is classical
Heisenberg group and $\Gamma =\Bbb{Z}\times H_{3}(\Bbb{Z})$ a lattice in $N$%
. In the Lie algebra $\frak{n}=\Bbb{R}V\oplus \frak{h}_{3}$ we pick two
different supplementary subspaces of $[\frak{n},\frak{n}]=\Bbb{R}Z,$ i.e. $%
m_{1}=span\{V,X,Y\}$ and $m_{1}^{\prime }=span\{V+Z,X,Y\}$, where $\frak{h}%
_{3}$ is the Lie algebra of $H_{3}(\Bbb{R})$ spanned by $X,Y$ and $Z=[X,Y].$
We consider the $L^{1}$-norm on $m_{1}$ (resp. $m_{1}^{\prime }$)
corresponding to the basis $(V,X,Y)$ (resp. $(V+Z,X,Y)$). Both norms induce
the same norm on $\frak{n}/[\frak{n},\frak{n}].$ They give rise to left
invariant Carnot-Caratheodory Finsler metrics on $N$, say $d_{\infty }$
(resp. $d_{\infty }^{\prime }$). We use the coordinates $(v,x,y,z)=\exp
(vV+xX+yY+zZ)$.

According to Remark $(2)$ after Theorem \ref{Pansu}, $d_{\infty }$ and $%
d_{\infty }^{\prime }$ are asymptotic. Let us show that they are not at a
bounded distance. First observe that, since $V$ is central, $d_{\infty
}(e,(v;(x,y,z)))=|v|+d_{H_{3}}(e,(x,y,z))$ where $d_{H_{3}}$ is the
Carnot-Caratheodory Finsler metric on $H_{3}(\Bbb{R})$ defined by the
standard $L^{1}$-norm on the $span\{X,Y\}.$ Similarly $d_{\infty }^{\prime
}(e,(v;(x,y,z)))=|v|+d_{H_{3}}(e,(x,y,z-v))).$ If $d_{\infty }$ and $d_{\infty
}^{\prime }$ were at a bounded distance, we would have a $C>0$ such that for
all $t>0$
\begin{equation*}
|d_{\infty }(e,(t;(0,0,t)))-t|\leq C
\end{equation*}
Hence $|d_{H_{3}}(e,(0,0,t))|\leq C,$ which is a contradiction.\\


$(B)$ Now let $\Omega =\{(1;(0,0,1))^{\pm 1},(1;(0,0,-1))^{\pm 1},(0;(1,0,0))^{\pm
1},(0;(0,1,0))^{\pm 1}\}$ be a generating set for $\Gamma $ and $\rho _{\Omega
}$ the word metric associated to it. Let $|\cdot |$ be a homogeneous
quasi-norm on $N$ which is at a bounded distance from $\rho _{\Omega },$
i.e. $|\rho _{\Omega }(e,g)-|g||$ is bounded. Then $|\cdot |$ is asymptotic
to $\rho _{\Omega },$ hence is equal to the Carnot-Caratheodory Finsler
metric $d$ asymptotic to $\rho _{\Omega }$ and homogeneous with respect to
the same one parameter group of dilations $\{\delta _{t}\}_{t>0}.$ Let $%
m_{1}=\{v\in \frak{n}$, $\delta _{t}(v)=tv\}.$ Then $d$ is induced by some
norm $\left\| \cdot \right\| _{0}$ on $m_{1},$ whose unit ball is given,
according to Theorem \ref{MetComp} by the convex hull of the projections to $%
m_{1}$ of the generators in $\Omega $. There is a unique vector in $m_{1}$
of the form $V+z_{0}Z.$ Its $\left\| \cdot \right\| _{0}$-norm is $1$ and $%
d(e,(1;(0,0,z_{0})))=1.$ However $%
d(e,(v;(x,y,z)))=|v|+d_{H_{3}}(e,(x,y,z-vz_{0}))$. Since $\rho _{\Omega
}(e,(n;(0,0,n)))=n,$ we get
\begin{equation*}
d(e,(n;(0,0,n)))-\rho _{\Omega }(e,(n;(0,0,n)))=d_{H_{3}}(e,(0,0,n(1-z_{0})))
\end{equation*}
If this is bounded, this forces $z_{0}=1.$ But we can repeat the same
argument with $(n;(0,0,-n))$ which would force $z_{0}=-1.$ A contradiction.\\

$(C)$ Let now $\Omega_2:=\{(1;(0,0,0))^{\pm 1},(0;(1,0,0))^{\pm
1},(0;(0,1,0))^{\pm 1}\}$ and $\rho_{\Omega_2}$ the associated word metric on $\Gamma$. Then again $\rho_\Omega$ and $\rho_{\Omega_2}$ are asymptotic by Theorem \ref{Pansu} because the convex hull of their projection modulo the $z$-coordinate coincide. However $\rho_{\Omega_2}$ is a product metric, namely we have $\rho_{\Omega_2}(e,(v;(x,y,z)))=|v|+\rho(e,(x,y,z))$, where $\rho$ is the word metric on the discrete Heisenberg group $H_3(\Z)$ with standard generators $\{(1,0,0)^{\pm
1},(0,1,0)^{\pm 1}\}$. In particular
\begin{equation*}
\rho_{\Omega}(e,(n;(0,0,n)))-\rho _{\Omega_2}(e,(n;(0,0,n)))=\rho(e,(0,0,n))
\end{equation*}
which is unbounded.\\

\begin{remark}[An abnormal geodesic] We refer the reader to \cite{breuillard-ledonne} for more on these examples. In particular we show there that $\rho_1$ and $\rho_2$ above are not $(1,C)$-quasi-isometric for any $C>0$. The key phenomenon behind this example is the presence of an \emph{abnormal geodesic} (see \cite{Monty}), namely the one-parameter group $\{(t;(0,0,0))\}_t$.
\end{remark}

\begin{remark}[Speed of convergence in the nilpotent case]
The slow speed phenomenon in Theorem \ref{SmallSpeed} relied crucially on the presence
of a non-trivial semisimple part in $G_{\alpha }$ ; this doesn't occur in nilpotent
groups. In \cite{breuillard-ledonne}, we show that for word metrics on finitely generated nilpotent groups, the convergence in Theorem \ref{Pansu} has a polynomial speed with an error term at least as good as $O(d_\infty(e,x)^{-\frac{2}{3r}})$, where $r$ is the nilpotency class. We conjecture there that the optimal exponent is $\frac{1}{2}$. This involves refining quantitatively the estimates of the above proof of Theorem \ref{Pansu}.
\end{remark}

\section{Appendix: the Heisenberg groups}

Here we show how to compute the asymptotic shape of balls in the
Heisenberg groups $H_{3}(\Bbb{Z})$ and $H_{5}(\Bbb{Z})$ and their volume,
thus giving another approach to the main result of Stoll \cite{Sto2}.  The
leading term for the growth of $H_{3}(\Bbb{Z})$ is rational for all
generating sets (Prop. \ref{rati} below), whereas in $H_{5}(\Bbb{Z})$ with
its standard generating set, it is transcendental.  This explains how our
Figure 1 was made (compare with the odd \cite{Kar} Fig. 1).

\subsection{3-dim Heisenberg group}

Let us first consider the Heisenberg group
\begin{equation*}
H_{3}(\Bbb{Z})=\left\langle a,b|[a,[a,b]]=[b,[a,b]]=1\right\rangle .
\end{equation*}
We see it as the lattice generated by $a=\exp (X)$ and $b=\exp (Y)$ in the
real Heisenberg group $H_{3}(\Bbb{R})$ with Lie algebra $\frak{h}_{3}$
generated by $X,Y$ and spanned by $X,Y,Z=[X,Y].$ Let $\rho _{\Omega }$ be
the standard word metric on $H_{3}(\Bbb{Z})$ associated to the generating
set $\Omega =\{a^{\pm 1},b^{\pm 1}\}.$ According to Theorem \ref{MetComp},
the limit shape of the $n$-ball $\Omega ^{n}$ in $H_{3}(\Bbb{Z})$ coincides
with the unit ball $\mathcal{C}_{3}=\{g\in H_{3}(\Bbb{R}),d_{\infty
}(e,g)\leq 1\}$ for the Carnot-Caratheodory metric $d_{\infty }$ induced on $%
H_{3}(\Bbb{R})$ by the $\ell ^{1}$-norm $\left\| xX+yY\right\| _{0}=|x|+|y|$
on $m_{1}=span\{X,Y\}\subset \frak{h}_{3}.$

Computing this unit ball is a rather simple task. Exchanging the roles of $X$
and $Y$, we see that $\mathcal{C}_{3}$ is invariant under the reflection $%
z\mapsto -z.$ Then clearly $\mathcal{C}_{3}$ is of the form $\{xX+yY+zZ,$
with $|x|+|y|\leq 1$ and $|z|\leq z(x,y)\}.$ Changing $X$ to $-X$ and $Y$ to
$-Y,$ we get the symmetries $z(x,y)=z(-x,y)=z(x,-y)=z(y,x).$ Hence when
determining $z(x,y)$, we may assume $0\leq y\leq x\leq 1,$ $x+y\leq 1.$

The following well known observation is crucial for computing $z(x,y)$. If $%
\xi (t)$ is a horizontal path in $H_{3}(\Bbb{R})$ starting from $id,$ then $%
\xi (t)=\exp (x(t)X+y(t)Y+z(t)Z)$, where $\xi ^{\prime }(t)=x(t)X+y(t)Y$ and
$z(t)$ is the ``balayage'' area of the between the path $\{x(s)X+y(s)Y\}_{0%
\leq s\leq t}$ and the chord joining $0$ to $x(t)X+y(t)Y.$

Therefore, $z(x,y)$ is given by the solution to the ``Dido isoperimetric
problem'' (see \cite{Monty}): find a path in the $X,Y$-plane between $0$ and
$xX+yY$ of $\left\| \cdot \right\| _{0}$-length $1$ that maximizes the
``balayage area''. Since $\left\| \cdot \right\| _{0}$ is the $\ell ^{1}$%
-norm in the $X,Y$-plane, as is well-known (see \cite{Bus}), such extremal curves are given
by arcs of square with sides parallel to the $X,Y$-axes. There is therefore
a dichotomy: the arc of square has either $3$ or $4$ sides (it may have $1$
or $2$ sides, but these are included are limiting cases of the previous
ones).

If there are $3$ sides, they have length $\ell ,$ $x$ and $y+\ell $ with $%
y+\ell \leq x.$ Hence $1=\ell +x+y+\ell $ and $z(x,y)=\ell x+\frac{1}{2}xy.$
Therefore this occurs when $y\leq 3x-1$ and we then have $z(x,y)=\frac{x(1-x)%
}{2}.$

If there are $4$ sides, they have length $\ell ,x+u,y+\ell $ and $u$, with $\ell +y=x+u.$ Hence $1=2\ell +2u+x+y$ and $z(x,y)=(\ell +y)(x+u)-\frac{xy}{2}.$ This occurs when $y\geq 3x-1$ and we then have $z(x,y)=\frac{(1+x+y)^{2}}{16}-\frac{xy}{2}.$

Hence if $0\leq y\leq x\leq 1$ and $x+y\leq 1$
\begin{equation}
z(x,y)=1_{y\leq 3x-1}\frac{x(1-x)}{2}+1_{y>3x-1}\frac{(1+x+y)^{2}}{16}-\frac{%
xy}{2}  \label{zizi}
\end{equation}
The unit ball $\mathcal{C}_{3}$ drawn in Figure 1 is the solid body
$\mathcal{C}_{3}=\{xX+yY+zZ,$ with $|x|+|y|\leq 1$ and $|z|\leq z(x,y)\}.$

A simple calculation shows that $vol(\mathcal{C}_{3})=\frac{31}{72}$ in the
Lebesgue measure $dxdydz.$ Since $H_{3}(\Bbb{Z})$ is easily seen to have
co-volume $1$ for this Haar measure on $H_{3}(\Bbb{R})$ (actually $%
\{xX+yY+zZ,x\in [0,1),y\in [0,1),z\in [0,1)\}$ is a fundamental domain), it
follows that
\begin{equation*}
\lim_{n\rightarrow \infty }\frac{\#(\Omega ^{n})}{n^{4}}=vol(\mathcal{C}%
_{3})=\frac{31}{72}
\end{equation*}
We thus recover a well-known result (see \cite{Ben}, \cite{Sha} where even
the full growth series is computed and shown to be rational).

One can also determine exactly which points of the sphere $\partial \mathcal{%
C}_{3}$ are joined to $id$ by a unique geodesic horizontal path. The reader
will easily check that uniqueness fails exactly at the points $(x,y,\pm
z(x,y))$ with $|x|<\frac{1}{3}$ and $y=0,$ or $|y|<\frac{1}{3}$ and $x=0,$
or else at the points $(x,y,z)$ with $|x|+|y|=1$ and $|z|<z(x,y).$

The above method also yields the following result.

\begin{prop}\label{rati}
Let $\Omega $ be any symmetric generating set for $H_{3}(\Bbb{Z}).$ Then
the
leading coefficient in $\#(\Omega ^{n})$ is rational, i.e.
\begin{equation*}
\lim_{n\rightarrow \infty }\frac{\#(\Omega ^{n})}{n^{4}}=r
\end{equation*}
is a rational number.
\end{prop}

\proof
We only sketch the proof here. We can apply the method above and compute
$r$ as the volume of the unit $CC$-ball $\mathcal{C}(\Omega )$ of the
limit
$CC$-metric $d_{\infty }$ defined in Theorem \ref{MetComp}. Since we know what
is the norm $\left\| \cdot \right\| $ in the $(x,y)$-plane $%
m_{1}=span\left\langle X,Y\right\rangle $ that generates $d_{\infty }$ (it
is the polygonal norm given by the convex hull of the points of $\Omega
$),
we can compute $\mathcal{C}(\Omega )$ explicitly. We need to know the
solution to Dido's isoperimetric problem for $\left\| \cdot \right\| $ in
$m_{1}$, and as is well known (see \cite{Bus}) it is given by polygonal lines from the dual
polygon rotated by $90^{\circ }$. Since the polygon defining $\left\|
\cdot \right\| $ is made of rational lines (points in $\Omega $ have
integer coordinates), any vector with rational coordinates has rational
$\left\|\cdot \right\| $-length, and the dual polygon is also rational. The
equations defining $z(x,y)$ will therefore have only rational
coefficients,
and $z(x,y)$ will be piecewisely given by a rational quadratic form in $x$
and $y,$ where the pieces are rational triangles in the $(x,y)$-plane. The
total volume of $\mathcal{C}(\Omega )$ will therefore be rational.
\edpf

\subsection{5-dim Heisenberg group}

The Heisenberg group $H_{5}(\Bbb{Z})$ is the group generated by $%
a_{1},b_{1},a_{2},b_{2}$,$c$ with relations $c=[a_{1},b_{1}]=[a_{2},b_{2}]$,
$a_{1}$ and $b_{1}$ commute with $a_{2}$ and $b_{2}$ and $c$ is central. Let
$\Omega =\{a_{i}^{\pm 1},b_{i}^{\pm 1},i=1,2\}.$ Let us describe the limit
shape of $\Omega ^{n}$. Again, we see $H_{5}(\Bbb{Z})$ as a lattice of
co-volume $1$ in the group $H_{5}(\Bbb{R})$ with Lie algebra $\frak{h}_{5}$
spanned by $X_{1},Y_{1}X_{2},Y_{2}$ and $Z=[X_{i},Y_{i}].$ By Theorem \ref
{MetComp}, the limit shape is the unit ball $\mathcal{C}_{5}$ for the
Carnot-Caratheodory metric on $H_{5}(\Bbb{R})$ induced by the $\ell ^{1}$%
-norm $\left\| x_{1}X_{1}+y_{1}Y_{1}+x_{2}X_{2}+y_{2}Y_{2}\right\|
_{0}=|x_{1}|+|y_{1}|+|x_{2}|+|y_{2}|.$

Since $X_{1},Y_{1}$ commute with $X_{2},Y_{2},$ in any piecewise linear
horizontal path in $H_{5}(\Bbb{R}),$ we can swap the pieces tangent to $%
X_{1} $ or $Y_{1}$ with those tangent to $X_{2}$ or $Y_{2}$ without changing
the end point of the path. Therefore if $\xi (t)=\exp
(x_{1}(t)X_{1}+y_{1}(t)Y_{1}+x_{2}(t)X_{2}+y_{2}(t)Y_{2}+z(t)Z)$ is a
horizontal path, then $z(t)=z_{1}(t)+z_{2}(t)$, where $z_{i}(t)$, $i=1,2,$
is the ``balayage area'' of the plane curve $\{x_{i}(s)X_{i}+y_{i}(s)Y_{i}%
\}_{0\leq s\leq t}.$

Since, just like for $H_{3}(\Bbb{Z}),$ we know the curve maximizing this
area, we can compute the unit ball $\mathcal{C}_{5}$ explicitly. In
exponential coordinates it will take the form $\mathcal{C}_{5}=\{\exp
(x_{1}X_{1}+y_{1}Y_{1}+x_{2}X_{2}+y_{2}Y_{2}+zZ),$ $%
|x_{1}|+|y_{1}|+|x_{2}|+|y_{2}|\leq 1$ and $|z|\leq
z(x_{1},y_{1},x_{2},y_{2})\}.$ Then $z(x_{1},y_{1},x_{2},y_{2})=\sup_{0\leq
t\leq 1}\{z_{t}(x_{1},y_{1})+z_{1-t}(x_{2},y_{2})\},$ where $z_{t}(x,y)$ is
the maximum ``balayage area'' of a path of length $t$ between $0$ and $xX+yY.
$ It is easy to see that $z_{t}(x,y)=t^{2}z(x/t,y/t)$ where $z$ is given by $%
(\ref{zizi})$. Hence $z_{t}$ is a piecewise quadratic function of $t.$ Again
$z(x_{1},y_{1},x_{2},y_{2})$ is invariant under changing the signs of the $%
x_{i}$,$y_{i}$'s, and swapping $x$ and $y,$ or else swapping $1$ and $2.$ We
may thus assume that the $x_{i}$,$y_{i}$'s lie in $D=\{0\leq y_{i}\leq
x_{i}\leq 1$ and $x_{1}+y_{1}+x_{2}+y_{2}\leq 1$, and $x_{2}-y_{2}\geq
x_{1}-y_{1}\}.$ We may therefore determine explicitly the supremum $%
z(x_{1},y_{1},x_{2},y_{2})$, which after some straightforward calculations
takes on $D$ the following form:
\begin{equation*}
z(x_{1},y_{1},x_{2},y_{2})=1_{A}\max \{d_{1},d_{2}\}+1_{B}\max
\{d_{1},c_{1}\}+1_{C}\max \{c_{1},c_{2}\}
\end{equation*}
where $d_{1}=\frac{x_{1}y_{1}}{2}+\frac{x_{2}}{2}(1-x_{1}-y_{1}-x_{2})$, $%
c_{1}=\frac{1}{16}(1+x_{1}+y_{1}-x_{2}-y_{2})^{2}+\frac{x_{2}y_{2}-x_{1}y_{1}%
}{2},$ and $d_{2}$ and $c_{2}$ are obtained from $d_{1}$ and $c_{1}$ by
swapping the indices $1$ and $2.$ The sets $A,B$ and $C$ form the following
partition of $D:$ $A=D\cap \{m\leq x_{1}-y_{1}\},$ $B=D\cap
\{x_{1}-y_{1}<m<x_{2}-y_{2}\}$ and $C=D\cap \{x_{2}-y_{2}\leq m\},$ where $%
m=(1-x_{1}-x_{2}-y_{1}-y_{2})/2.$

Since $\mathcal{C}_{5}$ has such an explicit form, it is possible to compute
its volume. The fact that $z(x_{1},y_{1},x_{2},y_{2})$ is piecewisely given
by the maximum of two quadratic forms makes the computation of the integral
somewhat cumbersome but tractable. Our equations coincide (fortunately!)
with those of Stoll (appendix of \cite{Sto2}), where he computed the main
term of the asymptotics of $\#(\Omega ^{n})$ by a different method. Stoll
did calculate that integral and obtained
\begin{equation*}
\lim_{n\rightarrow \infty }\frac{\#(\Omega ^{n})}{n^{6}}=vol(\mathcal{C}%
_{5})=\frac{2009}{21870}+\frac{\log (2)}{32805}
\end{equation*}
which is transcendental. It is also easy to see by this method that
 if we change the generating set to $\Omega_0=\{a_{1}^{\pm 1}b_{1}^{\pm 1}a_2^{\pm 1}b_{2}^{\pm 1}\}$,
then we get a rational volume. Hence the rationality of the growth series of $H_5(\Bbb Z)$ depends on
the choice of generating set, which is Stoll's theorem.

One advantage of our method is that it can also apply to fancier generating sets. The case of Heisenberg
groups of higher dimension with the standard generating set is analogous:  the function
$z(\{x_{i}\},\{y_{i}\})$ is again piecewisely defined as the maximum of finitely many explicit quadratic
forms on a linear partition of the $\ell ^{1}$-unit ball $\sum |x_{i}|+|y_{i}|\leq 1.$\\

\noindent {\bf Acknowledgments.}
I\ would like to thank Amos Nevo for his hospitality at the Technion of
Haifa in December 2005, where part of this work was conducted, and for triggering my interest
in this problem by showing me the possible implications of Theorem \ref
{firsthm} to Ergodic Theory.
My thanks are also due to V. Losert for pointing out an inaccuracy in my first proof of Theorem \ref{weaklycom} and for his other remarks on the manuscript.
Finally I thank Y. de Cornulier, M. Duchin, E. Le Donne, Y. Guivarc'h, A. Mohammadi, P. Pansu and R. Tessera for several useful conversations.

\end{document}